\documentclass[12pt]{article}
\usepackage{eurosym}
\usepackage{amssymb}
\usepackage{amsfonts}
\usepackage{amsmath}
\usepackage{graphicx,epsf}
\usepackage{enumerate}
\usepackage{url}
\usepackage{bbm}
\usepackage{xcolor}
\usepackage{comment}
\usepackage[utf8]{inputenc}
\usepackage[T1]{fontenc}

\usepackage{arydshln}

\RequirePackage[colorlinks,citecolor=blue,urlcolor=blue]{hyperref}
\newcommand{\bomega}{\boldsymbol{\omega}}

\newcommand{\bgamma}{\boldsymbol{\gamma}}

\newcommand{\bepsilon}{\boldsymbol{\epsilon}}

\newcommand{\blambda}{\boldsymbol{\lambda}}

\newcommand{\btheta}{\boldsymbol{\theta}}

\newcommand{\bTheta}{\boldsymbol{\Theta}}
\newcommand{\bphi}{\boldsymbol{\phi}}
\newcommand{\bPhi}{\boldsymbol{\Phi}}
\newcommand{\bPi}{\boldsymbol{\Pi}}
\newcommand{\bpsi}{\boldsymbol{\psi}}

\newcommand{\bPsi}{\boldsymbol{\Psi}}

\newcommand{\brho}{\boldsymbol{\rho}}
\newcommand{\bzero}{\boldsymbol{0}}

\newcommand{\bmu}{\boldsymbol{\mu}}
\newcommand{\bxi}{\boldsymbol{\xi}}
\newcommand{\bDelta}{\boldsymbol{\Delta}}
\newcommand{\bSigma}{\boldsymbol{\Sigma}}

\newcommand{\bd}{\boldsymbol{d}}

\newcommand{\bE}{\boldsymbol{E}}

\newcommand{\bA}{\boldsymbol{A}}
\newcommand{\bB}{\boldsymbol{B}}
\newcommand{\bC}{\boldsymbol{C}}

\newcommand{\bD}{\boldsymbol{D}}

\newcommand{\bI}{\boldsymbol{I}}
\newcommand{\bJ}{\boldsymbol{J}}
\newcommand{\bP}{\boldsymbol{P}}

\newcommand{\bs}{\boldsymbol{s}}
\newcommand{\bS}{\boldsymbol{S}}

\newcommand{\bV}{\boldsymbol{V}}

\newcommand{\bv}{\boldsymbol{v}}

\numberwithin{equation}{section}
\usepackage{amsthm}
\theoremstyle{definition}

\newtheorem{thm}{Theorem}[section]
\newtheorem{example}{Example}
\newtheorem{ass}{Assumption}
\newtheorem{rem}{Remark}

\theoremstyle{plain} 
\newtheorem{prop}{Proposition}[section]
\newtheorem{lem}{Lemma}[section]

\usepackage{bm}

\providecommand{\skakko}[1]{\left(#1\right)}
\providecommand{\mkakko}[1]{\left\{#1\right\}}

\usepackage{natbib}
\allowdisplaybreaks

\begin{document}
%\vspace*{-1cm}
\title{Mixed difference integer-valued GARCH model \\for $ \mathbb{Z}$-valued time series}
\author{Abdelhakim~Aknouche\thanks{%
			Qassim University, e-mail: aknouche\_ab@yahoo.com}
            , \ Christian~Francq\thanks{%
			CREST and University of Lille, e-mail: christian.francq@ensae.fr} \ and \ Yuichi~Goto\thanks{%
			  Kyushu University, e-mail: yuichi.goto@math.kyushu-u.ac.jp}
              \thanks{The three authors contributed equally to the paper and are listed alphabetically}}
\date{}
              
\maketitle
\vspace*{-1cm}
\begin{abstract}
In this paper, we introduce flexible observation-driven $\mathbb{Z}$-valued time series models constructed from mixtures of negative and non-negative components. Compared to models based on the standard Skellam distribution or on a difference of two integer-valued variables, our specification offers greater versatility. For example, it easily allows for skewness and bimodality. Furthermore, the observation of one component of the mixture makes interpretation and statistical analysis easier.
We establish conditions for stationarity and mixing, and develop a mixed Poisson quasi-maximum likelihood estimator with proven asymptotic properties. A portmanteau test is proposed to diagnose residual serial dependence. The finite-sample performance of the methodology is assessed via simulation, and an empirical application on tick prices demonstrates its practical usefulness.
\end{abstract}

\nopagebreak[4]\noindent {\it Keywords}:   Discrete difference distribution; GARCH for tick-by-tick data, Mixed difference;  Mixed Poisson QMLE;   Random-weighting bootstrap; $\mathbb{Z}$-valued time series.\\
\section{Introduction}

Signed integer-valued ($\mathbb{Z}$-valued) time series are common in applications such as finance (price changes), macroeconomics (interest rates in discrete steps), credit rating (agencies rate entities on a discrete scale), ecology (temperatures),  sports (score differences) and sometimes result from differenced count series.

Existing modeling approaches for $\mathbb{Z}$-valued time series fall into three categories.
%	, each with drawbacks.
The first category relies on stochastic models involving signed thinning operators, based on modifications of the INAR model \citep{kp08, kt11, ao14} or integer-valued random coefficient autoregressions \citep{ags23}.
However, these models lead to complex likelihoods involving intractable convolutions, making estimation more difficult.
The second approach is based on parameter-driven (or state-space)
formulations \citep{rs03,lnp06, s11, sm14, kll17,bk18}. Although these formulations are flexible, latent dynamics render the likelihood intricate,  which makes estimation computationally demanding and complicates prediction.
The third approach is based on observation-driven GARCH-like representations,  
specifically, by multiplying non-negative INGARCH models by a sign selector, which is an independent and identically distributed (i.i.d.)\ Bernoulli or i.i.d.\ three-point random variable that takes values in $\{-1,0,1\}$ (see \cite{ha21, xz22, lcz24}). The laws of the innovations of these time series models are difference distributions, such as the standard Skellam distribution  \citep{i37,s46} or the difference of two more general independent integer-valued variables. 
Unlike parameter-driven models, this approach offers mathematical tractability. In particular, the conditional likelihood is available, and prediction is straightforward. Therefore, this paper focuses on this approach, but we adopt a flexible semi-parametric approach and we relax the independence assumption on the two components of the difference distribution (goals in a football match are not independent) as well as the restrictive assumption of an independent sign selector.

In this paper, we propose a highly flexible framework for $\mathbb{Z}$-valued time series: the dynamic Mixed Difference INGARCH (MD-INGARCH) model. Our approach generates a series as a mixed difference between two non-negative components, with a Bernoulli selector determining which component is active at each time. The core innovation of our model is that the Bernoulli selector sequence itself can be non-i.i.d.\  (e.g., following its own INGARCH process). This crucial feature enables the conditional transition probability of the series' sign to evolve over time, a level of realism unattainable by existing models based on an i.i.d.\ sign selector. Furthermore, unlike usual existing models that imply symmetry or unimodality, our framework naturally accommodates asymmetric and multimodal conditional distributions. Finally, we adopt a semi-parametric approach. This means that we do not assume that the conditional distributions of the two INGARCHs in the mixed difference are Poisson-distributed or follow any other specific distribution. Thanks to estimation by a quasi-maximum likelihood approach in particular, our results remain valid for a wide range of unspecified conditional distributions.

In financial econometrics, modeling high-frequency price changes is complicated by their discrete integer-valued nature and the well-documented asymmetry in volatility dynamics--often termed the ``leverage effect''. Existing models either ignore the discreteness (continuous-valued GARCH models) or impose restrictive independence assumptions on price movement directions and do not directly model the volatility\footnote{Standard count time series models are called ``INGARCH'' but this terminology is misleading because, unlike GARCH models, these models are for the conditional mean rather than the conditional variance (which is determined by the conditional distribution and the conditional mean)}. We address this issue by introducing a model that captures both magnitude dynamics and the state-dependent probability of price increases (i.e., tick-by-tick price variations), providing a unified framework for analyzing market microstructure and volatility clustering.

The paper is organized as follows.
	Section~\ref{sec:2} formally introduces the MD-INGARCH model, highlighting its advantages over existing approaches.
	Section~\ref{sec:3} establishes conditions for the existence of stationary and ergodic solutions, and for the beta-mixing property.
	Section~\ref{sec:4} proposes a mixed Poisson quasi-maximum likelihood estimator and shows its strong consistency and asymptotic normality.
	A portmanteau test  with bootstrap improvements  for model diagnostics is also developed.
Finite-sample performance is investigated via simulation in Section~\ref{sec:6}.
The model's practical utility is demonstrated through a real-data application in Section~\ref{sec:7}. Section~\ref{Conclusion} concludes by focusing on applications in financial econometrics.
	All proofs are collected  in 
    %Section~\ref{sec:8} of the 
    a supplementary file, and any reference to this file begins with the letter ‘S’.

\section{A mixed difference INGARCH model}\label{sec:2}
In response to the increased demand for modeling dependent discrete processes, the study of non-negative integer-valued time series has become widespread. Popular models include integer-valued autoregressive (INAR) models  (see for example \citealp{m85,aa87}) and integer-valued generalised autoregressive conditional heteroscedasticity (INGARCH) models (see for example \citealp{flo06}). More recently, hierarchical models have been proposed (\citealp{cm21}), with further developments detailed in \cite{ak23}, \cite{af23}, 
\cite{ag24},
\cite{gf25}, among others.
In this section, we introduce a new class of $\mathbb{Z}$-valued time series models based on the concept of mixed difference.

\subsection{Model}
Let $F_{\lambda}^1$ and $F_{\lambda^*}^2$ denote distribution functions supported on $\mathbb{N}_{0}=\left\{ 0,1,...\right\} $ and $\mathbb{N} =\left\{
1,2,...\right\}$, respectively, with means $\lambda>0$ and $\lambda^*>1$, respectively. Let $\left\{ B_{t},t\in \mathbb{Z} \right\} $
be a $\{0,1\}$-valued sequence. We will define a $\mathbb{Z} $-valued
process $\left\{ Y_{t},t\in \mathbb{Z}\right\} $ whose sign is positive if $%
\{B_t=1\}$ and negative if $\{B_t=0\}$. Let $\mathcal{F}_{t}$ denote the
sigma field generated by $\left\{Y_{u},u\leq t\right\}$. Let $%
\pi_t=P(B_t=1\mid \mathcal{F}_{t-1})$.
The process $\left\{ Y_{t},t\in \mathbb{Z}\right\} $ is said to follow a 
\textit{dynamic mixed difference} INGARCH (MD-INGARCH) model if for each $%
t\in \mathbb{Z}$%
\begin{equation}
Y_{t}=B_{t}X_{1t}-\left( 1-B_{t}\right) X_{2t},\quad X_{st}\mid 
\mathcal{F}_{t-1}\sim F_{\lambda _{st}}^{s},\quad s=1,2,  \label{3.2a}
\end{equation}%
where%
\begin{equation}
\lambda _{st}=\omega _{s}+\sum_{i=1}^{q}\alpha _{si}\left\vert
Y_{t-i}\right\vert +\sum_{j=1}^{p}\beta _{sj}\lambda _{s,t-j},\quad s=1,2,
\label{3.2b}
\end{equation}%
and $\omega _{1}>0,\ \alpha _{si}\geq 0,$ $\beta _{sj}\geq 0
$. 
The condition $\omega_2>1$ suffices to ensure $\lambda _{2t}>1$ a.s.\ for any $t$. 
We however impose the weaker condition $
0<1-\sum_{j=1}^{p}\beta_{2j}<\omega_{2}$ since
\begin{align*}
\lambda _{2t}
=&\ \omega_{2}
+\sum_{i=1}^{q}\alpha _{2i}\left\vert
Y_{t-i}\right\vert +\sum_{j=1}^{p}\beta _{2j}\lambda _{2,t-j}
\geq
\omega_{2}
+\sum_{j=1}^{p}\beta _{2j}\lambda _{2,t-j},
\end{align*}
and then, denoting by $B$ the backshift operator
\begin{align*}
(1-\sum_{j=1}^{p}\beta_{2j}B^j)
\lambda _{2t}
\geq
\omega_{2}
\quad
\Rightarrow
\quad
\lambda _{2t}
\geq
\frac{\omega_{2}}{1-\sum_{j=1}^{p}\beta_{2j}},
\end{align*}
which entails
$\lambda _{2t}>1$ a.s.
In the case $p=q=1$, we will use the simplified notation $\alpha_s=\alpha_{s1}$ and  $\beta_s=\beta_{s1}$.
It is assumed that $B_{t}$ and $X_{st}$ ($s=1,2$) are conditionally
independent given $\mathcal{F}_{t-1}$.
The conditional mean and variance are then
\begin{align*}
E\left( Y_{t}|\mathcal{F}_{t-1}\right)
=&\pi _{t}\lambda _{1t}-\left(
1-\pi _{t}\right) \lambda _{2t},\\
{\rm Var}\left( Y_{t}|\mathcal{F}_{t-1}\right)
=&\pi _{t}{\rm Var}\left( X_{1t}|\mathcal{F}_{t-1}\right)
+\left(
1-\pi _{t}\right){\rm Var}\left( X_{2t}|\mathcal{F}_{t-1}\right)
+\pi_{t}\left( 1-\pi
_{t}\right)\left(\lambda _{1t}+ \lambda
_{2t}\right) ^{2}.
\end{align*}
Let us give some examples of $F_{\lambda}^1$, $F_{\lambda}^2$, and $\{B_t\}$.
\begin{example}[Poisson MD-INGARCH model]\label{ex:Pois}
Prominent examples of $F_{\lambda}^1$ and $F_{\lambda}^2$ are the cdfs of the Poisson distribution $\mathcal{P}\left(
\lambda _{1t}\right) $ and the\textbf{\ }(right) shifted Poisson
distribution $\mathcal{SP}\left( \lambda _{2t}\right) $, respectively. 
The conditional distribution of $Y_{t}$ is given by%
\begin{equation}
Y_{t}|\mathcal{F}_{t-1}\sim \pi _{t}\mathcal{P}\left( \lambda _{1t}\right)
+\left( 1-\pi _{t}\right) _{-}\mathcal{SP}\left( \lambda _{2t}\right) ,
\label{$3.4$}
\end{equation}%
where $_{-}\mathcal{SP}\left( \lambda \right) $\ stands for the cdf of a
negative shifted Poisson variable $Z:=-X-1$ with mean $-\lambda $  and
probability mass function (pmf)
\begin{eqnarray*}
f_{Z}\left( z\right)  
&=&e^{-\lambda +1}\frac{(\lambda -1)^{-z-1}}{\left( -z-1\right) !},\quad
z\in \left\{ ...,-2,-1\right\} \text{,}
\end{eqnarray*}%
$X=-Z-1\sim \mathcal{P}\left( \lambda -1\right) $ being Poisson distributed
with $\lambda >1$.\qed
\end{example}

\begin{example}[mixed Poisson MD-INGARCH]\label{ex:mixedPois}
Other examples of $F_\lambda^1$ and $F_{\lambda}^2$ are, for given a mixing cumulative distribution function $G$ on $[0,\infty)$, the associated mixed Poisson distribution defined by the masses
$$p_k=\int_0^{\infty}\frac{(\lambda x)^k}{k!}e^{-\lambda x}dG(x),\quad k=0,1,\dots,$$ and its shifted distribution, respectively. 
Particularly, if we choose $G$ as a c.d.f.\ of Gamma distribution with shape $r$ and scale $p/(1-p)$, the mixed Poisson distribution is a negative binomial distribution $\mathcal{NB}(r,p)$, where
$r$ is the number of successes and $p$ is the probability of success on each trial.
Mixed Poisson distributions are often used in insurance \citep[see][]{willmot2001mixed} and have already been used for count time series models \citep[see][]{christou2015estimation, barreto2016general}.
\qed
\end{example}

\begin{example}[i.i.d.\ Bernoulli sequences]\label{ex:iid}
Obvious example of $\{B_t\}$ is i.i.d.\ Bernoulli sequences.\qed
\end{example}

\begin{example}[Bernoulli INGARCH model]\label{ex:BINGARCH}
An important example of $\{B_t\}$ is the Bernoulli INGARCH$\left(1,1\right)$ model
\begin{align}\label{demai}
P\left( B_{t}=1\mid \mathcal{F}_{t-1}^{B}\right)
=\pi _{t}, \quad
\pi _{t}
=c+aB_{t-1}+b\pi _{t-1},\quad t\in \mathbb{Z},
\end{align}
where $\mathcal{F}_{t}^{B}=\sigma \left\{ B_{t-u},u\geq 0\right\} $ denotes
the $\sigma $-algebra generated by $\left\{ B_{t-u},u\geq 0\right\} $ and $c,
$ $a$ and $b$ are constants satisfying $c>0$, $a\geq 0$, $b\geq 0$ with $a+b+c<1$ (see, e.g., \citealp{dk17}).
Note that $\mathcal{F}_{t-1}^B\subset \mathcal{F}_{t-1}$, and that $B_{t}$ and $\mathcal{F}_{t-1}$ are conditionally independent
given $\mathcal{F}_{t-1}^{B}$ in the sense that
$$P\left( B_{t}=1\mid \mathcal{F}_{t-1}\right) =P\left( B_{t}=1\mid \mathcal{F}_{t-1}^{B},\mathcal{F}_{t-1}\right) 
=P\left( B_{t}=1\mid \mathcal{F}_{t-1}^{B}\right) .$$ Since the Bernoulli
distribution belongs to the one-parameter exponential family, Theorems 3.1--3.2 in \cite{af21} ensure that $\{B_{t}\}$ is strictly
stationary, ergodic and beta-mixing. % under the condition $a+b<1$, which is already satisfied by model's definition. 
Note that the INGARCH sequence 
$\{B_{t}\}$ is just i.i.d.\ when $a=b=0$ with $\pi _{t}=c$ and reduces to a Markov chain when $b=0$ with $\pi _{t}=P\left( B_{t}=1\mid B_{t-1}\right)
=c+aB_{t-1}$. \qed
\end{example}

\begin{rem}[Link with the threshold ARCH model]\label{mele}
The MD-INGARCH model shares similarities with asymmetric GARCH models, particularly the threshold ARCH (TARCH) model proposed by \cite{zakoian1994threshold} and  expanded upon by \cite{pan2008estimation}. The TARCH model breaks down the symmetric effect of positive and negative shocks on the current volatility of the standard GARCH model introduced by \cite{bollerslev1986generalized}.
For simplicity assume that $B_t$ follows a Bernoulli ARCH(1), {\em i.e.} $b=0$ in \eqref{demai}. With an appropriate choice of the parameters, the  MD-INGARCH conditional mean 
$E(Y_t\mid {\cal F}_{t-1})=\lambda_{1t}\pi_t-\lambda_{2t}(1-\pi_t)$ can be (nearly) zero, as is the case for the TARCH.  Its volatility can then be measured by
$$E(|Y_t|\mid {\cal F}_{t-1})=c\lambda_{1t}+(1-c)\lambda_{2t}+B_{t-1} a\left(\lambda_{1t}-\lambda_{2t}\right).$$  When applying the model to price increments, it is often observed that volatility increases more after a price drop  ({\em i.e.} $B_{t-1}=0$) than after an increase ({\em i.e.} $B_{t-1}=1$) of the same magnitude. This stylized fact of financial series can be accounted for by setting $\lambda_{2t}>\lambda_{1t}$ (see the empirical application of Section~\ref{sec:7}). An important difference with TARCH models, however, is that the MD-INGARCH is applied to discretized data.

This connection to the TARCH model has important implications for financial econometrics. First, it allows us to quantify the asymmetric impact of past returns on future volatility in a setting where price changes are constrained to integer values—a common feature of tick-by-tick data that continuous-valued models ignore. Second, by modeling $E(|Y_t|\mid\mathcal{F}_{t-1})$ explicitly, we can test hypotheses about leverage effects (e.g.,  $\lambda_{2t} > \lambda_{1t}$) using standard Wald or likelihood ratio tests. Third, the Bernoulli INGARCH structure for $\pi_t$ provides a direct estimate of the time-varying probability of a price increase, which is of independent interest in market microstructure research (e.g., predicting order flow imbalance).

\end{rem}

Figure \ref{fig:tsplot} illustrates time series plots generated from the following Poisson MD-INGARCH model with a Bernoulli INGARCH structure:
\begin{align*}
Y_{t}|\mathcal{F}_{t-1}\sim \pi _{t}\mathcal{P}\left( \lambda _{1t}\right)
+\left( 1-\pi _{t}\right) _{-}\mathcal{SP}\left( \lambda _{2t}\right)
\text{ and }
B_{t}|\mathcal{F}_{t-1}^B\sim{\rm Ber}(\pi_t)
\end{align*}
where
\begin{align*}
\lambda _{st}=\omega _{s}+0.3\left\vert
Y_{t-1}\right\vert +0.3\lambda _{s,t-1}
\text{ for $s=1,2$ and }
\pi _{t} =c+aB_{t-1}+b\pi _{t-1}
\end{align*}
with $\omega_1=1$ and $\omega_2=4$.
When $a=b=0$, $B_{t}$ reduces to a Bernoulli random variable with a constant success probability $c$ (Example \ref{ex:iid}). 
In this case, the sign transitions of the time series are random, as shown in the left panels.
Conversely, the right panels depict the case of $(a,b)\neq(0,0)$ (Example \ref{ex:BINGARCH}).
In the bottom-right panel, when $B_{t-1}=1$ (resp. $0$), the conditional sign transition probability is  $0.9$ (resp. $0.1$), demonstrating a strong dependence on past values. 
The top-right panel illustrates a more complex scenario where the transition probability depends not only on $B_{t-1}$ but also on $\pi_{t-1}$.
\begin{figure}[htbp]
\begin{center}
\includegraphics[width=\linewidth]{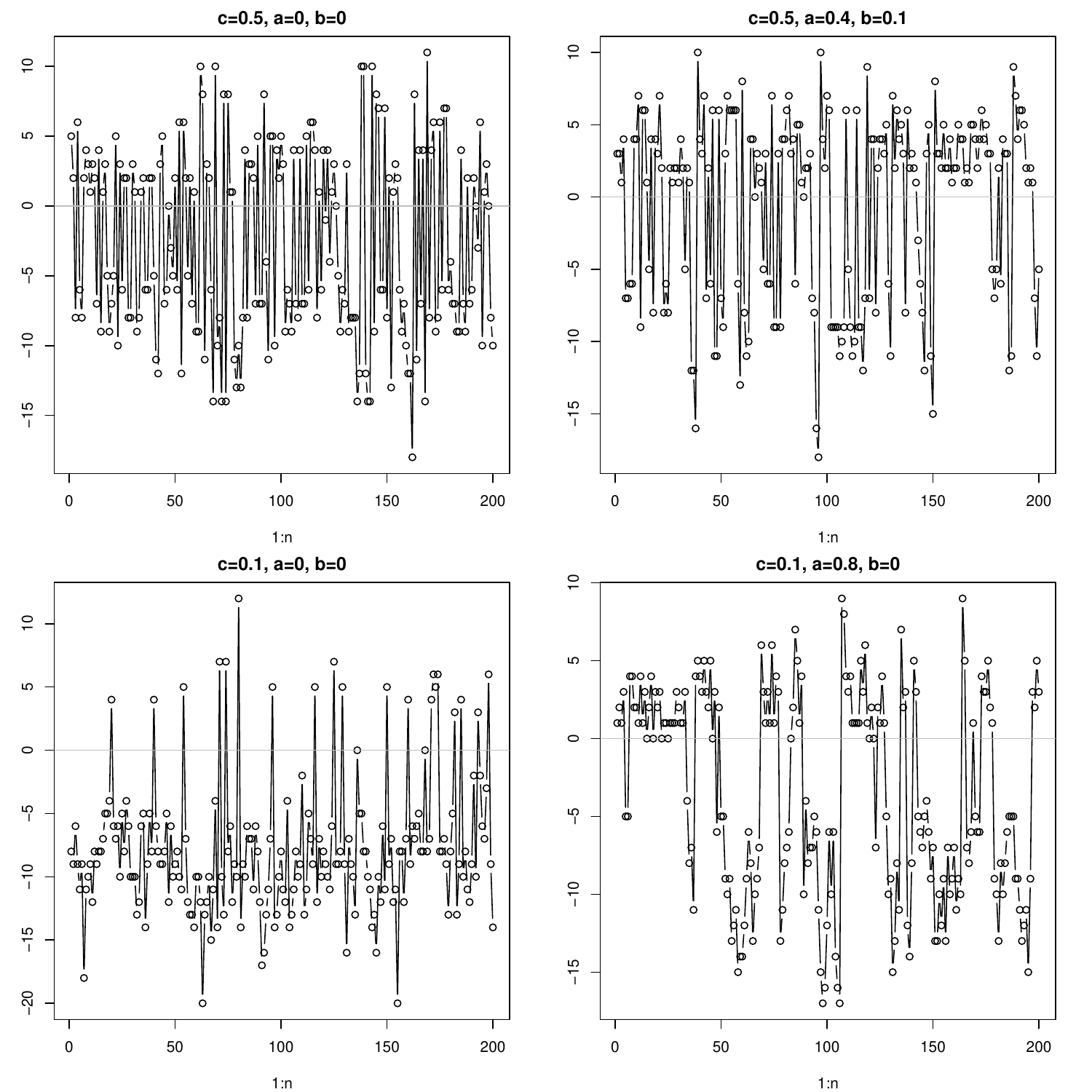}
\caption{Plots of time series generated by the MD-INGARCH model with $n=200$.
The left panels correspond to the case where $(a,b)=(0,0)$, indicating random transitions of the sign of the time series.
The right panels correspond to the case where $(a,b)\neq(0,0)$, where the conditional transition probability of the sign depends on past values, demonstrating more structured dynamics.
}\label{fig:tsplot}
\end{center}
\end{figure}

Figure \ref{fig:density} displays pmfs for several densities. 
Top-left panel corresponds to the pmf for the Skellam distribution with parameters $\lambda_1$ and $\lambda_2$, reflecting the modeling approach in \cite{aao18} and \cite{clz21}.
For $\lambda_1=\lambda_2$, the distribution is symmetric.
When $\lambda_1<\lambda_2$, the distribution is right-skewed, and when $\lambda_1>\lambda_2$, it is left-skewed. 
The Skellam family is relatively simple and lacks flexibility in representing complex shapes.
The other panels correspond to mixtures of non-negative and negative distributions, illustrating our proposed modeling approach.
Specifically, we use Poisson and negative binomial distributions for non-negative values, and negatively shifted Poisson and negative binomial distributions for negative values, combined with various mixing ratios.
Our proposed modeling framework offers extensive flexibility, enabling the representation of distributions with features such as skewness and bimodality.
\begin{figure}[htbp]
\begin{center}
\includegraphics[width=\linewidth]{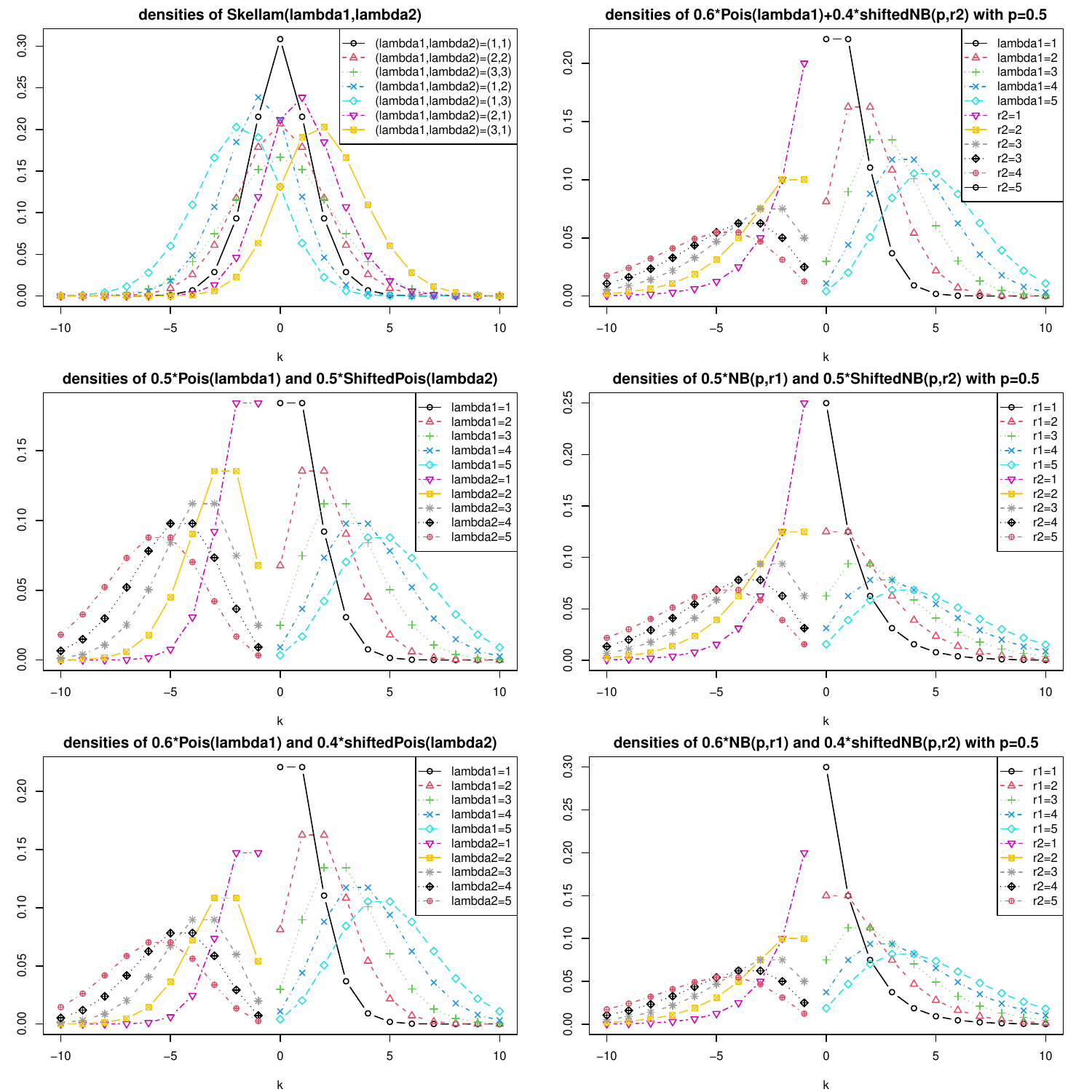}
\caption{Probability mass functions for various distributions. 
The top-left panel shows the Skellam distribution, corresponding to the approach in  \cite{aao18,clz21}, while the other panels show mixtures of positive and negative distributions with different mixing ratios, corresponding to our approach.
}\label{fig:density}
\end{center}
\end{figure}
%%%%%%%%%%%%%%%%%%%%%%%%%%%%%%%%%%%%%%%%%%%%%%%%%%%%%%%%%%%%%%%%%%%%%%%%%

\subsection{Mixed difference versus difference}
%This section explains our choice of a randomization approach over superposition.
This section explains our choice of a mixture-based approach over a direct difference approach.
Let $X_{1}$ and $X_{2}$ be two integer-valued random variables not
necessarily independent and valued in 
$\mathbb{N}_{0}$ and $\mathbb{N}$, respectively, with pmfs 
$f_{X_{i}}(\cdot)$. Let also $B$ a
Bernoulli random variable independent of $X_{1}$ and $X_{2}$ with parameter $\pi \in \left( 0,1\right) $.
Define the $\mathbb{Z%
}$-valued random variable $Y$ by
$Y := BX_{1}-\left( 1-B\right) X_{2}$.
Then, the pmf is given, for each $z\in \mathbb{Z}$, by
\begin{align}\label{eq:densityY}
f_{Y}\left( z\right)
=
{\rm P}\left(
Y=z\right) 
=
\left\{ \pi f_{X_{1}}\left( z\right) \right\}^{\mathbbm{1}_{z\geq 0
}}\left\{ \left( 1-\pi \right) f_{X_{2}}\left( -z\right) \right\}^{\mathbbm{1}_{z<0}}.
\end{align}
In other words, $Y\geq 0$ if and only if $B=1$ (i.e. $Y<0\ $ if and only if $B=0$), so $B$ is observable (and indicates the sign of $Y$).
Thus $Y$ is a mixture of $X_{1}$ and $-X_{2}$, which we call ``\textit{mixed difference}'' between $X_{1}$ and $X_{2}$. 

The mixed difference can be considered as an alternative to the standard difference $X_{1}-X_{2}$ 
which superposes $X_{1}$ and $-X_{2}$. 
Compared to the convolution $X_{1}-X_{2}$, the mixed difference $BX_{1}-\left( 1-B\right) X_{2}$ has a
much simpler distribution given by \eqref{eq:densityY} without the need to assume the
independence between $X_{1}$ and $X_{2}$. In particular, many distributions
for $X_{1}$ and $X_{2}$ are allowed, such as Poisson, negative binomial,
generalized Poisson, and even bounded-valued discrete distributions such as
Binomial and beta-Binomial.
%, and also multivariate distributions such as bivariate and multivariate Poisson distributions.
Note that the fact that $B$ is observable when $Y$ is observed greatly simplifies the estimation of the distribution of $B$ from data 
and thus the parameters of the distributions of $X_{1}$ and $X_{2}$.
From the definition of $Y$, $X_{1}$ and $X_{2}$ are not allowed to be equal because they do not assume the same support.
This assumption is made to make the supports of $X_{1}$ and $X_{2}$
forming a partition of $\mathbb{Z}$, which in turn will imply $B$
observable when $Y$ is. 
Otherwise, that is, when $X_{1}$ and $X_{2}$ are
both $\mathbb{N}_{0}$-valued then $B$ is not observable when $Y=0$ because such a value can come from both $X_{1}$ and $-X_{2}$. 
Note finally that $Y=\left(2B-1\right) X$ has a symmetric distribution on $\mathbb{Z}$ when $\pi=0.5$.
For the above reasons, we propose a mixture-based approach instead of a superposition of two opposite discrete independent random variables.

\section{Properties of the MD-INGARCH model}\label{sec:3}
This section presents the theoretical properties of the MD-INGARCH model.
Let $F_{\lambda }$ be a discrete cumulative distribution function (cdf) with
mean $\lambda =\int_{0}^{\infty }xdF_{\lambda }\left( x\right)$. 
In \cite{af21}, $F_{\lambda }$ is said to satisfy a
``stochastic-equal-mean order property'' if 
\begin{equation}  \label{3.1}
\lambda \leq \lambda ^{\ast }\quad \Rightarrow \quad F_{\lambda }^{-}(u)\leq
F_{\lambda ^{\ast }}^{-}(u),\;\forall u\in (0,1),
\end{equation}%
where $F_{\lambda }^{-}$ is the generalized inverse of $F_{\lambda }$.
Denote by $\mathbb{F}$ the class of such cdfs. 
Not only Poisson in Example \ref{ex:Pois} and mixed Poisson distributions in Example \ref{ex:mixedPois} satisfy \eqref{3.1}, but also all members of the one-parameter exponential family (e.g., Poisson, negative binomial, ...), the double Poisson, the negative binomial distribution with time-varying number of failures, the zero-inflated Poisson distribution, the zero-inflated negative binomial distribution, and mixed distributions of any distribution satisfying \eqref{3.1} enjoy the property \eqref{3.1}.
See \cite{af21} for the details.

In contrast to fully parametric $\mathbb{Z}$-valued time series models---such as the INGARCH model of \cite{ha21} and the GZG model of \cite{xz22}---our approach is semi-parametric. We model the conditional means $\lambda_{st}$ parametrically, but make only minimal assumptions ({\em i.e.} $F_{\lambda}^s \in \mathbb{F}$) about the form of the conditional distributions.

Assume there exist constants $\pi_1^+\in (0,1)$ and $\pi_0^+\in (1-\pi_1^+,1)
$ such that 
\begin{equation}  \label{fille}
\pi_t\leq \pi_1^+\qquad\mbox{ and }\qquad 1-\pi_t\leq \pi_0^+\quad\mbox{ almost surely (a.s.).}
\end{equation}
Set $r=\max \left( p,q\right) $ and consider the matrix $\boldsymbol{A}%
^{\left( l\right) } $ given by%
\begin{equation*}
\boldsymbol{A}^{\left( l\right) }=\left( 
\begin{array}{cc}
\alpha_{1l}\pi_1^++\beta_{1l} & \alpha_{1l}\pi_0^+ \\ 
\alpha_{2l}\pi_1^+ & \alpha_{2l}\pi_0^++\beta_{2l}%
\end{array}%
\right) ,\quad 1\leq l\leq r\text{.}
\end{equation*}%
Let $\boldsymbol{A} $ be defined as%
\begin{equation*}
\boldsymbol{A} =\left( 
\begin{array}{ccccc}
\boldsymbol{A}^{\left( 1\right) } & \boldsymbol{A}^{\left( 2\right) } & 
\cdots & \boldsymbol{A}^{\left( r-1\right) } & \boldsymbol{A}^{\left(
r\right) } \\ 
\bI_{2} & 0_{2\times 2} & \cdots & 0_{2\times 2} & 0_{2\times 2} \\ 
0_{2\times 2} & \bI_{2} & \cdots & 0_{2\times 2} & 0_{2\times 2} \\ 
\vdots & \vdots & \ddots & \vdots & \vdots \\ 
0_{2\times 2} & 0_{2\times 2} & \cdots & \bI_{2} & 0_{2\times 2}%
\end{array}%
\right)
\end{equation*}%
and denote by $\rho \left( \boldsymbol{A}\right) $ the spectral radius of $%
\boldsymbol{A}$, i.e., the maximum absolute eigenvalues of $\boldsymbol{A}$.

\subsection{Stationarity conditions}
The following result gives a sufficient condition for the existence of a stationary and ergodic solution to the MD-INGARCH model.

\begin{prop}\label{Pro3.1} 
Assume that
$F_{\lambda}^1$ and $F_{\lambda}^2$ satisfy \eqref{3.1} and have respective
supports $\mathbb{N}_{0}=\left\{ 0,1,...\right\} $ and $\mathbb{N} =\left\{
1,2,...\right\}$. Assume that
$\left\{ B_{t},t\in \mathbb{Z} \right\} $ is stationary and ergodic with \eqref{fille}. 
There exists a stationary and ergodic process $%
\{Y_t\}$ satisfying $\eqref{3.2a}$-$\eqref{3.2b}$ if 
\begin{equation}
\rho \left( \boldsymbol{A} \right) <1\text{.}  \label{3.7}
\end{equation}
This solution satisfies  
\begin{equation}\label{moment1}
E|Y_t|<\infty\qquad\mbox{ and }\qquad E\lambda_{st}<\infty\quad\mbox{ for }s=1,2.
\end{equation}
If there exists a stationary process $
\{Y_t\}$ satisfying $\eqref{3.2a}$-$\eqref{3.2b}$ and \eqref{moment1} then 
\begin{equation}
\sum_{j=1}^{p}\beta _{sj}<1\quad\mbox{ for }s=1,2.  \label{3.8a}
\end{equation}%
\end{prop}
Obviously, the sufficient stationarity condition \eqref{3.7} implies the necessary stationarity condition \eqref{3.8a}, but the former is generally much more restrictive. 
The following proposition shows that this sufficient condition is optimal in the case of Example \ref{ex:iid}.
\begin{prop}\label{Pro3.1CNS}

If $\{B_{t}\}$ is i.i.d.\ (Example \ref{ex:iid}), with $P(B_{t}=1)=\pi \in (0,1)$, then Proposition \ref{Pro3.1} holds with 
\begin{equation*}
\boldsymbol{A}^{\left( l\right) }=\left( 
\begin{array}{cc}
\alpha _{1l}\pi +\beta _{1l} & \alpha _{1l}(1-\pi ) \\ 
\alpha _{2l}\pi  & \alpha _{2l}(1-\pi )+\beta _{2l}%
\end{array}%
\right) ,\quad 1\leq l\leq r.
\end{equation*}
Moreover, the condition \eqref{3.7} is necessary and sufficient for the existence of a stationary solution such that $E\lambda_{st}<\infty$ for $s=1,2$.
\end{prop}

If $\{B_{t}\}$ is a Markov chain, with transition
probabilities $p(i,j)=P(B_{t}=j\mid B_{t-1}=i)\in (0,1)$ for $i,j\in \{0,1\}$%
, then Proposition \ref{Pro3.1} holds with $\pi _{1}^{+}=\max
\{p(0,1),p(1,1)\}$ and $\pi _{0}^{+}=\max \{p(1,0),p(0,0)\}$.
For the Bernoulli INGARCH model (Example \ref{ex:BINGARCH}), Proposition \ref{Pro3.1} holds with $\pi_{1}^{+}=a+b+c$ and $\pi_{0}^{+}=1-c$.

 A necessary condition, sharper than \eqref{3.8a}, for the existence of a stationary solution with a finite mean is given by the following result.

\begin{prop}\label{prop:parameter}
Consider the case where $\{B_t\}$ is an i.i.d.\ Bernoulli %${\cal B}(\pi)$ 
sequence with parameter $\pi\in(0,1)$.
If there exists a stationary process $
\{Y_t\}$ satisfying $\eqref{3.2a}$-$\eqref{3.2b}$  and \eqref{moment1} then \eqref{3.8a} and
\begin{equation}
\pi \left( 1-\sum\limits_{j=1}^{p}\beta _{1j}\right)
^{-1}\sum\limits_{i=1}^{q}\alpha _{1i}+\left( 1-\pi \right) \left(
1-\sum\limits_{j=1}^{p}\beta _{2j}\right) ^{-1}\sum\limits_{i=1}^{q}\alpha
_{2i}<1  \label{3.8b}
\end{equation}
hold.
\end{prop}

Under the conditions of the previous proposition, we have
\begin{equation*}
E|Y_{t}|=\left( 1-\pi \frac{\sum\limits_{i=1}^{q}\alpha _{1i}}{%
1-\sum\limits_{j=1}^{p}\beta _{1j}}-\left( 1-\pi \right) \frac{%
\sum\limits_{i=1}^{q}\alpha _{2i}}{1-\sum\limits_{j=1}^{p}\beta _{2j}}%
\right) ^{-1}
\skakko{\pi \frac{\omega _{1}}{1-\sum\limits_{j=1}^{p}\beta _{1j}}+(1-\pi )\frac{\omega _{2}}{1-\sum\limits_{j=1}^{p}\beta _{2j}}}.
\end{equation*}%
For example, when $p=q=1$, the means of $Y_{t}$ and $\left\vert
Y_{t}\right\vert $ are given by%
\begin{align*}
E\left( \left\vert Y_{t}\right\vert \right) & =\frac{\pi \omega
_{1}(1-\beta _{2})+\left( 1-\pi \right) \omega _{2}(1-\beta _{1})}{(1-\beta
_{1})(1-\beta _{2})-\pi \alpha _{1}(1-\beta _{2})-\left( 1-\pi \right)
\alpha _{2}(1-\beta _{1})}, \\
E\left( Y_{t}\right) & =\pi \frac{\omega _{1}+\alpha _{1}E\left( \left\vert
Y_{t}\right\vert \right) }{1-\beta _{1}}-\left( 1-\pi \right) \left( \frac{%
\omega _{2}+\alpha _{2}E\left( \left\vert Y_{t}\right\vert \right) }{1-\beta
_{2}}\right) \text{.}
\end{align*}

The next proposition shows that
the explicit necessary conditions  \eqref{3.8a} and \eqref{3.8b} can be necessary and sufficient.
\begin{prop}\label{prop:NSC}
Consider a process $\{Y_t\}$ satisfying \eqref{3.2a}–\eqref{3.2b} and \eqref{moment1}.
Suppose $p=q=1$ and $\{B_t\}$ is an i.i.d.\ Bernoulli sequence with parameter $\pi\in(0,1)$.
Then, the conditions
\eqref{3.8a} and \eqref{3.8b} are equivalent to \eqref{3.7}.
\end{prop}

\subsection{Mixing conditions}

For a real process $\{X_t, t\in \mathbb{Z}\}$ and $-\infty\leq t\leq s\leq
\infty$ we use the notation $X_{t:s}=(X_t,X_{t+1},\dots,X_s)$ and $%
X_{s:t}=(X_s,X_{s-1},\dots,X_t)$.  Let $%
\{Y_t,t\in\mathbb{Z}\}$ be the stationary process defined in the proof of Proposition \ref{Pro3.1}. Denote by $%
\mathcal{B}$ the Borel sigma-algebra of $\mathbb{R}^{\infty}$. For $h\geq
0$,  let the $\beta$-mixing coefficient 
\begin{equation*}
\beta_Y(h)=E\sup_{A\in \mathcal{B}}
\left|P\left(Y_{t+h:\infty}\in A\mid
Y_{-\infty:t}\right)-P\left(Y_{t+h:\infty}\in A\right)\right|.
\end{equation*}
We can think of $\beta_Y(h)$ as a kind of measure of the
dependence between the past and future of the process $\{Y_t, t\in \mathbb{Z}\}$, if the past and future are separated by  $h$ dates (see \cite{bradley2005basic} for a survey on mixing). The geometric decrease of the $\beta$-mixing
coefficients shown in the next proposition is a desirable probabilistic property which, for instance, entails the existence of a central limit theorem on general transformations of the process 
\citep[see {\em e.g.}][]{herrndorf1984functional}.

\begin{prop}
\label{mixingNL} 
Suppose $\{B_t\}$ follows the Bernoulli INGARCH$\left(1,1\right)$ model of Example~\ref{ex:BINGARCH} and the other assumptions of Proposition \ref{Pro3.1}, in particular \eqref{3.7}, are satisfied.  There exists a stationary and ergodic  MD-INGARCH process $%
\{Y_t, t\in\mathbb{Z}\}$ satisfying $\eqref{3.2a}$-$\eqref{3.2b}$,  and there exist constants $K>0$ and $\varrho\in(0,1)$ such that
\begin{equation*}
\beta_Y(h)\leq K\varrho^h,\qquad h\geq 0.
\end{equation*}
\end{prop}

\section{Inference}\label{sec:4}

Let $Y_{1},Y_{2},...,Y_{n}$ be a finite sample-path generated from model %
\eqref{3.2a} and $B_{1},B_{2},...,B_{n}$ be the corresponding
observation-sign indicators. Assume that the generating$\ $process $\{B_{t}\} $ has the Bernoulli INGARCH$\left( 1,1\right) $ representation
of Example \ref{ex:BINGARCH}. Our first aim is to estimate the true parameter $\btheta _{0}=\left(
\bphi_{0}^\top,\bpsi _{0}^\top\right) ^\top\in $ $\bTheta
:=\bPhi \times \bPsi $ where $\bphi _{0}=\left( a_{0},b_{0},c_{0}\right)
^\top\in \bPhi \subset \left( 0,1\right) \times \left[ 0,1\right) ^{2}$ and $%
\bpsi _{0}^\top=\left( \bpsi _{01}^\top,\bpsi _{02}^\top\right)
\in \bPsi =\bPsi _{1}\times \bPsi _{2}$ with $\bpsi _{0s}=\left(
\omega _{0s},\alpha _{0s1},...,\alpha _{0sq},\beta _{0s1},...,\beta
_{0sp}\right)^\top \in \bPsi _{s}\subset \left( 0,\infty \right) \times \lbrack
0,\infty )^{q+p}$  for $s=1,2$.
Since we do not want to make strong parametric assumptions on the conditional distributions $F_{\lambda}^1$ and $F_{\lambda}^2$, we propose a semi-parametric estimator, more precisely, the mixed Poisson maximum quasi-likelihood (MP-QML) estimator.

Writing in \eqref{3.2b} the true intensities $\lambda _{st}:=\lambda
_{st}\left( \bpsi _{0s}\right) $ in terms of the true parameter $\bpsi _{0s}$,
let $\lambda _{st}\left( \bpsi _{s}\right) $ ($s=1,2$) be the generic
intensity function defined for any generic parameter $\bpsi :=\left( \bpsi
_{1}^\top,\bpsi _{2}^\top\right) ^\top\in \bPsi $ and all $t\in 
\mathbb{Z}$ by%
\begin{equation}
\lambda _{st}\left( \bpsi _{s}\right) =\omega _{s}+\sum_{i=1}^{q}\alpha
_{si}\left\vert Y_{t-i}\right\vert +\sum_{j=1}^{p}\beta _{sj}\lambda
_{s,t-j}\left( \bpsi _{s}\right) ,\quad s=1,2\text{,}  \label{4.1}
\end{equation}%
where $\bpsi _{s}=\left( \omega _{s},\alpha _{s1},...,\alpha _{sq},\beta
_{s1},...,\beta _{sp}\right) ^\top$. The sequence $\left\{ \lambda
_{st}\left( \bpsi _{s}\right) ,t\in \mathbb{Z}\right\} $ uniformly exists
almost surely whenever
\begin{equation}
\sum\limits_{j=1}^{p}\beta _{sj}<1,\quad s=1,2  \label{4.2}
\end{equation}%
for every $\bpsi _{s}\in \bPsi _{s}$. Since the $\beta_{sj}$'s are positive, by Corollary 2.2 in \cite{fz19},  when the $\bPsi_{s}$'s are compact, \eqref{4.2} is equivalent to
\begin{equation}
\max_{\bpsi _{s}\in \bPsi _{s}}\rho(\bB_s)<1,\quad s=1,2  \label{4.2bis}
\end{equation}%
where
$$
\bB_s
:=
\begin{pmatrix}
\beta_{s1}&\beta_{s2}&\cdots&\beta_{sp}\\
1&0&\cdots&0\\
\vdots&\ddots&\ddots&\vdots\\
0&\cdots&1&0
\end{pmatrix}.
$$

As $\lambda _{st}\left( \bpsi _{s}\right) 
$ is not computable from a finite dataset, let $\widetilde{\lambda }%
_{10},...,\widetilde{\lambda }_{1,1-p}$, $\widetilde{\lambda }_{20},...,%
\widetilde{\lambda }_{2,1-p}$, $Y_{0},...,Y_{1-q}$ 
be fixed starting values
and define $\widetilde{\lambda }_{st}\left( \bpsi _{s}\right) $ to be a
data-computable counterpart of $\lambda _{st}\left( \btheta \right) $ given by%
\begin{equation*}
\widetilde{\lambda }_{st}\left( \bpsi _{s}\right) =\omega
_{s}+\sum_{i=1}^{q}\alpha _{si}\left\vert Y_{t-i}\right\vert
+\sum_{j=1}^{p}\beta _{sj}\widetilde{\lambda }_{s,t-j}\left( \bpsi
_{s}\right) ,\quad t\geq 1\text{.} 
\end{equation*}%
Similarly, for any $\bphi \in \bPhi $, let $\pi _{t}\left( \bphi \right) $ be the probability parameter function given by 
\begin{equation}
\pi _{t}\left( \bphi \right) =c+aB_{t-1}+b\pi _{t-1}\left( \bphi \right) \text{%
,}\quad t\in \mathbb{Z}  \label{4.4}
\end{equation}%
and let $\widetilde{\pi }_{t}\left( \bphi \right) $ be an observable proxy of 
$\pi _{t}\left( \bphi \right) $ given by%
\begin{equation*}
\widetilde{\pi }_{t}\left( \bphi \right) =c+aB_{t-1}+b\widetilde{\pi }%
_{t-1}\left( \bphi \right) \text{, \ }t\geq 1\text{,}  
\end{equation*}%
for some arbitrary fixed initial values for ${B}_{0}$ and $\widetilde{\pi }%
_{0}.$ 
Note that \eqref{4.4} is well-defined since $b\in \left[ 0,1\right) $%
. It will be seen that under \eqref{4.2}, the choice of initial values $%
\widetilde{\lambda }_{10},...,\widetilde{\lambda }_{1,1-p}$, $\widetilde{%
\lambda }_{20},...,\widetilde{\lambda }_{2,1-p},$ $Y_{0},...,Y_{1-q}$, $
{B}_{0},\widetilde{\pi }_{0}$ is unimportant asymptotically.

\subsection{Mixed Poisson QMLE}

Model \eqref{3.2a} entails
\begin{equation*}
f_{Y_{t}\mid \mathcal{F}_{t-1}}\left( y\right) =\left( \pi
_{t}f_{X_{1t}}\left( y\right) \right) ^{\mathbbm{1}_{y\geq 0}}\left(
\left( 1-\pi _{t}\right) f_{-X_{2t}}\left( y\right) \right) ^{\mathbbm{1}_{y<0}}
\end{equation*}%
where $f_{Y_{t}\mid \mathcal{F}_{t-1}}\left( y\right) =P\left( Y_{t}=y\mid 
\mathcal{F}_{t-1}\right) $ denotes the conditional pmf. If $X_{1t}$ were Poisson distributed with mean $\lambda _{1t}$, and $%
X_{2t}$ were  shifted Poisson distributed with mean $\lambda _{2t}$
(i.e. $X_{2t}=Z_{t}+1$ with $Z_{t}\sim \mathcal{P}\left( \lambda
_{2t}-1\right) $), then the log-conditional pmf of $Y_{t}$ would be given by%
\begin{align*}
\log (f_{Y_{t}\mid \mathcal{F}_{t-1}}\left( y\right) )=&
\left( \log
\left( \pi _{t}\right) -\lambda _{1t}+y\log (\lambda _{1t})-\log (y!)\right)
\mathbbm{1}_{ y\geq 0} \\
&+ 
\left( \log \left( 1-\pi
_{t}\right) -\lambda _{2t}+1-\left( y+1\right) \log (\lambda _{2t}-1)-\log
(\left( -y-1\right) !)\right) \mathbbm{1}_{ y<0 }.
\end{align*}

Thus, we propose to estimate the true parameter $\btheta _{0}$ using the
mixed Poisson QMLE (MP-QMLE), which is a measurable solution to the following
problem%
\begin{equation}
\widehat{\btheta }_{n}=\arg \max_{\btheta \in \bTheta }\widetilde{L}%
_{n}\left( \btheta \right) \ \text{with }\widetilde{L}_{n}\left( \btheta
\right) =\frac{1}{n}\sum_{t=1}^{n}\widetilde{\ell}_{t}\left( \btheta \right) 
\label{4.6}
\end{equation}%
where
\begin{align}\nonumber
\widetilde{\ell}_{t}\left( \btheta \right)
=&
\left( \log \left(\widetilde{\pi }
_{t}\left( \bphi \right) \right) -\widetilde{\lambda }_{1t}\left( \bpsi
_{1}\right) +Y_{t}\log (\widetilde{\lambda }_{1t}\left( \bpsi_{1}\right)
)\right) \mathbbm{1}_{ Y_{t}\geq 0 }\\
\label{4.7}
&+  
\left( \log \left( 1-\widetilde{\pi }%
_{t}\left( \bphi \right) \right) -\widetilde{\lambda }_{2t}\left( \bpsi
_{2}\right) -\left( Y_{t}+1\right) \log (\widetilde{\lambda}_{2t}\left(
\bpsi _{2}\right) -1)\right) \mathbbm{1}_{ Y_{t}<0}\text{.}
\end{align}

To study the consistency and asymptotic normality of the MP-QMLE given by 
\eqref{4.6} we consider the following assumptions.

\begin{ass}\label{ass:A}

\begin{enumerate}
\textbf{}

\textbf{A1.} Conditions of Proposition \ref{Pro3.1} (at $\btheta_0$) %\eqref{3.7} 
and \eqref{4.2} (or equivalently \eqref{4.2bis}) are
satisfied.

\textbf{A2.} We have $E\left( X_{st}^{\tau }\right)<\infty $ for some $\tau >1$, $%
s=1,2$.

\textbf{A3(i).} The polynomials $\alpha _{0s}\left( z\right)
=\sum\limits_{i=1}^{q}\alpha _{0si}z^{i}$ and $\beta _{0s}\left( z\right)
=1-\sum\limits_{j=1}^{p}\beta _{0sj}z^{j}$\ have no common root,\ $\alpha
_{0s}\left( 1\right) \neq 0$, and $\alpha _{0sq}+\beta _{0sp}\neq 0$, $s=1,2$%
.

\textbf{A3(ii).} The coefficient $a_{0}\neq 0$.

\textbf{A4.} The parameter space $\bTheta $ is compact and $\btheta _{0}\in \bTheta $.

\textbf{A5.} The true parameter $\btheta _{0}$ is in the interior of $\bTheta $.

\textbf{A6.} For $s=1,2$ and some $\varepsilon>0$ we have $E\left(\frac{{\rm Var}\left( X_{st}\mid 
\mathcal{F}_{t-1}\right)}{\lambda_{st}(\bpsi_{0s})}\right)^{1+\varepsilon}<\infty$.

\end{enumerate}

\end{ass}

These assumptions are standard and resemble those given
for similar QMLEs of integer-valued models. See e.g. \cite{af16} for the Poisson QMLE. In particular {\bf A3(i)}-{\bf A3(ii)} are identifiability conditions. If {\bf A3(ii)} would not hold we could write $\pi_t=c_0=c_0(1-b)+b\pi_{t-1}$ under many forms.  Let $\overset{D}{\underset{n\rightarrow \infty }%
{\rightarrow }}$ and $\overset{a.s.}{\underset{n\rightarrow \infty }{%
\rightarrow }}$ denote, respectively, the convergence in distribution and
almost sure convergence as $n\rightarrow \infty $. We will show the existence and invertibility of 
{\small
\begin{eqnarray*}
\bPi  &=&E\left( \frac{1}{\pi _{t}\left( \bphi _{0}\right) \left( 1-\pi
_{t}\left( \bphi _{0}\right) \right) }\frac{\partial \pi _{t}\left( \bphi
_{0}\right) }{\partial \bphi }\frac{\partial \pi _{t}\left( \bphi _{0}\right) 
}{\partial \bphi ^\top}\right)   \nonumber \\
\bJ_{1} &=&E\left( \tfrac{Y_{t}}{\lambda _{1t}^{2}\left( \bpsi _{01}\right) }%
\tfrac{\partial \lambda _{1t}\left( \bpsi _{01}\right) }{\partial \bpsi _{1}}%
\tfrac{\partial \lambda _{1t}\left( \bpsi _{01}\right) }{\partial \bpsi
_{1}^{\top }}1_{\left[ Y_{t}\geq 0\right] }\right),
\quad
\bI_{1} =E\left( \left( \tfrac{Y_{t}-\lambda _{1t}\left( \bpsi _{01}\right) }{%
\lambda _{1t}\left( \bpsi _{01}\right) }\right) ^{2}\tfrac{\partial \lambda
_{1t}\left( \bpsi _{01}\right) }{\partial \bpsi _{1}}\tfrac{\partial \lambda
_{1t}\left( \bpsi _{01}\right) }{\partial \bpsi _{1}^{\top }}1_{\left[
Y_{t}\geq 0\right] }\right)  \\
\bJ_{2} &=&E\left( \tfrac{Y_{t}+1}{\left( \lambda _{2t}\left( \bpsi
_{02}\right) -1\right) ^{2}}\tfrac{\partial \lambda _{2t}\left( \bpsi
_{02}\right) }{\partial \bpsi _{2}}\tfrac{\partial \lambda _{2t}\left( \bpsi
_{02}\right) }{\partial \bpsi _{2}^{\top }}1_{\left[ Y_{t}<0\right]
}\right),
\quad
\bI_{2} =E\left( \left( \tfrac{Y_{t}+\lambda _{2t}\left( \bpsi _{02}\right) }{%
 \lambda _{2t}\left( \bpsi _{02}\right) -1}\right) ^{2}%
\tfrac{\partial \lambda _{2t}\left( \bpsi _{02}\right) }{\partial \bpsi _{2}}%
\tfrac{\partial \lambda _{2t}\left( \bpsi _{02}\right) }{\partial \bpsi
_{2}^{\top }}1_{\left[ Y_{t}<0\right] }\right),
\end{eqnarray*}
}
where $\sigma_{st}^2={\rm Var}\left( X_{st}\mid 
\mathcal{F}_{t-1}\right)$.
Let the block-diagonal matrices 
\begin{equation*}\bJ=\mbox{diag}(\bPi ,\bJ_{1},\bJ_{2}),\qquad 
\bSigma =\mbox{diag}\left( \bPi
^{-1},\bJ_{1}^{-1}\bI_{1}\bJ_{1}^{-1},\bJ_{2}^{-1}\bI_{2}\bJ_{2}^{-1}\right). 
\end{equation*}%

\begin{thm}\label{thm:est}
i) \textit{Under} \textit{{\bf A1}-{\bf A4}},%
\begin{equation}
\widehat{\btheta }_{n}\overset{a.s.}{\underset{n\rightarrow \infty }{%
\rightarrow }}\btheta _{0}\text{.}  \label{4.9}
\end{equation}

ii) \textit{If, in addition, {\bf A5}-{\bf A6}\ hold, then }
\begin{equation}
\sqrt{n}\left( \widehat{\btheta }_{n}-\btheta _{0}\right)=\bJ^{-1}\frac{1}{\sqrt{n}}\sum_{t=1}^n\bDelta_t\bxi_t+o_P(1) \overset{D}{%
\underset{n\rightarrow \infty }{\rightarrow }}\mathcal{N}\left( 0,\bSigma
\right), \label{4.10}
\end{equation}
where 
$$\bDelta_t=\mbox{diag}\left(\frac{1}{\pi_t(1-\pi_t)}\frac{\partial \pi_t}{\partial\bphi}, \frac{1}{\lambda_{1t}}\frac{\partial \lambda_{1t}}{\partial\bpsi_1}, \frac{1}{\lambda_{2t}-1}\frac{\partial \lambda_{2t}}{\partial\bpsi_2}\right)\in {\cal F}_{t-1}$$ and 
$\bxi_t^\top=\left\{B_t-\pi_t,B_t(X_{1t}-\lambda_{1t}),(1-B_t)(X_{2t}-\lambda_{2t})\right\}$ such that $\{\bxi_t,{\cal F}_{t}\}$ is a martingale difference sequence.
\end{thm}

To apply \eqref{4.10}, an estimator of $\bSigma$ is required. The matrix $\bPi$ can be estimated empirically by
$$\widehat\bPi:=\frac{1}{n}\sum_{t=1}^n\frac{1}{\widetilde\pi _{t}\left( \widehat\bphi_{n}\right) \left( 1-\widetilde\pi
_{t}\left( \widehat\bphi_{n}\right) \right) }\frac{\partial \widetilde\pi _{t}\left( \widehat\bphi_{n}\right) }{\partial \bphi }\frac{\partial \widetilde\pi _{t}\left( \widehat\bphi _{n}\right) 
}{\partial \bphi ^\top}.$$ Similar empirical estimators $\widehat\bJ_1$,  $\widehat\bI_1$, $\widehat\bJ_2$ and  $\widehat\bI_2$  are defined. We then set $\widehat\bJ=\mbox{diag}(\widehat\bPi, \widehat\bJ_{1}, \widehat\bJ_{2})$ and $
\widehat\bSigma =\mbox{diag}\left( \widehat\bPi
^{-1},\widehat\bJ_{1}^{-1} \widehat\bI_{1}\widehat\bJ_{1}^{-1}, \widehat\bJ_{2}^{-1}\widehat\bI_{2}\widehat\bJ_{2}^{-1}\right)$ when the matrix $\widehat\bJ$ is invertible, which holds true almost surely when $n$ is large enough. 
\begin{rem}
When $X_{1t}$ is Poisson distributed and $X_{2t}$ is
shifted Poisson distributed, we have ${\rm Var}\left( X_{1t}|\mathcal{F}_{t-1}\right)
=\lambda _{1t}\left( \bpsi _{01}\right) $ and ${\rm Var}\left( X_{2t}|\mathcal{F}%
_{t-1}\right) =\lambda _{2t}\left( \bpsi _{01}\right) -1$. Hence the
equalities $\bI_{1}=\bJ_{1}$ and $\bI_{2}=\bJ_{2}$ hold so $\bSigma $ simplifies to $%
\mbox{diag}\left( \bPi ^{-1},\bJ_{1}^{-1},\bJ_{2}^{-1}\right) $, which is the inverse of the Fisher information matrix. 
Thus, in this case, the
MP-QMLE is asymptotically efficient.
\end{rem}

\begin{rem}
Note that the first-order conditions of the optimization problem \eqref{4.6}-\eqref{4.7} define $\widehat{\bphi }_{n}$, $\widehat{\bpsi }_{2n}$   and $\widehat{\bpsi }_{1n}$  independently. Therefore  $\widehat{\btheta }_{n}$  is defined by
\begin{eqnarray*}
\widehat{\bphi }_{n} &=&\arg \max_{\bphi \in \bPhi }\frac{1}{n}%
\sum_{t=1}^{n}\log \left( \widetilde{\pi }_{t}\left( \bphi \right) \right) \mathbbm{1}_{Y_{t}\geq 0}+\log \left( 1-\widetilde{\pi }_{t}\left( \bphi
\right) \right) \mathbbm{1}_{Y_{t}<0} \\
\widehat{\bpsi }_{1n} &=&\arg \max_{\bpsi _{1}\in \bPsi _{1}}\frac{1}{n}%
\sum_{t=1}^{n}\left( -\widetilde{\lambda }_{1t}\left( \bpsi _{1}\right)
+Y_{t}\log (\widetilde{\lambda }_{1t}\left( \bpsi _{1}\right) )\right) \mathbbm{1}_{Y_{t}\geq 0} \\
\widehat{\bpsi }_{2n} &=&\arg \max_{\bpsi _{2}\in \bPsi _{2}}\frac{1}{n}%
\sum_{t=1}^{n}\left( -\widetilde{\lambda }_{2t}\left( \bpsi _{2}\right)
-\left( Y_{t}+1\right) \log (\widetilde{\lambda }_{2t}\left( \bpsi
_{2}\right) -1))\right) \mathbbm{1}_{ Y_{t}<0}\text{.} 
\end{eqnarray*}%
Since the latter estimates are computationally much easier than \eqref{4.6}-\eqref{4.7}, we will 
adopt them in applications.

Note that the block-diagonal form of $\bSigma$ is consistent with the fact that $\widehat{\btheta }_{n}$  is obtained by computing three separate MP-QMLEs.

\end{rem}

\subsection{Goodness-of-fit tests}\label{sec:5}
In this section, we present a portmanteau test to evaluate the suitability of MD-INGARCH models. These tests examine whether the first residual autocorrelations (up to a certain lag) are jointly significant. Portmanteau tests are among the most popular diagnostic tools in time series analysis.   For a reference book devoted to this class of tests, see \cite{li2003diagnostic}. For more recent sources, see  \cite{boubacar2022portmanteau} and the references therein.
For a Portmanteau test for integer-valued time series, see also \cite{a16}.

The residuals of the MD-INGARCH model are  $\widehat{\epsilon}_1,\dots, \widehat{\epsilon}_n$, where $\widehat{\epsilon}_t=\widetilde\epsilon_{t}(\widehat{\btheta}_n)$ with
\begin{align*}
\widetilde\epsilon_{t}(\btheta)
:= Y_{t} - I_{\{Y_t\geq0\}} \widetilde\lambda _{1t}(\bpsi _{1}) + I_{\{Y_t<0\}} \widetilde\lambda _{2t}(\bpsi _{2}),\quad 1\leq t\leq n,
\end{align*}
and the convention $\widetilde\epsilon_{t}(\cdot)\equiv 0$ when $t<1$ or $t>n$.
Let $\epsilon_{t}(\btheta)
= Y_{t} - I_{\{Y_t\geq0\}} \lambda _{1t}(\bpsi _{1}) + I_{\{Y_t<0\}} \lambda _{2t}(\bpsi _{2})$ and $\epsilon_{t}=\epsilon_{t}(\btheta_0).$
Under the assumption of Theorem \ref{thm:est},  $\{\epsilon_{t}\}$ is uncorrelated: for $h>0$ we have
$
E\left(\epsilon_{t-h}\epsilon_{t}\right)
=
E\left\{\epsilon_{t-h}E\left(\epsilon_{t}\mid \mathcal F_{t-1}\right)\right\}=0.
$
The empirical residual autocorrelation at lag $h$ is defined by
$$
\widehat\rho_{h}=
\frac{\widehat\gamma_{h}}{\widehat\gamma_{0}},\quad \widehat\gamma_{h}
=
\frac{1}{n}\sum_{t=1}^{n}
\widehat\epsilon_{t} \widehat\epsilon_{t - h}.
$$ The idea behind the goodness-of-fit portmanteau test is that, if the model is correctly specified and the other assumptions of  Theorem \ref{thm:est} hold, then the residuals should be close to the innovations, and thus the $\widehat\rho_{h}$'s should all be close to zero.
For some fixed integer $k>0$, let the vector $\widehat\brho_{1:k}=(\widehat\rho_{1},\dots, \widehat\rho_{k})^\top$.
We are going to show that, under the assumptions of Theorem \ref{thm:est},  $$\sqrt{n}\widehat\brho_{1:k}\underset{n\to\infty}{\overset{D}{\to}}{\cal N}(\bzero,\bV_0)$$ for some invertible matrix $\bV_0$. 
To define $\bV_0$ we need to introduce the following matrices. 
Let $\bE$ be the $k\times k$ matrix with generic element  $\bE(i,j)=E\epsilon_t^2\epsilon_{t-i}\epsilon_{t-j}$. Let $\bD$  be the $k\times d$ matrix whose line $i$ is  $\bD(i,\cdot)=E\epsilon_t\frac{\partial}{\partial\btheta^{\top}}\epsilon_{t+i}(\btheta_0)$,
where $d=2(p+q)+3$ is the dimension of $\btheta$. Let $\bC$  be the $k\times d$ matrix whose line $i$ is  $\bC(i,\cdot)=E\epsilon_t\epsilon_{t-i}\frac{\partial}{\partial\btheta^{\top}}\ell_t(\btheta_0)$. We also need the  $d\times d$ matrix  $\bJ=diag(\bPi ,\bJ_{1},\bJ_{2})$.
In the next theorem, we will show that
$$\bV_0=\frac{1}{(E\epsilon_t^2)^2}
\left\{\bE + \bm C \bm J^{-1} \bD^\top + \bD \bm J^{-1} \bm C^\top + \bD \bSigma \bD^\top\right\}.$$ 
Let the empirical estimators
$$\widehat\bE:=\left(\frac{1}{n}\sum_{t=1}^n\widehat\epsilon_t^2\widehat\epsilon_{t-i}\widehat\epsilon_{t-j}\right)_{1\leq i,j\leq k}, \qquad \widehat\bD:=\left(\frac{1}{n}\sum_{t=1}^n\widehat\epsilon_t\frac{\partial}{\partial \theta_j}\widehat\epsilon_{t+i}\right)_{1\leq i\leq k, 1\le j\leq d},$$
$$\widehat E\epsilon_t^2:=\frac{1}{n}\sum_{t=1}^n\widehat\epsilon_t^2, \qquad \widehat\bC:=\left(\frac{1}{n}\sum_{t=1}^n\widehat\epsilon_t\widehat\epsilon_{t-i}\frac{\partial}{\partial \theta_j}\widetilde\ell_{t}(\widehat{\btheta}_n)\right)_{1\leq i\leq k, 1\le j\leq d},$$
and let $\widehat\bV$ be the empirical estimator of $\bV_0$ obtained by plugging.
We generally have
	\begin{equation}
	\label{convergences}
	\widehat\brho_{1:k}\overset{a.s.}{\underset{n\rightarrow \infty }{%
\rightarrow }}\brho_{1:k}\qquad\mbox{ and }\qquad\widehat{\bV}\overset{a.s.}{\underset{n\rightarrow \infty }{%
\rightarrow }} \bV
	\end{equation}
with $\brho_{1:k}=\mathbf{0}_k$ and $\bV=\bV_0$ when the model is correctly specified, and  possibly with $\brho_{1:k}\neq\mathbf{0}_k$ when the model is misspecified. Let $\chi_k^2(\alpha)$ be the $\alpha$-quantile of the chi-squared distribution with $k$ degrees of freedom.

\noindent For the existence of the matrices $\bD$ and $\bE$, we need the following moment condition. 

\textbf{A7.} For $s=1,2$  we have $EX_{st}^{4}<\infty$.

\noindent We also need the following identifiability condition. 

\textbf{A8.} With probability one, the random variable $X_{1t}-\lambda_{1t}+X_{2t}-\lambda_{2t}$ is not zero.

\noindent The previous assumption is very mild, but necessary in order to rule out pathological distributions, as illustrated by the following example.  
\begin{rem}[An example showing that {\bf A8} is not vacuous]
Assume that $X_{1t}=X_{2t}=c$ for some constant $c\in \mathbb{N}$ such that (for some particular values of $\omega$, $\alpha$ and $\beta$)
$\lambda_{st}\equiv c=\omega+\alpha c +\beta c$, then {\bf A8} is not satisfied.    
\end{rem}
\begin{thm}\label{thm:gof}
Suppose $\pi_t$ follows the Bernoulli INGARCH(1,1) model of Example \ref{ex:BINGARCH}.  Under {\bf A1}-{\bf A8}, we have the null
$$H_0:\eqref{convergences}\mbox{ holds with }\brho_{1:k}=\bzero_k$$
and $\bV=\bV_0$ is non-singular, and the test of rejection region $C=\{n
\widehat\brho_{1:k}^\top
\widehat\bV^{-1}
\widehat\brho_{1:k}
>
\chi_{k}^2(1-\alpha)\}$
has the asymptotic level $\alpha\in(0,1).$ Under the alternative
$$H_1:\eqref{convergences}\mbox{ holds with }\brho_{1:k}\neq\bzero_k, $$
if $\bV$ is invertible, the test of rejection region $C$ is consistent. 
\end{thm}

To better control first-kind error, we apply the random-weighting (RW) bootstrapping method (see \cite{z16} and its references). 
Let a sequence of positive i.i.d.\ random variables $\{w_{t}^{*}\}$, independent of the data sequence $\left\{Y_t\right\}$, with both mean and variance one.
Step 1 of Zhu's RW method consists of computing a bootstrapped estimator of $\btheta_0$ through the numerical optimization of an RW objective function. 
To reduce the computational burden of this first step, we used the trick proposed by \cite{kreiss2011range} and \cite{shimizu2013bootstrap} and define
$$\widehat{\btheta}_n^*
= 
\widehat{\btheta}_n 
+ 
\widehat\bJ^{-1} 
\frac{1}{n}\sum_{t=1}^n
\skakko{w_t^*-1}\frac{\partial}{\partial \btheta} \widetilde{\ell}_t\skakko{\widehat{\btheta}_n}.$$ This trick saves time by using a Newton-Raphson type iteration instead of a full numerical optimization to compute the bootstrap estimator. 
We then proceed to Step 2 of  Zhu's RW method and compute
the bootstrapped quantities
\begin{align*}
\widehat\brho_{1:k}^*
=
\begin{pmatrix}
\widetilde\rho_{1}^*\skakko{\widehat{\btheta}_{n}^*}\\
\vdots\\
\widetilde\rho_{k}^*\skakko{\widehat{\btheta}_{n}^*}
\end{pmatrix}-\widehat\brho_{1:k},
\end{align*}
 where
\begin{align*}
\widetilde\rho_{h}^*\skakko{\cdot}
=
\frac{\widetilde\gamma_{h}^*\skakko{\cdot}}
{\widehat\gamma_{0}}
%{\widetilde\gamma_{0}^*\skakko{\cdot}},
\quad
\text{and}
\quad
\widetilde\gamma_{h}^*\skakko{{\btheta}}
=
\frac{1}{n}\sum_{t=1}^{n}
w_{t}^*
\widetilde\epsilon_{t}\skakko{{\btheta}} \widetilde\epsilon_{t - h}\skakko{{\btheta}}.
\end{align*}

\begin{thm}\label{thm:gof2}Suppose $\pi_t$ follows the Bernoulli INGARCH(1,1) model of Example \ref{ex:BINGARCH}.  
%Assume $E\skakko{\omega_{t}^*}^{2+\delta}<\infty$ for some $\delta>0$.
Conditional on almost all realization of $\{Y_t\}$, under {\bf A1}-{\bf A8}, $\sqrt n \widehat\brho_{1:k}^*$ converges in distribution to the centered Gaussian distribution with variance $\bm V$  as $n\to\infty$.
\end{thm}

Let $\widehat\brho_{1:k}^{*(1)},\dots, \widehat\brho_{1:k}^{*(B)}$ be $B$ independent replications of $\widehat\brho_{1:k}^*$, and let 
$\widehat{\bm V}_{n}^*$ their empirical variance.
In view of the previous theorem, it is natural to consider the tests based on the following p-values
$$
p_{1,n} := 
1-\Psi_{k}\skakko{
n
\widehat\brho_{1:k}^\top
{\widehat{\bm V}_{n}^{*{-1}}}
\widehat\brho_{1:k}}
$$
and 
$$
p_{2,n} := \frac{1}{B}\sum_{j=1}^{B}\mathbbm{1}_{\mkakko{
\widehat\brho_{1:k}^{* (j)\top} \widehat\brho_{1:k}^{* (j)}> \widehat\brho_{1:k}^\top\widehat\brho_{1:k}}},
$$
where $\Psi_{k}(\cdot)$ denotes the cumulative distribution function of the chi-square distribution with $k$ degrees of freedom.
We reject $H_0$ whenever $p_{j,n}<0.05$. 
Based on the Monte Carlo experiments we conducted, it seems that $B=500$ replications are sufficient and that the method is not very sensitive to the distribution of $\omega_t^*$, so that 
%Bernoulli distribution 
the standard exponential distribution used in \cite{z16}  is appropriate.
%}

\section{Numerical study}\label{sec:6}
In this section, we conduct numerical simulations to evaluate the finite-sample performance of the proposed method.

\subsection{Parameter estimation}\label{sec:6.1}
First, we evaluate the finite-sample performance of the mixed Poisson QMLE.
The procedure is as follows:
For each time series length $n\in\{1800, 3600, 7200\}$, we generate 1000 independent trajectories from the MD-INGARCH model
\begin{align}\label{MDINGARH_Pois}
Y_{t}|\mathcal{F}_{t-1}\sim \pi _{t}\mathcal{P}\left( \lambda _{1t}\right)
+\left( 1-\pi _{t}\right) _{-}\mathcal{SP}\left( \lambda _{2t}\right)
\end{align}
or
\begin{align}\label{MDINGARH_NB}
Y_{t}|\mathcal{F}_{t-1}\sim \pi _{t}\mathcal{NB}\left(p, \frac{p \lambda _{1t}}{(1-p)}\right)
+\left( 1-\pi _{t}\right) _{-}\mathcal{SNB}\left(p, \frac{p \lambda _{2t}}{(1-p)}\right)
\end{align}
with a Bernoulli INGARCH structure
$B_{t}|\mathcal{F}_{t-1}^B\sim{\rm Ber}(\pi_t)$,
where
\begin{align*}
\lambda _{st}=\omega _{s}+0.3\left\vert
Y_{t-1}\right\vert +0.3\lambda _{s,t-1}
\text{ for $s=1,2$ and }
\pi _{t} =0.2+0.2B_{t-1}+0.2\pi _{t-1}
\end{align*}
with $p=0.5$, $\omega_1=1$, and $\omega_2=2$.
Then, we compute the mixed Poisson (Q)MLE. Based on the estimators' asymptotic distribution, we also obtain approximations of their standard deviations.

Figure \ref{fig:bias_boxplot} reports boxplots of the bias (estimator minus true value) for each parameter.
For each sample size, the labels Pois and NB correspond to the cases where the data are generated from the mixed Poisson and mixed negative binomial models, respectively, while estimation is carried out by the mixed Poisson (Q)MLE in both cases.
The estimation biases for all parameters are close to zero, and the dispersion of the estimates decreases as the sample size $n$ increases.
As expected, the variability is generally larger when the data are generated from the mixed negative binomial model than from the mixed Poisson model.
\begin{figure}[htbp]
\includegraphics[page=1,width=\linewidth]{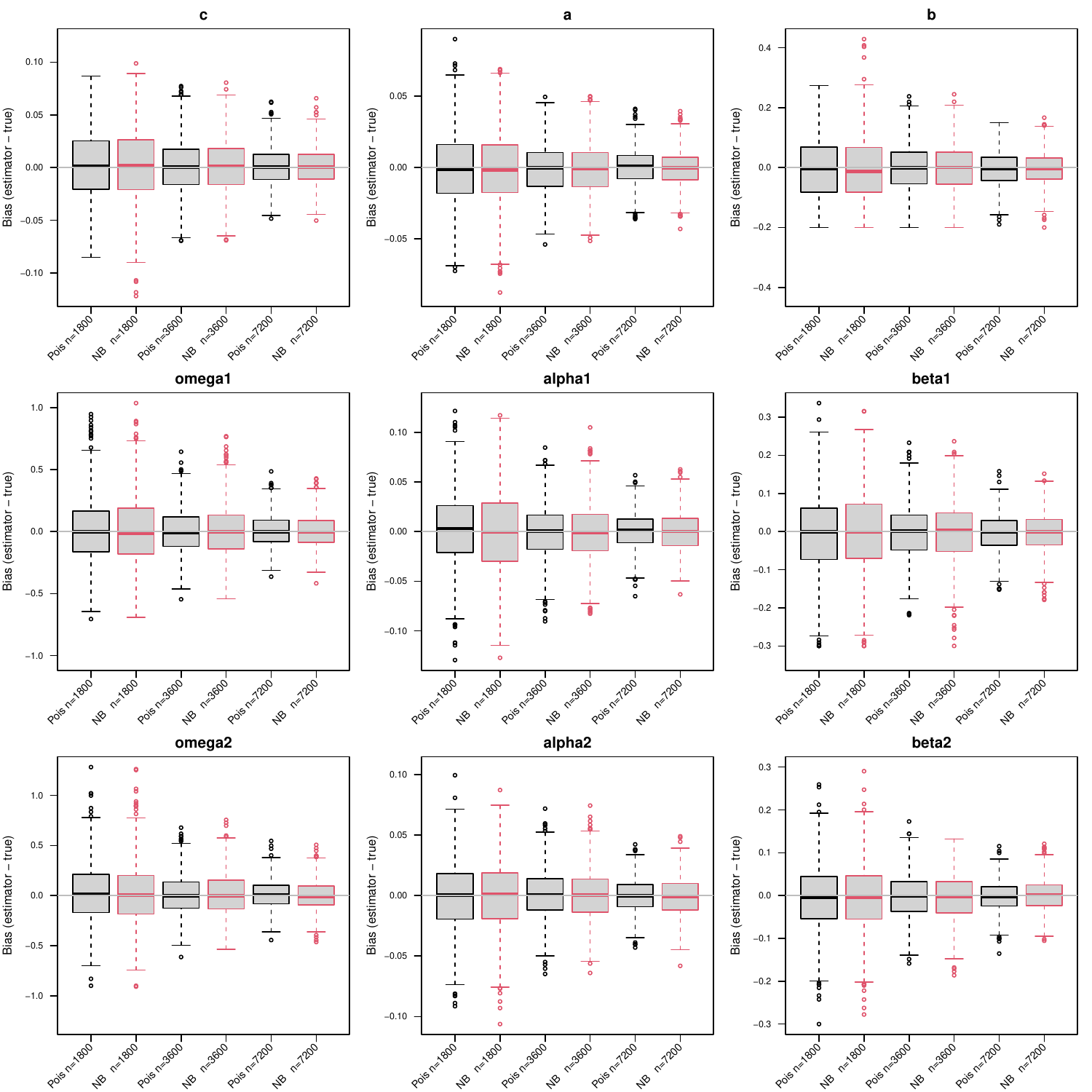}
\caption{
Boxplots of the bias (estimator minus true value) for each parameter based on 1000 replications.
For each parameter, the six boxplots correspond to $n=1800, 3600, 7200$ under data generated from the mixed Poisson model (Pois) and the mixed negative binomial model (NB).
In all cases, estimation is carried out by mixed Poisson (Q)MLE.
}
\label{fig:bias_boxplot}
\end{figure}

Figure \ref{fig:var_boxplot} presents a comparison between the asymptotic standard errors and the empirical standard deviations of the mixed Poisson QMLE. The boxplots summarize the estimated standard deviations obtained from the asymptotic variance expressions, and the solid lines represent the empirical standard deviations computed over replications.
The variances under data generated from the mixed Poisson models are smaller than those under data generated from the mixed NB models.
As the sample size increases, the variances decrease.
Moreover, in most cases, the empirical standard deviations lie within the corresponding boxplots of the estimated standard errors, indicating that the asymptotic variance formulas provide reasonable approximations in finite samples.
Finally, the variance estimators of $c$, $\omega_1$, and $\omega_2$ tend to exhibit an upward bias.
Conversely, the estimators of $b$, $\beta_1$, and $\beta_2$ exhibit a downward bias.

\begin{figure}[htbp]
\includegraphics[page=1,width=\linewidth]{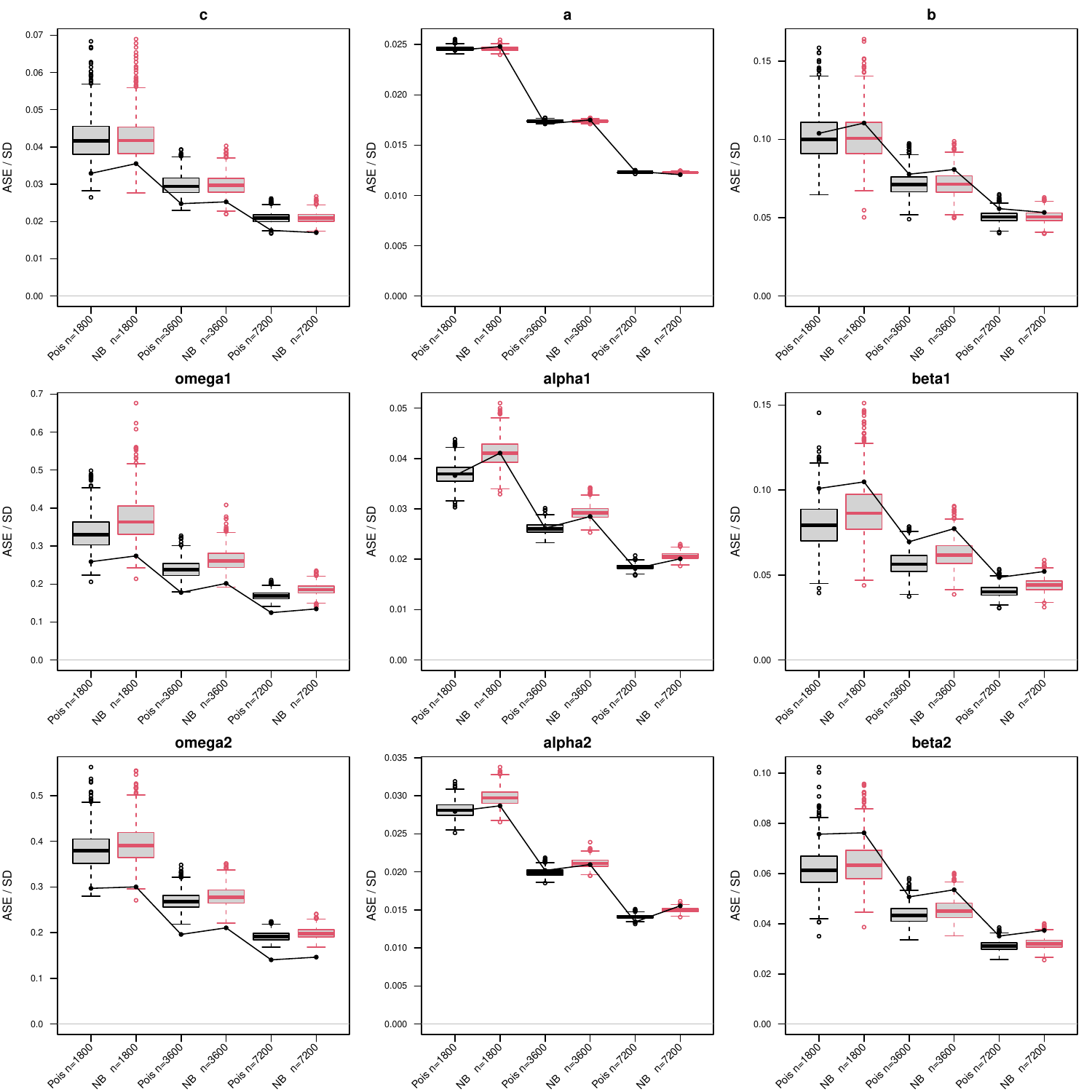}
\caption{
Boxplots of estimated standard deviations of the mixed Poisson QMLE based on the asymptotic variance formulas, with empirical standard deviations computed over replications overlaid as solid lines, for $n=1800$, $3600$, and $7200$, under data generated from the mixed Poisson model (Pois) and the mixed negative binomial model (NB).
}
\label{fig:var_boxplot}
\end{figure}

\subsection{Portmanteau test}
Second, we evaluate the performance of the random weight-based portmanteau test. 
We generate time series of length $n \in \{300, 600, 900\}$ from different models under the null and alternative hypotheses. 
Under the null hypothesis, the data are generated from the models described in Section~\ref{sec:6.1}. 
Under the alternative hypothesis, we modify these models to a log-linear form, where
\begin{align*}
\lambda_{st} = \exp\left(
  \omega_s +
  0.2 \log(|X_{t-1}| + 1) +
  0.2 \log(\lambda_{s,t-1})
\right).
\end{align*}
We then fit the time series using the MD-INGARCH(1,1) model and apply the test with $B = 500$ bootstrap replications and lag order $d = 10$. 
Here, the lag order $d$ specifies how many lags of the residual autocorrelation are included in the test statistic. 
The random weights are generated from an i.i.d.\ standard exponential distribution. 
This procedure is repeated 1000 times, and we compute the empirical size and power of the test.

Figure \ref{fig:res_test} shows empirical rejection probability under the null and alternative hypotheses.
The test based on $p_{1,n}$ exhibits accurate size control, while the test based on $p_{2,n}$ tends to be undersized, that is, its rejection probabilities fall below the nominal significance level.
The choice between the Poisson and negative binomial data-generating processes has minimal impact on the results, as the empirical rejection probabilities are nearly identical in both cases.
Both tests exhibit reasonable power.

\begin{figure}[htbp]
  \centering
\includegraphics[width=\textwidth]{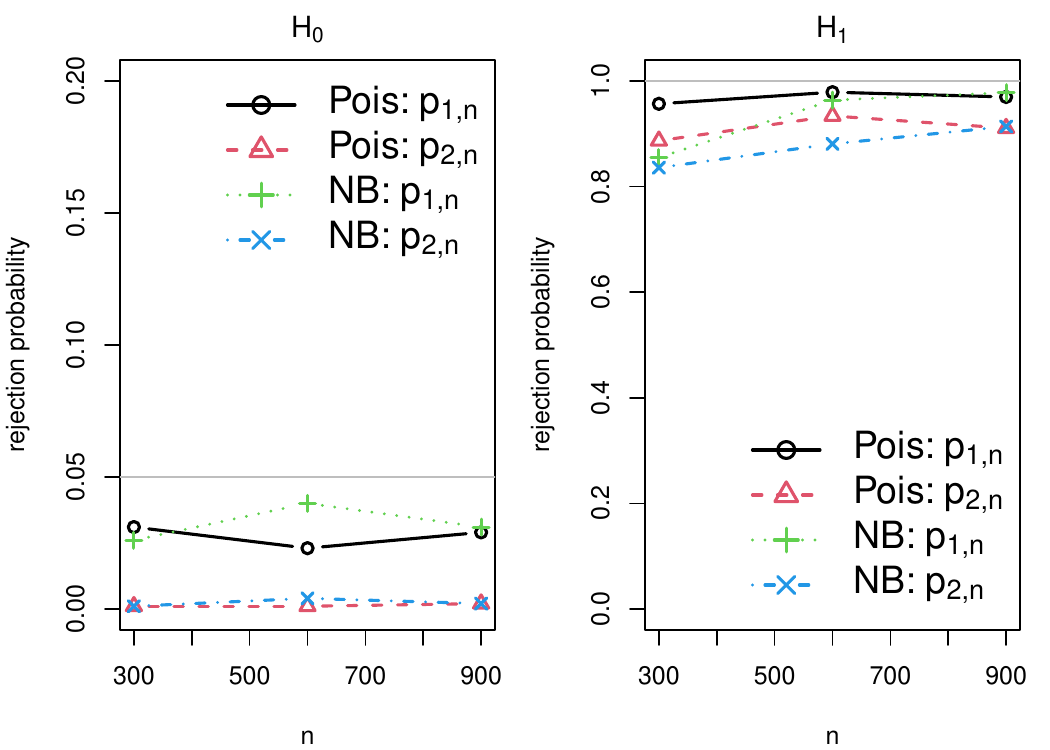}
  \caption{
Empirical sizes and powers of the portmanteau tests based on $p_{1,n}$ and $p_{2,n}$ across different time series lengths $n$ at the nominal level of 0.05, where the data are generated from Poisson (Pois) and negative binomial (NB) linear MD-INGARCH (1,1) models under the null hypothesis (left panel) and
log-linear MD-INGARCH models (1,1) under the alternative hypothesis (right panel),
and are fitted using the same linear MD-INGARCH(1,1) models.
The vertical and horizontal axes represent the rejection probability and the time series length, respectively.
}
  \label{fig:res_test}
\end{figure}

\section{Empirical analysis of Bank of America returns}\label{sec:7}
\subsection{In-sample model fitting and diagnostics}
To evaluate finite-sample performance of the proposed model on real-world data, we analyze the rescaled integer-valued returns of Bank of America. The dataset consists of 4,193 observations recorded every 2 minutes between November 19 and December 19, 2025.
The returns are rescaled by the minimum tick size of one cent, ensuring that the resulting observations are $\mathbb{Z}$-valued.
%This dataset was previously analyzed by \cite{xz22}.
The data was retrieved from Yahoo Finance \citep{YahooFinance} via the yfinance API \citep{yfinance}

Figure \ref{fig:plot_ts} displays the stock price, the rescaled daily returns, their histogram, and the partial autocorrelation function (PACF) plots of the returns and of their signs. 
While the original stock price exhibits a clear trend, the rescaled returns appear to be stationary.
The distribution appears approximately symmetric around zero, although the right tail contains some relatively large positive observations.
%The histogram suggests an asymmetric distribution, with fewer observations in the moderate positive range compared to the corresponding negative range. 
The PACF plots reveal significant serial dependence in both the rescaled returns and their signs.
\begin{figure}[htbp]
\includegraphics[width=0.9\linewidth]{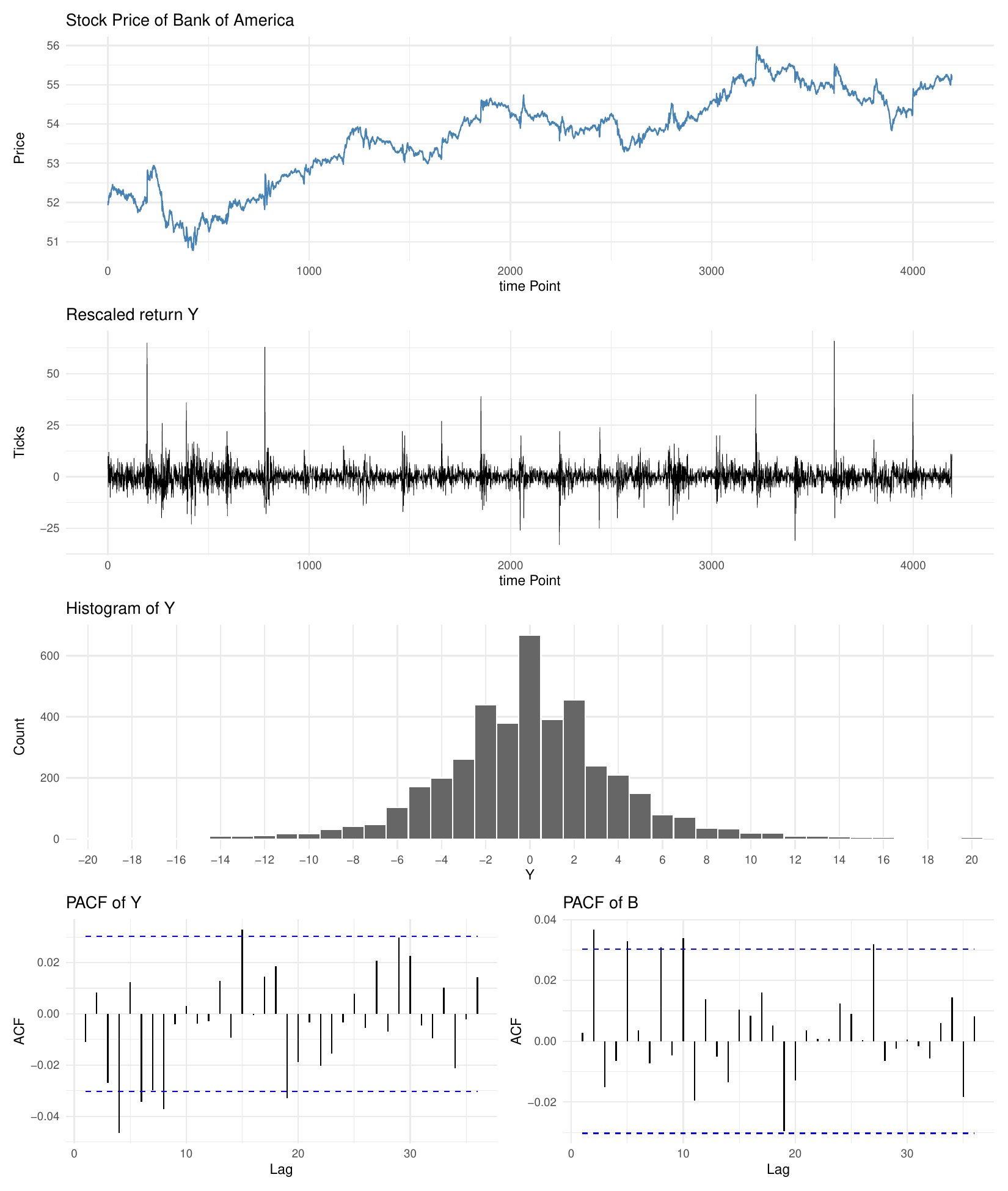}
\caption{The panels in the first and second rows show time series plots for the stock price and the rescaled price differences of Bank of America, respectively.
The panel in third row displays a histogram of the rescaled data.
The left and right panels in the forth row present the autocorrelation function (ACF) plots of the rescaled data and their signs, respectively.
}
\label{fig:plot_ts}
\end{figure}
%The mean and variance for non-negative part and negative part are (,) and (,), respectively. So, the data exhibits strong over-dispersion.
The data contains 15.9\% zero values.

We then fit the data using an MD-INGARCH model with $p=q=1$ and $\{B_t\}$ following Bernoulli INGARCH(1,1).
The estimated values of parameters based on mixed Poisson QMLE are given in Table \ref{tab:realdata}.
\begin{table}[htbp]
\centering
\caption{Parameter estimates of the Poisson MD-INGARCH model for the rescaled daily price differences of Bank of America.
Numbers in parentheses are estimated standard errors.
%\red{Do we list ASE?}
}
\vspace{0.5em}
\label{tab:realdata}
\begin{tabular}{ccccccccc}
\hline
$\widehat c$ & $\widehat a$ & $\widehat b$ & $\widehat\omega_{1}$ & $\widehat\alpha_{1}$ & $\widehat\beta_{1}$ & $\widehat\omega_{2}$ & $\widehat\alpha_{2}$ & $\widehat\beta_{2}$ \\
\hline
0.035&0.010&0.930&0.079&0.143&0.813&0.282&0.112&0.824\\
%(11.58)&(1.00)&(1.40)&(9.38)&(1.84)&(0.80)&(15.54)&(1.75)&(0.85)\\
(0.18)&(0.02)&(0.02)&(0.14)&(0.03)&(0.01)&(0.24)&(0.03)&(0.01)\\
\hline
\end{tabular}
\end{table}
Given that the estimates $(\widehat{\alpha}_1, \widehat{\beta}_1)$ are close to $(\widehat{\alpha}_2, \widehat{\beta}_2)$, the fitted model exhibits only mild asymmetry, which is consistent with the observations from the histogram. 
The finding $\widehat\beta_{2} > \widehat\beta_{1}$ suggests that negative price movements have longer-lasting effects on future volatility than positive movements--a pattern consistent with the leverage effect in equity markets.
Conversely, $\widehat\alpha_1 > \widehat\alpha_2$ indicates that non-negative dynamics exhibit a slightly stronger reaction to recent shocks. Figure \ref{fig:pi_t} indicates that $\lambda_{1t}\skakko{\widehat{\bpsi }_{1n}}$ displays greater variability compared to $\lambda_{2t}\skakko{\widehat{\bpsi }_{2n}}$. 
Overall, the persistence levels measured by $\widehat\alpha_j + \widehat\beta_j$ are comparable across regimes: $\widehat\alpha_{1} + \widehat\beta_{1} = 0.956$ for positive shocks and $\widehat\alpha_{2} + \widehat\beta_{2} = 0.936$ for negative shocks, indicating near-unit-root behavior. This is consistent with the widely documented persistence in financial volatility.
Regarding the switching mechanism, the estimated parameter $\widehat{a} = 0.010$ is small, suggesting that the switching probability $\pi_t$ exhibits low variability over time. This behavior is visually supported by the estimated path of $\pi_t\skakko{\widehat{\bpsi }_{n}}$ in Figure \ref{fig:pi_t}. The near-unit-root estimate of $\widehat{c} + \widehat{a} + \widehat{b} = 0.975$ in the sign process suggests that the probability of a price increase is highly persistent, potentially reflecting slowly changing market sentiment or information flow.

\begin{figure}[htbp]
\begin{center}
\includegraphics[width=\linewidth]{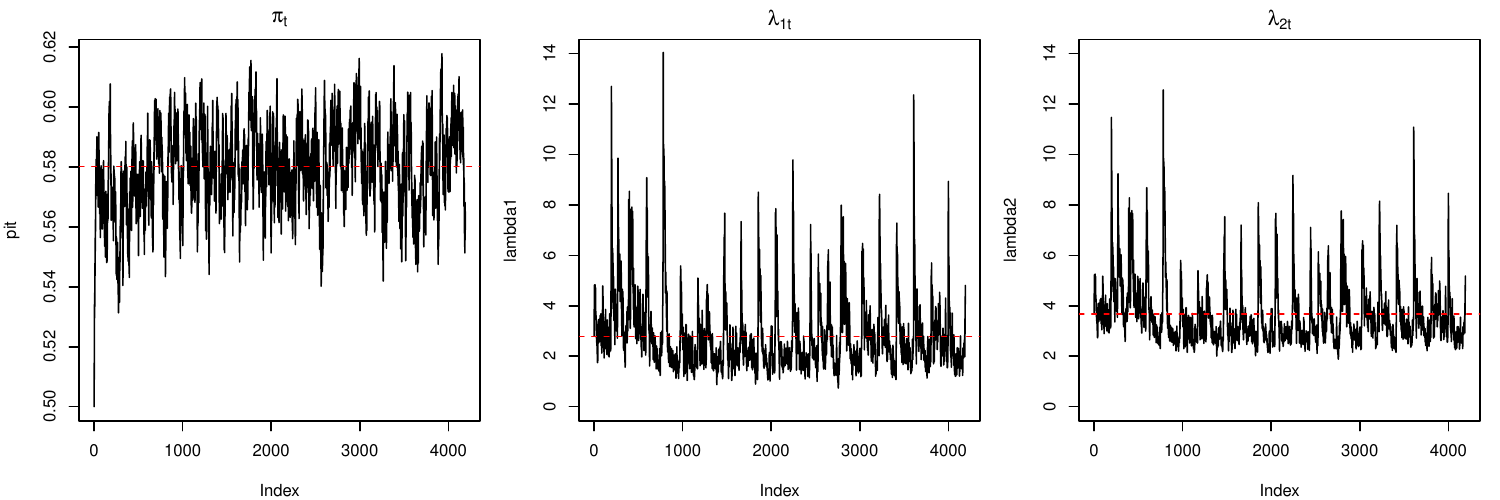}
\caption{
Plots of $\pi_t\skakko{\widehat{\bpsi }_{n}}$ (left), $\lambda_{1t}\skakko{\widehat{\bpsi }_{1n}}$ (middle), and $\lambda_{2t}\skakko{\widehat{\bpsi }_{w
2n}}$ (right), where the red dashed reference lines correspond to the empirical mean of $B_t$ (left), $X_{1t}$ (middle), $X_{2t}$ (right).
}
\label{fig:pi_t}
\end{center}
\end{figure}
The proposed portmanteau test with $B=5000$ yields $p_{1,n}=0.511$ and $p_{2,n}=0.655$, indicating no evidence against the fitted model.
Hence, we conclude that the fitted model adequately captures the dynamics of the tick data.

%Interestingly, the fitted model is strongly asymmetric (see Remark \ref{mele}), with a larger coefficient $\widehat\beta_2$ than $\widehat\beta_1$, which signifies that negative shocks are more persistent than positive ones. This is consistent with the leverage effect that is often observed with series of financial returns.

\subsection{Distributional adequacy via PIT diagnostics}
The Bank of America dataset exhibits overdispersion:
its overall sample mean and variance are 0.0790 and 24.28, respectively.
When decomposed into non-negative and negative components, the non-negative observations have a sample mean of 2.792
and a variance of 16.95, while the negative observations have
a sample mean of $-3.67$ and a variance of 10.17.
Both components therefore display pronounced overdispersion.
Any adequate model, therefore, needs to explicitly account for this feature.
To assess distributional adequacy, we construct non-randomized
probability integral transform (PIT) histograms following \cite{cf14},
and compare our proposed models with existing $\mathbb Z$-valued
time series models.

The existing models considered in the comparison are briefly summarized as follows.
\begin{enumerate}
    \item[] \cite{clz21}:
    a modified Skellam INGARCH model.
    \item[] \cite{ha21}:
    a Poisson INGARCH model multiplied by an i.i.d.\ Bernoulli sequence.
    \item[] \cite{xz22}:
    a shifted geometric INGARCH model multiplied by an i.i.d.\ sequence taking values in $\{-1,0,1\}$.
\end{enumerate}
We consider mixed Poisson and mixed negative binomial (NB) models.
Parameter estimation is based on Mixed Poisson QMLE.
For the mixed NB model, the dispersion parameters $r_1$ and $r_2$,
corresponding to the non-negative and negative parts, respectively,
are estimated by
\begin{align*}
\hat r_1
:=
\skakko{
\frac{1}{n}\sum_{t=1}^n
\frac{
\skakko{Y_t
-\widetilde{\lambda}_{1t}\skakko{\widehat{\bpsi} _{1n}}}^2
I_{\{Y_t\geq0\}}-
\widetilde{\pi }_{t}\left( \widehat{\bphi }_{n}\right)
\widetilde{\lambda}_{1t}\skakko{\widehat{\bpsi} _{1n}}
}{\widetilde{\pi }_{t}\left( \widehat{\bphi }_{n}\right)
\widetilde{\lambda}_{1t}^2\skakko{\widehat{\bpsi} _{1n}}
}
}^{-1}
\end{align*}
and
\begin{align*}
\hat r_2
:=
\skakko{
\frac{1}{n}\sum_{t=1}^n
\frac{
\skakko{Y_t
+\widetilde{\lambda}_{2t}\skakko{\widehat{\bpsi} _{2n}}}^2
I_{\{Y_t<0\}}-
\skakko{1-\widetilde{\pi }_{t}\left( \widehat{\bphi }_{n}\right)}
\skakko{\widetilde{\lambda}_{2t}\skakko{\widehat{\bpsi} _{2n}}-1}
}{
\skakko{1-\widetilde{\pi }_{t}\left( \widehat{\bphi }_{n}\right)}
\skakko{\widetilde{\lambda}_{2t}\skakko{\widehat{\bpsi} _{2n}}-1}^2
}
}^{-1}.
\end{align*}
These estimators are constructed by exploiting the fact that the discrepancy between nonparametric and NB-based conditional variance estimates forms a martingale difference sequence. The estimated dispersion parameters are $\hat r_1=0.804$ and $\hat r_2 =1.449$.

Figure~\ref{fig:PIT} displays the resulting non-randomized PIT histograms for the proposed models and competing approaches.
If the conditional distribution is correctly specified,
the PIT values should be approximately uniformly distributed on $(0,1)$.
Consequently, the relative frequency in each of the $J=10$ bins
of the PIT histogram is expected to be close to $1/J$.
The visual reference bands represent pointwise 95\% intervals
based on a normal approximation to the binwise frequencies under i.i.d.\ uniformity
and are included solely for visual guidance.

Both the proposed mixed Poisson model and the model of \cite{ha21} exhibit significant distortions around 0.5. %, corresponding to the boundaries between the non-negative and negative parts.
This pattern indicates that Poisson-based specifications fail to adequately capture the 15.9\% zero inflation. Although the data exhibits substantial over-dispersion, these models appear to prioritize fitting the extreme values at the expense of zero inflation. Due to the inherent mean-variance equality of the Poisson distribution, the need to accommodate over-dispersion forces an inflation of the mean parameter, which in turn leads to a severe under-estimation of the point mass at zero. 
The model of \cite{clz21} exhibits pronounced deviations from $1/J$ at both extremes. Although this model incorporates a specific modification for the zero mass to handle zero inflation, it remains fundamentally based on the Skellam distribution, which inherits the light-tailed nature of the Poisson distribution. 

In contrast, the proposed NB-based model and the geometric-based model of \cite{xz22} exhibit PIT histograms that are closer to uniformity. The superior performance of the proposed NB-based model can be attributed to the properties of the NB distribution. For a given mean, a smaller dispersion parameter simultaneously leads to a larger probability mass at zero and heavier tails. The small estimated values $\widehat{r}_1 = 0.804$ and $\widehat{r}_2 = 1.449$ allow the model to capture both the zero inflation and the heavy tails. 
Similarly, while the model of \cite{xz22} is based on the geometric distribution, which is a special case of the NB distribution with $r=1$, it achieves an excellent fit by directly assigning probability mass to zero through its specific multiplication mechanism. 
Overall, our proposed model achieves a comparable or superior distributional fit by utilizing the flexibility of the negative binomial specification.

\begin{figure}[htbp]
\begin{center}
\includegraphics[width=0.75\linewidth]{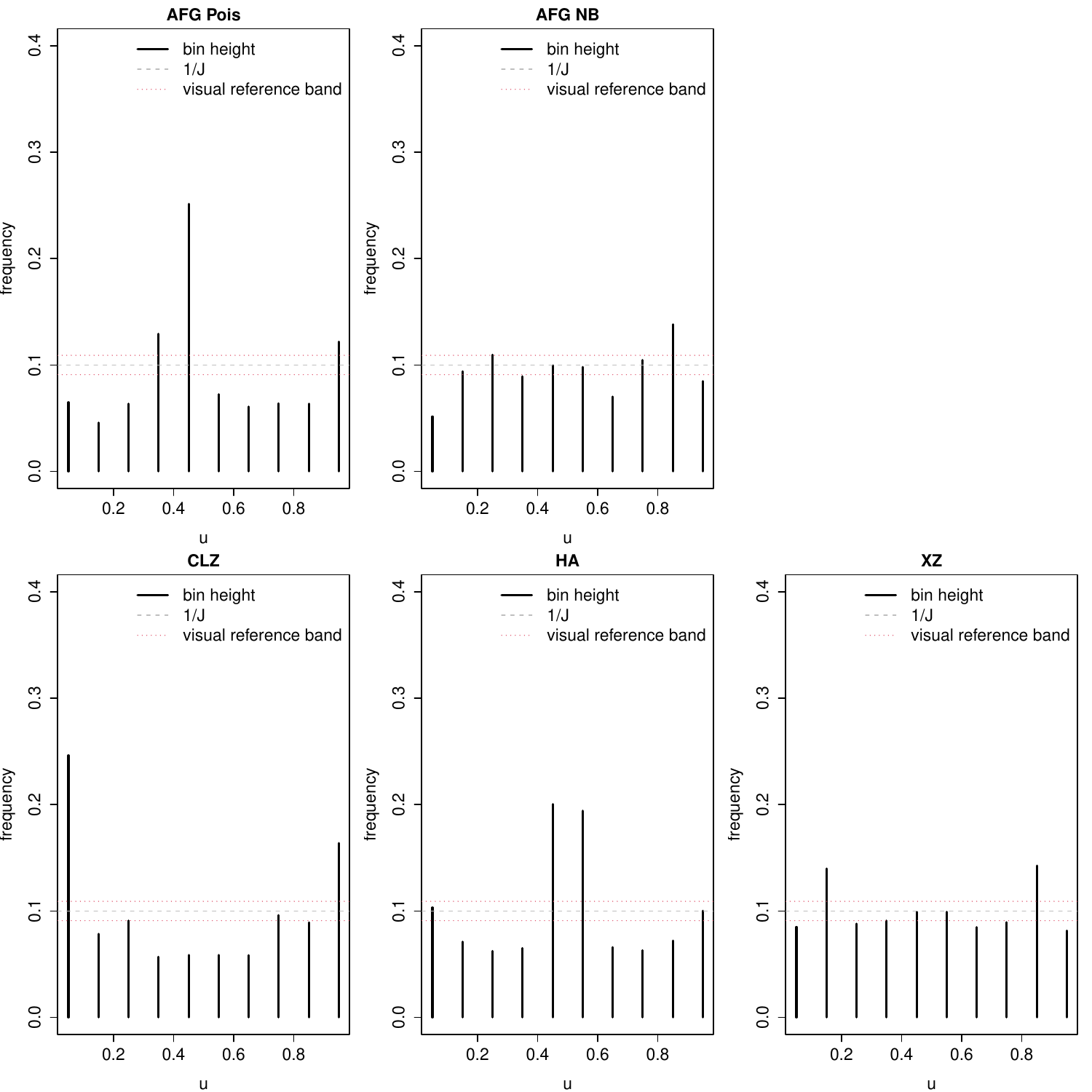}
\caption{
Non-randomized PIT histograms for the Bank of America data.
AFG Pois and AFG NB denote the proposed mixed Poisson and mixed negative binomial
INGARCH-type models, respectively,
while CLZ, HA, and XZ refer to the competing $\mathbb Z$-valued time series models
proposed by \cite{clz21}, \cite{ha21}, and \cite{xz22}, respectively.
The top row corresponds to the proposed models, whereas the bottom row shows
the competing approaches.
Under a correct conditional distributional specification, the PIT values are approximately uniformly distributed,
so that each bin has an expected height of $1/J$ with $J=10$.
The horizontal reference bands indicate pointwise 95\% intervals
based on a normal approximation to the binwise frequencies under i.i.d.\ uniformity
and are included solely for visual guidance.
}
\label{fig:PIT}
\end{center}
\end{figure}

\subsection{Out-of-sample evaluation of sign dynamics}
Out-of-sample evaluation is conducted to examine whether modeling sign dynamics via a Bernoulli INGARCH process improves predictive performance relative to two benchmarks: a fixed probability of $0.5$ and the expanding sample mean.

We employ the mean absolute error (MAE):
\begin{align*}
{\rm MAE}_1:=&
\frac{1}{n-m+1}
\sum_{t=m}^n
\left|B_t
- 
\widetilde{\pi }_{t}\left( \widehat{\bphi }_{t-1:1}\right) \right|,
\quad
{\rm MAE}_2:=
\frac{1}{n-m+1}
\sum_{t=m}^n
\left|B_t
- 
0.5\right|,\\
{\rm MAE}_3:=&
\frac{1}{n-m+1}
\sum_{t=m}^n
\left|B_t
- 
\frac{1}{t-1}\sum_{j=1}^{t-1}B_j
\right|,
\end{align*}
where $\widehat{\bphi }_{t-1:1}$ denotes the estimator based on the expanding subsample $X_1,\ldots,X_{t-1}$ with initial training sizes $m \in \{1000, 1500, 2000, 2500, 3000, 3500, 4000\}$.

Then, we apply the Diebold--Mariano test (see \citealp{forecast}) for the hypotheses
$$
H_0^{{\rm MAE}_2}: {\rm E}\skakko{{\rm MAE}_1-{\rm MAE}_2}=0
\quad
\text{vs}
\quad
H_1^{{\rm MAE}_2}: {\rm E}\skakko{{\rm MAE}_1-{\rm MAE}_2}<0
$$
and
$$
H_0^{{\rm MAE}_3}: {\rm E}\skakko{{\rm MAE}_1-{\rm MAE}_3}=0
\quad
\text{vs}
\quad
H_1^{{\rm MAE}_3}: {\rm E}\skakko{{\rm MAE}_1-{\rm MAE}_3}<0.
$$
\begin{table}[ht]
\centering
\caption{P-values of Diebold-Mariano test based on mean absolute error.}
\begin{tabular}{lccccccc}
\hline
$m$ & 1000 & 1500 & 2000 & 2500 & 3000 & 3500 & 4000 \\ \hline
$H_0^{{\rm MAE}_2}$ & $<0.001$ & $<0.001$ & $<0.001$ & $<0.001$ & $<0.001$ & $<0.001$ & $<0.001$ \\
$H_0^{{\rm MAE}_3}$ & 0.007 & 0.001 & 0.016 & 0.003 & 0.028 & 0.095 & 0.078 \\ \hline
\end{tabular}
\label{tab:dm_mae}
\end{table}
The results in Table \ref{tab:dm_mae} reveal that the Bernoulli INGARCH model outperforms the benchmark ${\rm MAE}_2$ with a significance level of 0.001 across all $m$. 
Moreover, the model also shows superior predictive performance compared to the benchmark ${\rm MAE}_3$ for $m \leq 3000$ at the 0.05 significance level. This suggests that the time-varying conditional probability $\pi_t$ captured by the INGARCH dynamics provides more accurate sign predictions than a simple static estimate of the sample mean.
For $m=3500$ and $4000$, although the INGARCH model still maintains lower MAE values, the p-values for $H_0^{{\rm MAE}_3}$ increase above 0.05. This may be attributed to the decreasing number of out-of-sample observations available for the test as $m$ increases, which reduces the power of the Diebold-Mariano test.

%\subsection{Out-of-sample evaluation of mean}

\section{Conclusion}\label{Conclusion}
This paper proposes a novel and versatile framework for $\mathbb{Z}$-valued time series--a class of data frequently encountered in financial econometrics, particularly in high-frequency price change analysis. 
By introducing a dynamic sign process within a semi-parametric specification, the model addresses important limitations of existing methods, notably the assumption of an independent sign selector. By allowing for time-varying and potentially asymmetric conditional absolute moments, the model can be viewed as a GARCH-type model for $\mathbb{Z}$-valued processes--a setting that is highly relevant to the fields of financial econometrics, market microstructure and risk management.
Conditions for stationarity, ergodicity, and $\beta$-mixing are derived. A robust Mixed Poisson quasi-maximum likelihood estimator is developed, and its consistency and asymptotic normality are established.  Diagnostic tools are also developed. The empirical relevance of the framework is demonstrated through Monte Carlo simulations and an application to high-frequency stock price changes, where the model successfully captures sign dependence, overdispersion, and excess zeros. 
Our analysis on the tick-by-tick price changes of Bank of America stock reveals several economically meaningful findings: 1) the estimated sign dynamics indicate strong persistence in the probability of price increases, challenging the i.i.d.\ sign assumption underlying existing models;
2) the volatility components exhibit asymmetry, which implies that negative shocks have longer-lasting effects on volatility than positive shocks—a classic ``leverage effect'' pattern in financial returns; 3)
out-of-sample evaluation using Diebold-Mariano tests demonstrates that modeling sign dynamics significantly improves prediction of price movement directions compared to usual benchmarks.

More generally, our model offers several tools for financial decision-making. The estimated sign probabilities $\widehat{\pi}_t$ can inform short-term trading strategies by predicting the direction of price movements. The volatility components $\widehat{\lambda}_{1t}$  and $\widehat{\lambda}_{2t}$  provide inputs for value-at-risk calculations that respect the discrete nature of price changes. Moreover, the asymmetry parameters allow regulators to assess whether negative price shocks disproportionately increase market fragility--a question of importance for financial stability monitoring.

In future work, the authors intend to develop a multivariate extension of the MD-INGARCH model that can accommodate vectors of signed count data. This framework is of great importance for financial applications, since no existing multivariate GARCH model can handle $\mathbb{Z}$-valued observations and the existing multivariate count time series (see \cite{fokianos2024multivariate}) do not directly model volatilities.

\section*{Acknowledgments} Christian Francq acknowledges research support from the French National Research Agency (ANR) under grant ANR-21-CE26-0007-01. 
Yuichi Goto acknowledges support from JSPS Grant-in-Aid for Early-Career Scientists (JP23K16851) and the Research Fellowship Promoting International Collaboration of the Mathematical Society of Japan.
The authors used AI-powered writing tools (including ChatGPT, DeepL Write, DeepSeek, and Gemini) to assist with proofreading and to improve the clarity of the writing.

\bibliographystyle{econ}
\bibliography{ref}

\newpage

% === ここからサプリメント用の設定（Section 8 を S1 にリセット） ===
\setcounter{section}{0}
\renewcommand{\thesection}{S\arabic{section}}

\counterwithout{equation}{section}
\setcounter{equation}{0}
\renewcommand{\theequation}{S\arabic{equation}}

\setcounter{figure}{0}
\renewcommand{\thefigure}{S\arabic{figure}}

\setcounter{table}{0}
\renewcommand{\thetable}{S\arabic{table}}
% ==================================================================

\begin{center}

{\Large  SUPPLEMENT TO}\\%[0.5em]
{\Large ``Mixed difference integer-valued GARCH model\\ 
for $\mathbb{Z}$-valued time series''}

\vspace{1.5em}

{\large
Abdelhakim Aknouche$^{1}$,  Christian Francq$^{2}$, and  Yuichi Goto$^{3}$
}

\vspace{1em}

{\small

$^{1}$Qassim University\\
\texttt{aknouche\_ab@yahoo.com}

\vspace{0.5em}

$^{2}$CREST and University of Lille\\
\texttt{christian.francq@ensae.fr}

\vspace{0.5em}

$^{3}$Kyushu University\\
\texttt{yuichi.goto@math.kyushu-u.ac.jp}
}\\
\vspace{1em}
\textbf{Abstract}
\end{center}
\vspace{-1.4em}
\begin{quote}
This document provides supplementary material to the paper ``Mixed difference integer-valued GARCH model for $\mathbb{Z}$-valued time series''.
More specifically, we provide all proofs and complementary illustrations.
\end{quote}

\section{Proofs}\label{sec:8}

\subsection{Proof of Proposition \ref{Pro3.1}}
Let $\left\{ U_{t},t\in \mathbb{Z%
}\right\} $ be an iid sequence of random variables uniformly distributed in $%
\left[ 0,1\right] $, such that $(U_{t})$ and $(B_{t})$ are independent. We
are going to define a solution of $\eqref{3.2a}$-$\eqref{3.2b}$ such that $%
\mathcal{F}_{t}:=\sigma (Y_{s};s\leq t)=\sigma (U_{s},B_{s};s\leq t)$. Note
that $B_{s}\in \mathcal{F}_{t}$ for all $s\leq t$ (since $B_{s}=\mathbbm{1}%
_{\{Y_{s}\geq 0\}}$) and that $\pi _{t}=P(B_{t}=1\mid B_{s},s<t)$. Define $%
Y_{t}^{(k)}=\lambda _{1t}^{(k)}=\lambda _{2t}^{(k)}=0$ for $k\leq 0$, and
for $k>0$ 
\begin{equation}
Y_{t}^{(k)}=B_{t}F_{\lambda _{1t}^{(k)}}^{1\,-}(U_{t})-(1-B_{t})F_{\lambda
_{2t}^{(k)}}^{2\,-}(U_{t})  \label{3.8}
\end{equation}%
where 
\begin{equation}
\lambda _{st}^{(k)}=\omega _{s}+\sum\limits_{i=1}^{q}\alpha
_{si}|Y_{t-i}^{(k-i)}|+\sum\limits_{j=1}^{p}\beta _{sj}\lambda
_{s,t-j}^{(k-j)}.  \label{3.9}
\end{equation}
By induction on $k$, we will show that:

i) when $B_{t}=1$ we have%
\begin{equation}
0\leq \lambda _{1t}^{(k-1)}\leq \lambda _{1t}^{(k)}  \label{3.10}
\end{equation}%
and%
\begin{equation}
0\leq Y_{t}^{(k-1)}=F_{\lambda _{1t}^{(k-1)}}^{1\,-}(U_{t})\leq F_{\lambda
_{1t}^{(k)}}^{1,-}(U_{t})=Y_{t}^{(k)};  \label{3.11}
\end{equation}

ii) when $B_{t}=0$ we have 
\begin{equation}
0\leq \lambda _{2t}^{(k-1)}\leq \lambda _{2t}^{(k)}  \label{3.12}
\end{equation}%
and%
\begin{equation}
Y_{t}^{(k)}=-F_{\lambda _{2t}^{(k)}}^{2\,-}(U_{t})\leq -F_{\lambda
_{2t}^{(k-1)}}^{2\,-}(U_{t})=Y_{t}^{(k-1)}\leq 0.  \label{3.13}
\end{equation}

Indeed, assuming the induction assumption $\lambda _{su}^{(k-1-i)}\leq
\lambda _{su}^{(k-i)}$ and $|Y_u^{(k-1-i)}|\leq |Y_u^{(k-i)}|$ for all $%
i\geq 1$ and $u\in \mathbb{Z}$, and using the property \eqref{3.1} of
monotonicity of $F_{\lambda }^{s-}$ with respect to $\lambda $ we have 
\begin{eqnarray*}
\lambda _{st}^{(k-1)} &=&\omega _{s}+\sum\limits_{i=1}^{q}\alpha
_{si}|Y_{t-i}^{(k-1-i)}|+\sum\limits_{j=1}^{p}\beta _{sj}\lambda
_{s,t-j}^{(k-1-j)} \\
&\leq &\omega _{s}+\sum\limits_{i=1}^{q}\alpha
_{si}|Y_{t-i}^{(k-i)}|+\sum\limits_{j=1}^{p}\beta _{1j}\lambda
_{s,t-j}^{(k-j)} =\lambda _{st}^{(k)},
\end{eqnarray*}%
which shows $\eqref{3.10}$ and $\eqref{3.12}$. Relationships $\eqref{3.11}$
and $\eqref{3.13}$ then follow from \eqref{3.1}.

Thus the sequences $\left(\lambda _{st}^{(k)}\right)_{k}$ and $%
\left(|Y_{t}^{(k)}|\right)_{k}$ are increasing. In addition,%
\begin{eqnarray*}
&&E\left( \left| Y_{t}^{(k)}-Y_{t}^{(k-1)}\right| \mathbbm{1}_{B_{t}=1}
\mid \mathcal{F}_{t-1}\right) 
=E\left\{ \left(F_{\lambda _{1t}^{(k)}}^{1\,-}(U_{t})- F_{\lambda
_{1t}^{(k-1)}}^{1\,-}(U_{t})\right)\mathbbm{1}_{B_{t}=1} \mid \mathcal{F}%
_{t-1}\right\} \\
&=&E\left\{ \left(F_{\lambda _{1t}^{(k)}}^{1\,-}(U_{t})- F_{\lambda
_{1t}^{(k-1)}}^{1\,-}(U_{t})\right)\mid \mathcal{F}_{t-1}\right\} E\left( %
\mathbbm{1}_{B_{t}=1} \mid \mathcal{F}_{t-1}\right) =\left( \lambda _{1t}^{(k)}-\lambda _{1t}^{(k-1)}\right)\pi_t 
\end{eqnarray*}%
and%
\begin{eqnarray*}
E\left( \left\vert Y_{t}^{(k)}-Y_{t}^{(k-1)}\right\vert \mathbbm{1}%
_{\{B_{t}=0\}} \mid \mathcal{F}_{t-1}\right) 
&=&\left( \lambda _{2t}^{(k)}-\lambda _{2t}^{(k-1)}\right)(1-\pi_t). 
\end{eqnarray*}%
We then have 
$$
E\left\vert Y_{t}^{(k)}-Y_{t}^{(k-1)}\right\vert \\
=E\left\{\pi_t\left( \lambda _{1t}^{(k)}-\lambda
_{1t}^{(k-1)}\right)\right\} +E\left\{(1-\pi_t )\left( \lambda
_{2t}^{(k)}-\lambda _{2t}^{(k-1)}\right) \right\}.
$$
As the sign of $Y_{t}^{(k)}$ does not vary with respect to $k>0$, we also
have 
\begin{equation*}
E\left( |Y_{t}^{(k)}|-|Y_{t}^{(k-1)}|\right) =E\left\{\pi_t\left( \lambda
_{1t}^{(k)}-\lambda _{1t}^{(k-1)}\right)\right\} +E\left\{(1-\pi_t )\left(
\lambda _{2t}^{(k)}-\lambda _{2t}^{(k-1)}\right) \right\}.
\end{equation*}

First, consider the case $p=q=1$. From \eqref{3.9}, we obtain 
\begin{align*}
&E\left( \lambda _{st}^{(k)}-\lambda _{st}^{(k-1)}\right) \\
=&\alpha _{s}E\left( |Y_{t-1}|^{(k-1)}-|Y_{t-1}|^{(k-2)}\right) +\beta
_{s}E\left( \lambda _{s,t-1}^{(k-1)}-\lambda _{s,t-1}^{(k-2)}\right) \\
=&\alpha _{s}\left\{ E\pi_{t-1}\left( \lambda _{1,t-1}^{(k-1)}-\lambda
_{1,t-1}^{(k-2)}\right) +E(1-\pi_{t-1} )\left( \lambda
_{2,t-1}^{(k-1)}-\lambda _{2,t-1}^{(k-2)}\right) \right\} +\beta _{s}E\left(
\lambda _{s,t-1}^{(k-1)}-\lambda_{s,t-1}^{(k-2)}\right) \\
=&\left\{%
\begin{array}{lll}
E(\alpha _{1}\pi_{t-1} +\beta _{1})\left( \lambda _{1,t-1}^{(k-1)}-\lambda
_{1,t-1}^{(k-2)}\right) +\alpha_{1}E(1-\pi_{t-1} ) \left( \lambda
_{2,t-1}^{(k-1)}-\lambda _{2,t-1}^{(k-2)}\right) & \mbox{if} & s=1 \\ 
\alpha _{2} E\pi_{t-1}\left( \lambda_{1,t-1}^{(k-1)}-\lambda
_{1,t-1}^{(k-2)}\right) + E\left( \alpha _{2}(1-\pi_{t-1} )+\beta
_{2}\right)\left( \lambda _{2,t-1}^{(k-1)}-\lambda _{2,t-1}^{(k-2)}\right) & %
\mbox{if} & s=2.%
\end{array}%
\right.
\end{align*}%
Let $\boldsymbol{\chi }_{t}^{(k)}=\boldsymbol{\lambda} _{t}^{(k)}-%
\boldsymbol{\lambda} _{t}^{(k-1)}$, with $\boldsymbol{\lambda}
_{t}^{(k)}=\left( \lambda _{1t}^{(k)},\lambda _{2t}^{(k)}\right)^\top$.
We have 
\begin{equation*}
E\boldsymbol{\chi }_{t}^{(k)} =E\left(\boldsymbol{A}_{t-1} \boldsymbol{\chi }%
_{t-1}^{(k-1)}\right) =E\left(\boldsymbol{A}_{t} \boldsymbol{\chi }%
_{t}^{(k-1)}\right)\leq \boldsymbol{A} E\boldsymbol{\chi }_{t}^{(k-1)}
\end{equation*}%
element by element, using the positivity of all coefficients in the vectors
and matrices, where 
\begin{equation*}
\boldsymbol{A}_t =\left( 
\begin{array}{cc}
\alpha _{1}\pi_t +\beta _{1} & \alpha_{1} (1-\pi_t ) \\ 
\alpha_{2}\pi_t & \alpha_{2}(1-\pi_t )+\beta _{2}%
\end{array}%
\right),\quad \boldsymbol{A} =\left( 
\begin{array}{cc}
\alpha _{1}\pi_1^+ +\beta _{1} & \alpha_{1} \pi_0^+ \\ 
\alpha_{2}\pi_1^+ & \alpha_{2}\pi_0^++\beta _{2}%
\end{array}%
\right).
\end{equation*}

If $\rho (\boldsymbol{A} )<1$ then $E\left( \boldsymbol{\chi }%
_{t}^{(k)}\right) \rightarrow 0$ as $k\rightarrow \infty $ at an exponential
rate and $(E\boldsymbol{\lambda}_{t}^{(k)})_{k}$ is a Cauchy sequence. This
entails that $\boldsymbol{\lambda}_t:=\lim_{k\to\infty}\uparrow\boldsymbol{%
\lambda} _{t}^{(k)}$ has a finite expectation, and thus is finite almost
surely. Since $\boldsymbol{\lambda}_t$ is a measurable function of $%
\{B_s,U_s; s<t\}$, the process $(\boldsymbol{\lambda}_t)_{t\in\mathbb{Z}}$
is stationary and ergodic. By the same arguments, $Y_t$ defined as the limit
of $Y_{t}^{(k)}$  in $L^{1}$ and also almost surely, as $k\to\infty$, is
stationary and ergodic, and  \eqref{moment1} holds true. For general orders $p$ and $q$, the results \eqref{3.7}-\eqref{moment1}  follow
similarly.

Now, if there exists a stationary process $
\{Y_t\}$ satisfying $\eqref{3.2a}$-$\eqref{3.2b}$ and \eqref{moment1} then 
$$(1-\sum_{j=1}^{p}\beta _{sj})E\lambda _{st}=\omega _{s}+\sum_{i=1}^{q}\alpha _{si}E\left\vert
Y_{t}\right\vert
$$
The condition \eqref{3.8a} is necessary for the existence of a solution $E\lambda_{st}\geq 0$ to this equation. 
$\square $

\subsection{Proof of Proposition \ref{Pro3.1CNS}}
We only need to show that the condition \eqref{3.7} is necessary; the rest is obvious.  For simplicity, we will write down the proof for $p=q=1$.
Letting $\boldsymbol{\lambda}
_{t}=\left( \lambda _{1t},\lambda _{2t}\right)^\top$
we have 
\begin{equation*}
\blambda_{t} = \bomega+ \bA \blambda_{t-1},\quad \bomega=(\omega_1,\omega_2)^\top
\end{equation*}%
If $\blambda:=E\blambda_t$ is finite, we have
$
\bA \blambda< \blambda
$
and  the result follows from Corollary 8.1.29 of \cite{horn2012matrix}.
$\square $

\subsection{Proof of Proposition \ref{prop:parameter}}
 If the sequences $(Y_{t})$ and $(\lambda _{st})$ are
stationary and have finite means, then we have
\begin{equation*}
\skakko{1-\sum\limits_{j=1}^{p}\beta _{sj}}
E\left( \lambda _{st}\right) ={\omega _{s}+\sum\limits_{i=1}^{q}\alpha
_{si}E|Y_{t-i}|},\quad\text{for $s=1,2$}
\end{equation*}%
and 
\begin{equation*}
\left( 1-\pi \frac{\sum\limits_{i=1}^{q}\alpha _{1i}}{1-\sum%
\limits_{j=1}^{p}\beta _{1j}}-\left( 1-\pi \right) \frac{%
\sum\limits_{i=1}^{q}\alpha _{2i}}{1-\sum\limits_{j=1}^{p}\beta _{2j}}%
\right) E|Y_{t}|=
\pi \frac{\omega _{1}}{1-\sum\limits_{j=1}^{p}\beta _{1j}}+(1-\pi )\frac{\omega _{2}}{1-\sum\limits_{j=1}^{p}\beta _{2j}}.
\end{equation*}%
By positivity of $E(\lambda _{st})$  and $E|Y_{t}|$, it is thus necessary that the following conditions hold 
\begin{equation*}
0<1-\sum\limits_{j=1}^{p}\beta _{sj} 
\quad \text{for $s=1,2$}\quad
\text{and}
\quad
\pi \frac{\sum\limits_{i=1}^{q}\alpha _{1i}}{1-\sum\limits_{j=1}^{p}\beta
_{1j}}+\left( 1-\pi \right) \frac{\sum\limits_{i=1}^{q}\alpha _{2i}}{%
1-\sum\limits_{j=1}^{p}\beta _{2j}}<1,
\end{equation*}%
which are exactly \eqref{3.8a} and \eqref{3.8b}, respectively. $\square $

\subsection{Proof of Proposition \ref{prop:NSC}}

First, we derive the explicit form of $\rho \left( \boldsymbol{A}^{(1)} \right)$. 
A simple algebra gives that
\begin{align*}
&|\boldsymbol{A}^{\left( 1\right) }-\lambda I_2|\\
=&
\begin{vmatrix}
\alpha_{11}\pi+\beta_{11} - \lambda & \alpha_{11}(1-\pi)\\ 
\alpha_{21}\pi & \alpha_{21}(1-\pi)+\beta_{21} -\lambda
\end{vmatrix}\\
=&
(\alpha_{11}\pi+\beta_{11} - \lambda)(\alpha_{21}(1-\pi)+\beta_{21} -\lambda)
-
\alpha_{11}(1-\pi)\alpha_{21}\pi\\
=&
\lambda^2
-
\lambda (\alpha_{11}\pi+\beta_{11} + \alpha_{21}(1-\pi)+\beta_{21})
+
(\alpha_{11}\pi+\beta_{11}) (\alpha_{21}(1-\pi)+\beta_{21})
-
\alpha_{11}(1-\pi)\alpha_{21}\pi\\
=&
\lambda^2
-
\lambda (\alpha_{11}\pi+\beta_{11} + \alpha_{21}(1-\pi)+\beta_{21})
+
\alpha_{11}\pi\beta_{21}
+
\alpha_{21}(1-\pi)\beta_{11}
+
\beta_{11}\beta_{21}.
\end{align*}
Since
\begin{align*}
&(\alpha_{11}\pi+\beta_{11} + \alpha_{21}(1-\pi)+\beta_{21})^2
- 4 (\alpha_{11}\pi\beta_{21}
+
\alpha_{21}(1-\pi)\beta_{11}
+
\beta_{11}\beta_{21})\\
=&
(\alpha_{11}\pi+\beta_{11} - \alpha_{21}(1-\pi) -\beta_{21})^2
+ 4 \alpha_{11}\pi \alpha_{21}(1-\pi)\geq 0,
\end{align*}
the equation $|\boldsymbol{A}^{\left( 1\right) }-\lambda I_2|=0$ have real roots as a function of $\lambda$. The solutions are given by
\begin{align*}
\lambda
=& 
\frac{1}{2}
\Big(
\alpha_{11}\pi+\beta_{11} + \alpha_{21}(1-\pi)+\beta_{21}\\
&\pm
\sqrt{
(\alpha_{11}\pi+\beta_{11} + \alpha_{21}(1-\pi)+\beta_{21})^2
- 4 (\alpha_{11}\pi\beta_{21}
+
\alpha_{21}(1-\pi)\beta_{11}
+
\beta_{11}\beta_{21})
}
\Big),
\end{align*}
and, thus,
\begin{align*}
\rho \left( \boldsymbol{A}^{(1)} \right)
=&
\frac{1}{2}
\Big(
 \alpha_{11}\pi+\beta_{11} + \alpha_{21}(1-\pi)+\beta_{21}\\
&+
\sqrt{
(\alpha_{11}\pi+\beta_{11} + \alpha_{21}(1-\pi)+\beta_{21})^2
- 4 (\alpha_{11}\pi\beta_{21}
+
\alpha_{21}(1-\pi)\beta_{11}
+
\beta_{11}\beta_{21})
}
\Big).
\end{align*}
Then, we observe that
\begin{align*}
&
\rho \left( \boldsymbol{A}^{(1)} \right)< 1\\
\Leftrightarrow&
\sqrt{
(\alpha_{11}\pi+\beta_{11} + \alpha_{21}(1-\pi)+\beta_{21})^2
- 4 (\alpha_{11}\pi\beta_{21}+\alpha_{21}(1-\pi)\beta_{11}+\beta_{11}\beta_{21})
}\\
&<
2 - (\alpha_{11}\pi+\beta_{11} + \alpha_{21}(1-\pi)+\beta_{21})\\
\Leftrightarrow&
- 4 (\alpha_{11}\pi\beta_{21}
+
\alpha_{21}(1-\pi)\beta_{11}
+
\beta_{11}\beta_{21})
<
4 - 4(\alpha_{11}\pi+\beta_{11} + \alpha_{21}(1-\pi)+\beta_{21})\\
&\text{and}\quad
\alpha_{11}\pi+\beta_{11} + \alpha_{21}(1-\pi)+\beta_{21} < 2\\
\Leftrightarrow&
\alpha_{11}\pi(1-\beta_{21})
+
\alpha_{21}(1-\pi)(1-\beta_{11})
<
(1-\beta_{11})(1-\beta_{21})\\
&\text{and}\quad
\alpha_{11}\pi+\beta_{11} + \alpha_{21}(1-\pi)+\beta_{21} < 2.
\end{align*}

Suppose that
\eqref{3.7} holds.
Since every element of ${\rm diag}(\beta_{11},\beta_{21})$ is smaller than the corresponding element of  $\rho\left(\bA^{(1)}\right)$, 
it follows from Theorem 8.1.18 of \cite{horn2012matrix} that 
$\beta_{11},\beta_{21}<1$. This is exactly the condition \eqref{3.8a}.
Under \eqref{3.8a}, the condition
$\alpha_{11}\pi(1-\beta_{21})
+
\alpha_{21}(1-\pi)(1-\beta_{11})
<
(1-\beta_{11})(1-\beta_{21})$ is equivalent to \eqref{3.8b}.
Thus, we conclude that 
\eqref{3.7} implies
both \eqref{3.8a} and \eqref{3.8b}.

Conversely, suppose \eqref{3.8a} and \eqref{3.8b} hold.
Then, the inequality
$\alpha_{11}\pi(1-\beta_{21})
+
\alpha_{21}(1-\pi)(1-\beta_{11})
<
(1-\beta_{11})(1-\beta_{21})$ is automatically satisfied.
It remains to check 
the condition $\alpha_{11}\pi+\beta_{11} + \alpha_{21}(1-\pi)+\beta_{21} < 2$
is also implied by \eqref{3.8a} and \eqref{3.8b}.
Indeed, we can see that
\begin{align*}
&
\frac{\alpha_{11}\pi}{1-\beta_{11}}
+
\frac{\alpha_{21}(1-\pi)}{1-\beta_{21}}
<1
\text{ and }
0<\beta_{11},\beta_{21}<1\\
\Rightarrow&
\frac{\alpha_{11}\pi}{1-\beta_{11}}<1,\quad
\frac{\alpha_{21}(1-\pi)}{1-\beta_{21}}<1,
\text{ and }
0<\beta_{11},\beta_{21}<1\\
\Leftrightarrow&
{\alpha_{11}\pi}<1-\beta_{11},\quad
{\alpha_{21}(1-\pi)}<{1-\beta_{21}},
\text{ and }
0<\beta_{11},\beta_{21}<1\\
\Rightarrow&
{\alpha_{11}\pi}+\beta_{11}+
{\alpha_{21}(1-\pi)}+{\beta_{21}}<2,
\text{ and }
0<\beta_{11},\beta_{21}<1.
\end{align*}
Thus, \eqref{3.8a} and \eqref{3.8b} together imply \eqref{3.7}. This completes the proof. $\square $

\subsection{Proof of Proposition~\ref{mixingNL}}
Let $\left\{ U_{t},t\in \mathbb{N
}\right\} $ and $\left\{ V_{t},t\in \mathbb{N
}\right\} $  be two independent sequences of iid random variables uniformly distributed in $\left[ 0,1\right] $.
For $k=1,2$ define 
$\{B_{t}^{[k]}, t\geq 0\}$ and $
\{Y_{t}^{[k]}, t\geq 0\}$  by 
\begin{eqnarray}  \label{rouge}
B_{t}^{[k]}&=&\mathbbm{1}_{V_{t}>1-\pi_t^{[k]}},\quad Y_{t}^{[k]}=B_{t}^{[k]}F_{\lambda_{1t}^{[k]}}^{1,%
\,-}(U_{t})-(1-B_{t}^{[k]})F_{%
\lambda_{2t}^{[k]}}^{2,\,-}(U_{t}), \\
\pi_t^{[k]}&=&c+a B_{t-1}^{[k]} +b\pi_{t-1}^{[k]},\quad \lambda_{st}^{[k]}=\omega _{s}+\sum\limits_{i=1}^{q}\alpha
_{si}|Y_{t-i}^{[k]}|+\sum\limits_{j=1}^{p}\beta
_{sj}\lambda_{s,t-j}^{[k]},
\end{eqnarray}
for $t> 0$, where 
\begin{equation*}
\boldsymbol{Z}_{0}^{[1]}=\left(B^{[1]}_{0},\pi^{[1]}_{0},Y^{[1]}_{0},%
\ldots,Y^{[1]}_{1-q},\lambda^{[1]}_{1,0},
\ldots,\lambda^{[1]}_{1,1-p},\lambda^{[1]}_{2,0},
\ldots,\lambda^{[1]}_{2,1-p}\right)
\end{equation*}
and 
\begin{equation*}
\boldsymbol{Z}_{0}^{[2]}=\left(B^{[2]}_{0},\pi^{[2]}_{0},Y^{[2]}_{0},%
\ldots,Y^{[2]}_{1-q},\lambda^{[2]}_{1,0},
\ldots,\lambda^{[2]}_{1,1-p},\lambda^{[2]}_{2,0},
\ldots,\lambda^{[2]}_{2,1-p}\right)
\end{equation*}
are independent, independent of $(U_t)$ and $(V_t)$, and follow the law of \begin{equation*}
\boldsymbol{Z}_t=\left(B_t,\pi_t,Y_{t},\ldots,Y_{t-q},\lambda_{1t},\ldots, 
\lambda_{1,t-p},\lambda_{2t},\ldots,  \lambda_{2,t-p}\right). 
\end{equation*} The
distribution of $Y_{0:\infty}^{[k]}$ is thus equal to that of $%
Y_{0:\infty},$ and we have $P(Y_{h:\infty}\in A)=P(Y_{h:\infty}^{[2]}\in A\mid \boldsymbol{Z}_{0}^{[1]})$. We also have $P(Y_{h:\infty}\in A\mid Y_{-\infty:0})=P(Y_{h:\infty}\in A\mid Z_{0})=P(Y_{h:\infty}^{[1]}\in A\mid \boldsymbol{Z}_{0}^{[1]})$.

By the coupling arguments used to show (5.9) in Neumann (2011) 
(see also (5.6) in Davis and Liu (2016) or the proof of Theorem 3.3 in 
\cite{af21}), we then have 
\begin{align*}
\beta_{Y}(h)&=E\sup_{A\in \mathcal{B}}
\left|E\left(\mathbbm{1}_{Y^{[1]}_{h:\infty}\in A}\mid 
\boldsymbol{Z}^{[1]}_{0}\right)-E\left(\mathbbm{1}_{Y_{h:\infty}^{[2]}\in A}\mid \boldsymbol{Z}_{0}^{[1]}\right)\right| \\
&=E\sup_{A\in \mathcal{B}} \left|E\left(\left\{\mathbbm{1}_{Y^{[1]}_{h:\infty}\in A}-\mathbbm{1}_{Y_{h:\infty}^{[2]}\in A}\right\}\mid \boldsymbol{Z}%
^{[1]}_{0}\right)\right| \leq E\left(E\left\{\mathbbm{1}_{Y^{[1]}_{h:%
\infty}\neq Y_{h:\infty}^{[2]}}\mid \boldsymbol{Z}%
^{[1]}_{0}\right\}\right) \\
&=E
\mathbbm{1}_{Y_{h:\infty}^{[1]}\neq Y_{h:\infty}^{[2]}}
\leq \sum_{k=0}^{\infty}E%
\left|Y^{[1]}_{h+k}- Y^{[2]}_{h+k}\right|\mathbbm{1}_{\mathbb{B}_h}+P\left(
\overline{\mathbb{B}_h}\right).
\end{align*}
where $\mathbb{B}_h$ is the event $\left\{B_{[{h}/{2}]:
	\infty}^{[1]}= B_{[{h}/{2}]:\infty}^{[2]}\right\}$, $\overline{\mathbb{B}_h}$ the complement of $\mathbb{B}_h$,  and $[{h}/{2}]$ denotes the integer part of ${h}/{2}$.
The last inequality holds because $\left|Y^{[1]}_{h+k}-
Y^{[2]}_{h+k}\right|$ is valued in $\mathbb{N}$. For $t>0$, let ${\cal F}^*_{t-1}$ be the sigma-field generated by $\{U_s,V_s: 1\leq s<t\}$, $\boldsymbol{Z}^{[1]}_{0}$, and $\boldsymbol{Z}^{[2]}_{0}$. 

For $t>0$, we have
$$E\left(\left|B^{[1]}_{t}- B^{[2]}_{t}\right|\mid {\cal F}^*_{t-1}\right)=\left|\pi^{[1]}_{t}- \pi^{[2]}_{t}\right|=\left|a\left(B^{[1]}_{t-1}- B^{[2]}_{t-1}\right)+b\left(\pi^{[1]}_{t-1}- \pi^{[2]}_{t-1}\right)\right|$$ and thus
$$P\left(B^{[1]}_{t}\neq B^{[2]}_{t}\right)=E\left|B^{[1]}_{t}- B^{[2]}_{t}\right|\leq (a+b)E\left|B^{[1]}_{t-1}- B^{[2]}_{t-1}\right|.$$
It follows that
$$P\left(
\overline{\mathbb{B}_h}\right)\leq \sum_{k=0}^{\infty}E
\left|B^{[1]}_{[h/2]+k}- B^{[2]}_{[h/2]+k}\right|\leq K\varrho^h,$$
with $\varrho>\sqrt{a+b}$, $\varrho\in(0,1)$, and $K=1/(1-\varrho^2).$

Let $\mathbb{B}_{[{h}/{2}]:
	t}$ be the event $\left\{B_{[{h}/{2}]:
	t}^{[1]}= B_{[{h}/{2}]:
	t}^{[2]}\right\}$ for $t\geq[{h}/{2}]$.
Now, using \eqref{3.1}, for $j=1,2,$ we have
\begin{eqnarray*}
	\left|Y^{[1]}_{t}- Y^{[2]}_{t}\right|\mathbbm{1}_{B_t^{[1]}=B_t^{[2]}=2-j}&=&\left\{F_{\lambda_{jt}^{[1]}}^{j,
	\,-}(U_{t})-F_{\lambda_{jt}^{[2]}}^{j,
	\,-}(U_{t})\right\}\mathbbm{1}_{B_t^{[1]}=B_t^{[2]}=2-j}\mathbbm{1}_{\lambda_{jt}^{[1]}\geq \lambda_{jt}^{[2]}}\\&&+\left\{F_{\lambda_{jt}^{[2]}}^{j,
	\,-}(U_{t})-F_{\lambda_{jt}^{[1]}}^{j,
	\,-}(U_{t})\right\}\mathbbm{1}_{B_t^{[1]}=B_t^{[2]}=2-j}\mathbbm{1}_{\lambda_{jt}^{[1]}< \lambda_{jt}^{[2]}}.
\end{eqnarray*}
Therefore 
\begin{eqnarray*}
E\skakko{\left|Y^{[1]}_{t}- Y^{[2]}_{t}\right|
\mathbbm{1}_{\mathbb{B}_{[{h}/{2}]:t}}}
&=&
\sum_{j=1,2}
E\skakko{E\left(\left|Y^{[1]}_{t}- Y^{[2]}_{t}\right|
\mathbbm{1}_{B_t^{[1]}=B_t^{[2]}=2-j}
\mid
{\cal F}^*_{t-1}\right)
\mathbbm{1}_{\mathbb{B}_{[{h}/{2}]:t-1}}}
\\
&=&
\sum_{j=1,2}
E\left\{\left|\lambda_{jt}^{[1]}-\lambda_{jt}^{[2]}\right|
\mathbbm{1}_{\mathbb{B}_{[{h}/{2}]:t-1}}
E\left(\mathbbm{1}_{B_t^{[1]}=B_t^{[2]}=2-j}\mid {\cal F}^*_{t-1}\right)\right\}
	\end{eqnarray*}
Since $$P\left(B_t^{[1]}=B_t^{[2]}=1\mid {\cal F}^*_{t-1}\right)\leq E\left(B_t^{[1]}\mid  {\cal F}^*_{t-1}\right) =\pi_t^{[1]}\leq \pi_1^+$$ and $$P\left(B_t^{[1]}=B_t^{[2]}=0\mid {\cal F}^*_{t-1}\right)\leq 1-E\left(B_t^{[1]}\mid  {\cal F}^*_{t-1}\right) =1-\pi_t^{[1]}\leq \pi_0^+,$$ we then have 
\begin{eqnarray}\label{decadix}
E\left|Y^{[1]}_{t}- Y^{[2]}_{t}\right|
\mathbbm{1}_{\mathbb{B}_{[{h}/{2}]:t}}
&\leq&\boldsymbol{\pi}^\top E\skakko{\boldsymbol{d}_{t}
\mathbbm{1}_{\mathbb{B}_{[{h}/{2}]:t-1}}},
\end{eqnarray}
where $\boldsymbol{\pi}=(\pi_1^+,\pi_0^+)^\top$ and 
\begin{equation*}
\boldsymbol{d}_{t}=\left(\left|\lambda_{1t}^{[1]}-
\lambda_{1t}^{[2]}\right|,\left|\lambda_{2t}^{[1]}-
\lambda_{2t}^{[2]}\right|\right)^\top.
\end{equation*}
For $t\geq  [h/2]$, 
\begin{equation*}
E\left|Y^{[1]}_{t}- Y^{[2]}_{t}\right|
\mathbbm{1}_{\mathbb{B}_{[{h}/{2}]:t}}
=E\left||Y^{[1]}_{t}|-
|Y^{[2]}_{t}|\right|
\mathbbm{1}_{\mathbb{B}_{[{h}/{2}]:t}}
\end{equation*}
because $Y^{[1]}_{t}$ and $%
Y^{[2]}_{t}$ have the same sign. In the case $p=q=1$, 
we have 
\begin{eqnarray*}
\lambda_{1,t}^{[1]}-\lambda_{1,t}^{[2]}&=&\alpha_{1}%
\left(|Y_{t-1}^{[1]}|-|Y_{t-1}^{[2]}|\right)+\beta_{1}\left(%
\lambda_{1,t-1}^{[1]}-\lambda_{1,t-1}^{[2]}\right).
\end{eqnarray*}
Therefore, for $t>  [h/2]-1$, using \eqref{decadix} and noting that $\mathbbm{1}_{\mathbb{B}_{[{h}/{2}]:t-1}}\leq \mathbbm{1}_{\mathbb{B}_{[{h}/{2}]:t-2}}$
\begin{align*}
&E\left|\lambda_{1,t}^{[1]}-\lambda_{1,t}^{[2]}\right|
\mathbbm{1}_{\mathbb{B}_{[{h}/{2}]:t-1}}
\leq 
\alpha_{1}
E\left|Y^{[1]}_{t-1}-
Y^{[2]}_{t-1}\right|
\mathbbm{1}_{\mathbb{B}_{[{h}/{2}]:t-1}}
+
\beta_{1}
E\left|\lambda_{1,t-1}^{[1]}-%
\lambda_{1,t-1}^{[2]}\right|
\mathbbm{1}_{\mathbb{B}_{[{h}/{2}]:t-1}}\\
\leq&\;
\alpha_{1}\pi_1^+
E\left|\lambda_{1,t-1}^{[1]}-
\lambda_{1,t-1}^{[2]}\right|
\mathbbm{1}_{\mathbb{B}_{[{h}/{2}]:t-2}}
+
\alpha_{1}\pi_0^+
E\left|\lambda_{2,t-1}^{[1]}-%
\lambda_{2,t-1}^{[2]}\right|
\mathbbm{1}_{\mathbb{B}_{[{h}/{2}]:t-2}}
\\&+\beta_{1}
E\left|\lambda_{1,t-1}^{[1]}-\lambda_{1,t-1}^{[2]}%
\right|
\mathbbm{1}_{\mathbb{B}_{[{h}/{2}]:t-2}}.
\end{align*}
Similarly, for $t> [h/2]-1$, 
\begin{align*}
&E\left|\lambda_{2,t}^{[1]}-\lambda_{2,t}^{[2]}\right|
\mathbbm{1}_{\mathbb{B}_{[{h}/{2}]:t-1}}
\leq
\alpha_{2}\pi_1^+
E\left|\lambda_{1,t-1}^{[1]}-
\lambda_{1,t-1}^{[2]}\right|
\mathbbm{1}_{\mathbb{B}_{[{h}/{2}]:t-2}}
\\&+
\alpha_{2}\pi_0^+
E\left|\lambda_{2,t-1}^{[1]}-%
\lambda_{2,t-1}^{[2]}\right|
\mathbbm{1}_{\mathbb{B}_{[{h}/{2}]:t-2}}
+\beta_{2}E\left|\lambda_{2,t-1}^{[1]}-\lambda_{2,t-1}^{[2]}%
\right|
\mathbbm{1}_{\mathbb{B}_{[{h}/{2}]:t-2}}.
\end{align*}
We thus have, for $t> [h/2]-1$, 
\begin{eqnarray*}
E\boldsymbol{d}_{t}
\mathbbm{1}_{\mathbb{B}_{[{h}/{2}]:t-1}}
\leq 
\boldsymbol{A} E\boldsymbol{d}_{t-1}
\mathbbm{1}_{\mathbb{B}_{[{h}/{2}]:t-2}}
\leq \boldsymbol{A}^{t-[h/2]-1}E\boldsymbol{d}_{[h/2]+1}
\mathbbm{1}_{\mathbb{B}_{[{h}/{2}]:[h/2]}}.
\end{eqnarray*}
By \eqref{moment1}, the vector $E\boldsymbol{d}_{t}$ is finite. It follows that, for  $h> 2$,
\begin{align*}
&\sum_{k=0}^{\infty}E%
\left|Y^{[1]}_{h+k}- Y^{[2]}_{h+k}\right|
\mathbbm{1}_{\mathbb{B}_{h}}
\leq
\sum_{k=0}^{\infty}E%
\left|Y^{[1]}_{h+k}- Y^{[2]}_{h+k}\right|
\mathbbm{1}_{\mathbb{B}_{[{h}/{2}]:h+k}}\\
\leq&\;
\boldsymbol{\pi}^\top
\sum_{k=0}^{\infty}
E%
\bd_{h+k}
\mathbbm{1}_{\mathbb{B}_{[{h}/{2}]:h+k-1}}
\leq
\boldsymbol{\pi}^\top
\boldsymbol A^{h-[h/2]}
\sum_{k=0}^{\infty}
\boldsymbol A^{k-1}
E%
\bd_{[h/2]+1}
\mathbbm{1}_{\mathbb{B}_{[{h}/{2}]:[h/2]}}
\leq
K\varrho^h,
\end{align*}
where $\varrho\in\left(\rho^2(\boldsymbol{A}),1\right)$, and the result follows similarly in the general $(p,q)$ case. $\qquad\square$ %\end{proof}

\subsection{Proof of Theorem \ref{thm:est}}
 Let $L_{n}\left( \btheta \right) $ and $
\ell_{t}\left( \btheta \right) $ as $\widetilde{L}_{n}\left( \btheta
\right) $ and $\widetilde{\ell_{t}}\left( \btheta \right) $ in \eqref{4.7},
replacing $\widetilde{\lambda }_{st}\left( \btheta \right) $ with $\lambda
_{st}\left( \btheta \right) $.
The proof of the consistency result \eqref{4.9} follows from Lemma~\ref{lem4.3} below,  using
standard compactness arguments.

Regarding the asymptotic normality result \eqref{4.10}, note first that
the gradient and Hessian of $\ell_{t}\left( \btheta \right) $ have the form%
\begin{eqnarray*}
\frac{\partial \ell_{t}\left( \btheta \right) }{\partial \bphi } &=&\frac{%
\partial \pi _{t}\left( \bphi \right) }{\partial \bphi }\frac{1}{\pi
_{t}\left( \bphi \right) }\mathbbm{1}_{Y_{t}\geq 0}-\frac{1}{1-\pi
_{t}\left( \bphi \right) }\frac{\partial \pi _{t}\left( \bphi \right) }{%
\partial \bphi }\mathbbm{1}_{Y_{t}<0} \\
\frac{\partial \ell_{t}\left( \btheta \right) }{\partial \bpsi _{1}} &=&\left(
\left( \frac{Y_{t}}{\lambda _{1t}\left( \bpsi _{1}\right) }-1\right) \frac{%
\partial \lambda _{1t}\left( \bpsi _{1}\right) }{\partial \bpsi _{1}}\right)
\mathbbm{1}_{Y_{t}\geq 0} \\
\frac{\partial \ell_{t}\left( \btheta \right) }{\partial \bpsi _{2}} &=&\left(
\left( -\frac{Y_{t}+1}{\lambda _{2t}\left( \bpsi _{2}\right) -1}-1\right) 
\frac{\partial \lambda _{2t}\left( \bpsi _{2}\right) }{\partial \bpsi _{2}}%
\right) \mathbbm{1}_{Y_{t}<0}
\end{eqnarray*}%
and%
\begin{eqnarray*}
\frac{\partial ^{2}\ell_{t}\left( \btheta \right) }{\partial \bphi \partial \bphi
^\top} &=&\left( \frac{1}{\pi _{t}\left( \bphi \right) }\frac{\partial
^{2}\pi _{t}\left( \bphi \right) }{\partial \bphi \partial \bphi ^\top}-%
\frac{1}{\pi _{t}^{2}\left( \bphi \right) }\frac{\partial \pi _{t}\left(
\bphi \right) }{\partial \bphi }\frac{\partial \pi _{t}\left( \bphi \right) }{%
\partial \bphi ^\top}\right) \mathbbm{1}_{Y_{t}\geq 0} \\
&&-\left( \frac{1}{1-\pi _{t}\left( \bphi \right) }\frac{\partial ^{2}\pi
_{t}\left( \bphi \right) }{\partial \bphi \partial \bphi ^\top}+\frac{1}{%
\left( 1-\pi _{t}\left( \bphi \right) \right) ^{2}}\frac{\partial \pi
_{t}\left( \bphi \right) }{\partial \bphi }\frac{\partial \pi _{t}\left( \bphi
\right) }{\partial \bphi ^\top}\right) \mathbbm{1}_{Y_{t}<0} \\
\frac{\partial ^{2}\ell_{t}\left( \btheta \right) }{\partial \bpsi _{1}\partial
\bpsi _{1}^\top} &=&\left( \left( \frac{Y_{t}}{\lambda _{1t}\left( \bpsi
_{1}\right) }-1\right) \frac{\partial ^{2}\lambda _{1t}\left( \bpsi
_{1}\right) }{\partial \bpsi _{1}\partial \bpsi _{1}^\top}-\frac{Y_{t}}{%
\lambda _{1t}^{2}\left( \bpsi _{1}\right) }\frac{\partial \lambda
_{1t}\left( \bpsi _{1}\right) }{\partial \bpsi _{1}}\frac{\partial \lambda
_{1t}\left( \bpsi _{1}\right) }{\partial \bpsi _{1}^\top}\right) \mathbbm{1}_{Y_{t}\geq 0} \\
\frac{\partial ^{2}\ell_{t}\left( \btheta \right) }{\partial \bpsi _{2}\partial
\bpsi _{2}^\top} &=&\left( \left( -\frac{Y_{t}+1}{\lambda _{2t}\left(
\bpsi _{2}\right) -1}-1\right) \frac{\partial^{2} \lambda _{2t}\left( \bpsi
_{2}\right) }{\partial \bpsi _{2}\partial \bpsi _{2}^\top}\right.\\&&\left.-\frac{Y_{t}+1%
}{\left( \lambda _{2t}\left( \bpsi _{2}\right) -1\right) ^{2}}\frac{\partial
\lambda _{2t}\left( \bpsi _{2}\right) }{\partial \bpsi _{2}}\frac{\partial
\lambda _{2t}\left( \bpsi _{2}\right) }{\partial \bpsi _{2}^\top}\right)
\mathbbm{1}_{Y_{t}<0} \\
\frac{\partial ^{2}\ell_{t}\left( \btheta \right) }{\partial \bpsi _{1}\partial
\bpsi _{2}^\top} &=&0\text{, }\quad \frac{\partial ^{2}\ell_{t}\left( \btheta
\right) }{\partial \bphi \partial \bpsi _{s}^\top}=0\text{, }s=1,2.
\end{eqnarray*}
Note that $\frac{\partial \ell_{t}\left( \btheta_0 \right) }{\partial \btheta }=\bDelta_t\bxi_t$.  Taylor's expansion with Lemmas \ref{lem4.4}--\ref{lem4.6} yields the conclusion.

\subsection{Lemmas for Theorem~\ref{thm:est}}

\begin{lem}\label{lem4.1}
\textit{Under \textbf{A1} and \textbf{A4}}%
\begin{equation}
\sup\limits_{\btheta \in \bTheta }\left\vert \widetilde{L}_{n}\left( \btheta
\right) -L_{n}\left( \btheta \right) \right\vert \overset{a.s.}{\underset{%
n\rightarrow \infty }{\rightarrow }}0\text{.}  \label{4.11}
\end{equation}
\end{lem}

\begin{proof}
%Following the same device as in \citet[Remark 2.1]{af16} (see also \citealp{fz19}), 
It is easily seen that 
\eqref{4.2bis} and $\sup\limits_{\bphi \in \bPhi }b<1$ imply
\begin{eqnarray*}
\left\vert Y_{t}+1\right\vert ^{r}\sup\limits_{\bpsi _{s}\in \bPsi
_{s}}\left\vert \widetilde{\lambda }_{st}\left( \bpsi _{s}\right) -\lambda
_{st}\left( \bpsi _{s}\right) \right\vert \overset{a.s.}{\underset{%
t\rightarrow \infty }{\rightarrow }}0, 
& &\sup\limits_{\bphi \in \bPhi }\left\vert \widetilde{\pi }_{t}\left( \bphi
\right) -\pi _{t}\left( \bphi \right) \right\vert \overset{a.s.}{\underset{%
t\rightarrow \infty }{\rightarrow }}0
\end{eqnarray*}%
for $s=1,2$ and $r=0,1$. Now note that
\begin{equation}\label{blue}\underline{\kappa}:=\inf_{\btheta \in \bTheta}\min\{1, c, \omega_1, \omega_2-1\} >0.
\end{equation}
Hence using the inequality $\log (x)\leq x-1$,
we have 
\begin{align*}
&\;\sup\limits_{\btheta \in \bTheta }\left\vert \widetilde{L}_{n}\left( \btheta
\right) -L_{n}\left( \btheta \right) \right\vert  \\
=&\;\frac{1}{n}\sup\limits_{\btheta \in \bTheta }\sum_{t=1}^{n}\left\vert \log
(\frac{\widetilde{\pi }_{t}\left( \bphi \right) }{\pi _{t}\left( \bphi
\right) })+\left( \widetilde{\lambda }_{1t}\left( \bpsi _{1}\right) -\lambda
_{1t}\left( \bpsi _{1}\right) +Y_{t}\log (\frac{\widetilde{\lambda }%
_{1t}\left( \bpsi _{1}\right) }{\lambda _{1t}\left( \bpsi _{1}\right) }%
)\right) \mathbbm{1}_{Y_{t}\geq 0}\right.  \\
&\;\left. +\left( \log (\frac{1-\widetilde{\pi }_{t}\left( \bphi \right) }{%
1-\pi _{t}\left( \bphi \right) })+\lambda _{2t}\left( \bpsi _{2}\right) -%
\widetilde{\lambda }_{2t}\left( \bpsi _{2}\right) -\left( Y_{t}+1\right) \log
(\frac{\lambda _{2t}\left( \bpsi _{2}\right)-1 }{\widetilde{\lambda }%
_{2t}\left( \bpsi _{2}\right) -1})\right) \mathbbm{1}_{Y_{t}<0 }\right\vert 
\\
\leq &\;\frac{1}{n\underline{\kappa}}\sum_{t=1}^{n}\sup\limits_{\btheta \in \bTheta }\left\{\left\vert 
\widetilde{\pi }_{t}\left( \bphi \right) -\pi _{t}\left( \bphi \right)
\right\vert +\left(\left\vert \widetilde{\lambda }%
_{1t}\left( \bpsi _{1}\right) -\lambda _{1t}\left( \bpsi _{1}\right)
\right\vert+\left\vert \widetilde{\lambda }%
_{2t}\left( \bpsi _{2}\right) -\lambda _{2t}\left( \bpsi _{2}\right)
\right\vert\right)\left(1+|Y_t|\right) \right\}
\end{align*}%
so the result follows from C\'{e}saro's lemma.
\end{proof}

\begin{lem}\label{rousse}
Under \textbf{A1}-\textbf{A4}, $E\left(
\ell_{t}\left( \btheta _{0}\right) \right) <\infty $\textit{, }$E\left(
\ell_{t}\left( \btheta \right) \right) $\textit{\ is maximized at }$\btheta
=\btheta _{0}$\textit{, and }$E\left( \ell_{t}\left( \btheta _{0}\right) \right)
=E\left( \ell_{t}\left( \btheta \right) \right) \Rightarrow \btheta =\btheta _{0}$.
\end{lem}

\begin{proof}
In view of \eqref{moment1}, \eqref{4.2} and \eqref{blue},  $\left\vert \log \left( \lambda
_{1t}\left( \bpsi _{1}\right) \right) \right\vert $ and $\left\vert \log \left( \lambda
_{2t}\left( \bpsi _{2}\right)-1 \right) \right\vert $  admit moments of any
order, for all $\bpsi \in \bPsi$. By {\bf A2} and H\"{o}lder's inequality, we have 
$$E\left|Y_t\log \lambda
_{1t}\left( \bpsi _{1}\right)\mathbbm{1}_{
Y_{t}\geq 0}\right|\leq \left\|X_{1t}\right\|_{\tau}\left\|\log \lambda
_{1t}\left( \bpsi _{1}\right)\right\|_{\tau/(\tau-1)}<\infty$$ and
$$E\left|\left(Y_t+1\right)\log \left(\lambda
_{2t}\left( \bpsi _{2}\right)-1\right)\mathbbm{1}_{
Y_{t}< 0}\right|\leq \left\|X_{2t}\right\|_{\tau}\left\|\log \left(\lambda
_{2t}\left( \bpsi _{2}\right)-1\right)\right\|_{\tau/(\tau-1)}<\infty.$$
Moreover,  $\lambda
_{1t}\left( \bpsi _{1}\right)$, $\lambda
_{2t}\left( \bpsi _{2}\right)$, $\log \left( \pi _{t}\left(
\bphi\right) \right) $ and $\log \left( 1-\pi _{t}\left( \bphi
\right) \right) $ admit a finite moment (the latter two actually admit moments of any order). It follows that 
\begin{eqnarray}
\left\vert E\left( \ell_{t}\left( \btheta\right) \right) \right\vert  &\leq
&E\left( \left\vert \ell_{t}\left( \btheta\right) \right\vert \right)  
 <\infty.  \label{4.12}
\end{eqnarray}%
Now using the inequality $\log \left( x\right) \leq x-1$ and the facts
that 
\begin{equation}
\begin{array}{l}
E\left( \mathbbm{1}_{Y_{t}\geq 0 }\mid \mathcal{F}_{t-1}\right) =\pi
_{t}\left( \bphi _{0}\right) ,\text{ \ \ \ \ \ }E\left( \mathbbm{1}_{Y_{t}<0 }|\mathcal{F}_{t-1}\right) =1-\pi _{t}\left( \bphi _{0}\right)  \\ 
E\left( Y_{t}\mathbbm{1}_{Y_{t}\geq 0}\mid \mathcal{F}_{t-1}\right)
=\pi _{t}\left( \bphi _{0}\right) \lambda _{1t}\left( \bpsi _{01}\right)  \\ 
E\left( \left( Y_{t}+1\right) \mathbbm{1}_{Y_{t}<0}\mid \mathcal{F}%
_{t-1}\right) =-\left( 1-\pi _{t}\left( \bphi _{0}\right) \right) \left(
\lambda _{2t}\left( \bpsi _{02}\right) -1\right) 
\end{array}
\label{4.13}
\end{equation}%
we obtain%
\begin{eqnarray}
&&\left. E\left( \ell_{t}\left( \btheta \right) -\ell_{t}\left( \btheta _{0}\right)
\right) \right.   \notag \\
&&\left. \leq E\left( \left( \frac{\pi _{t}\left( \bphi \right) -\pi
_{t}\left( \bphi _{0}\right) }{\pi _{t}\left( \bphi _{0}\right) }-\lambda
_{1t}\left( \bpsi _{1}\right) +\lambda _{1t}\left( \bpsi _{01}\right) +Y_{t}%
\frac{\lambda _{1t}\left( \bpsi _{1}\right) -\lambda _{1t}\left( \bpsi
_{01}\right) }{\lambda _{1t}\left( \bpsi _{01}\right) }\right) \mathbbm{1}_{
Y_{t}\geq 0}\right) \right.   \notag \\
&&\left. 
+ E\left( \left( \frac{\pi _{t}\left( \bphi _{0}\right) -\pi
_{t}\left( \bphi \right) }{1-\pi _{t}\left( \bphi _{0}\right) }-\lambda
_{2t}\left( \bpsi _{2}\right) +\lambda _{2t}\left( \bpsi _{02}\right) -\left(
Y_{t}+1\right) \frac{\lambda _{2t}\left( \bpsi _{2}\right) -\lambda
_{2t}\left( \bpsi _{02}\right) }{\lambda _{2t}\left( \bpsi _{02}\right) -1}%
\right) \mathbbm{1}_{Y_{t}<0}\right) \right.   \notag \\
&&\left. =E\left( \pi _{t}\left( \bphi _{0}\right) \left( \lambda _{1t}\left(
\bpsi _{01}\right) -\lambda _{1t}\left( \bpsi _{1}\right) \right) \right)
+E\left( \pi _{t}\left( \bphi _{0}\right) \left( \lambda _{1t}\left( \bpsi
_{1}\right) -\lambda _{1t}\left( \bpsi _{01}\right) \right) \right) \right.  
\notag \\
&&\left. +E\left( \left( 1-\pi _{t}\left( \bphi _{0}\right) \right) \left(
\lambda _{2t}\left( \bpsi _{02}\right) -\lambda _{2t}\left( \bpsi _{2}\right)
\right) \right) +E\left( \left( 1-\pi _{t}\left( \bphi _{0}\right) \right)
\left( \lambda _{2t}\left( \bpsi _{2}\right) -\lambda _{2t}\left( \bpsi
_{02}\right) \right) \right) \right.   \notag \\
&&\left. =0\text{.}\right.   \label{4.14}
\end{eqnarray}%
Moreover, the inequality in \eqref{4.14} reduces to an equality iff%
\begin{eqnarray}
&&E\left( \left( \log \frac{\pi _{t}\left( \bphi \right) }{\pi _{t}\left(
\bphi _{0}\right) }-\lambda _{1t}\left( \bpsi _{1}\right) +\lambda _{1t}\left(
\bpsi _{01}\right) +Y_{t}\log \left( \frac{\lambda _{1t}\left( \bpsi
_{1}\right) }{\lambda _{1t}\left( \bpsi _{01}\right) }\right) \right) \mathbbm{1}_{Y_{t}\geq 0}\right) +  \notag \\
&&\left. E\left( \left( \log \frac{1-\pi _{t}\left( \bphi \right) }{1-\pi
_{t}\left( \bphi _{0}\right) }-\lambda _{2t}\left( \bpsi _{2}\right) +\lambda
_{2t}\left( \bpsi _{02}\right) -\left( Y_{t}+1\right) \log \left( \frac{%
\lambda _{2t}\left( \bpsi _{2}\right) -1}{\lambda _{2t}\left( \bpsi
_{02}\right) -1}\right) \right) \mathbbm{1}_{Y_{t}<0}\right) =0\text{,}%
\right.   \notag
\end{eqnarray}%
which holds iff almost surely $\pi _{t}\left( \bphi \right) =\pi _{t}\left( \bphi
_{0}\right) $, $\lambda _{1t}\left( \bpsi _{1}\right) =\lambda _{1t}\left(
\bpsi _{01}\right) $, and $\lambda _{2t}\left( \bpsi _{2}\right) =\lambda
_{2t}\left( \bpsi _{02}\right)$. By \textbf{A3(i)} and standard arguments (see {\em e.g.} (7.32) in \cite{fz19}) the last two equalities entail $\bpsi _{1}=\bpsi _{01}$, and $%
\bpsi _{2}=\bpsi _{02}$.
Now $\pi _{t}\left( \bphi \right) =\pi _{t}\left( \bphi
_{0}\right) $ with probability 1 (and for all $t$, by stationarity) entails
$$c-c_0+(a-a_0)B_{t}+(b-b_0)\pi_{t}(\bphi_0)=0.$$
Since, by \eqref{fille} and  \textbf{A3(ii)},  $B_t$ and $\pi_t(\bphi_0)$ are not degenerated, the previous equality entails $c=c_0$, $a=a_0$ and $b=b_0$.
Therefore  $\pi _{t}\left( \bphi \right) =\pi _{t}\left( \bphi
_{0}\right) $ a.s. iff $\bphi =\bphi _{0}$, which concludes. 
\end{proof}

\begin{lem}\label{lem4.3} Assume \textbf{A1}-\textbf{A4}.
For any $\btheta \neq \btheta _{0}$, \textit{there
is a neighborhood }$\mathcal{V}\left( \btheta \right) $ of $\btheta$\textit{\ such that}%
\begin{equation*}
\underset{n\rightarrow \infty }{\lim \sup }\sup\limits_{\overline{\btheta }%
\in \mathcal{V}\left( \btheta \right) }\widetilde{L}_{n}\left( \overline{%
\btheta }\right) <\text{ }\underset{n\rightarrow \infty }{\lim \sup }\;
\widetilde{L}_{n}\left( \btheta _{0}\right) \text{,\ }a.s.
\end{equation*}
\end{lem}

\begin{proof}
Let $V_{k}(\overline{\btheta })$ ($k\in 
\mathbb{N}
^{\ast }$, $\overline{\btheta }\in \bTheta $) be the open ball of center $%
\overline{\btheta }$ and radius $\frac{1}{k}$. Since $\sup_{\btheta \in V_{k}(%
\overline{\btheta })\cap \bTheta }\ell_{t}\left( \btheta \right) $ is a measurable
function of the stationary and ergodic process $\left\{ Y_{t},t\in \mathbb{Z}\right\} $, the process $\left\{
\sup_{\btheta \in V_{k}(\overline{\btheta })\cap \bTheta }\ell_{t}\left( \btheta
\right) ,t\in \mathbb{Z}\right\} $ is also strictly stationary and ergodic
and satisfies $E\sup_{\btheta \in V_{k}(\overline{\btheta })\cap \bTheta
}\left|\ell_{t}\left( \btheta \right) \right|<\infty$ by the arguments used to show \eqref{4.12}. Hence, by \eqref{4.11}
\begin{equation*}
\underset{n\rightarrow \infty }{\lim \sup }\sup_{\btheta \in V_{k}(\overline{%
\btheta })\cap \bTheta }\widetilde{L}_{n}\left( \btheta \right) =\text{ }%
\underset{n\rightarrow \infty }{\lim \sup }\sup_{\btheta \in V_{k}(\overline{%
\btheta })\cap \bTheta }L_{n}\left( \btheta \right) \leq E\left( \sup_{\btheta
\in V_{k}(\overline{\btheta })\cap \bTheta }\ell_{t}\left( \btheta \right) \right) 
\text{.}
\end{equation*}

By the monotone convergence theorem, $E\left( \sup\limits_{\btheta \in V_{k}(%
\overline{\btheta })\cap \bTheta }\ell_{t}\left( \btheta \right) \right) $
decreases to $E\left( \ell_{t}\left( \overline{\btheta }\right) \right) $ as $%
k\rightarrow \infty $ and the results follows from Lemma \ref{rousse}.
\end{proof}

\begin{lem}\label{lem4.4}
\textit{Under} \textit{\textbf{A1} and \textbf{A4}}%
\begin{equation*}
n^{\frac{1}{2}}\sup_{\btheta \in \bTheta }\left\Vert \frac{\partial 
\widetilde{L}_{n}\left( \btheta \right) }{\partial \btheta }-\frac{\partial
L_{n}\left( \btheta \right) }{\partial \btheta }\right\Vert \underset{%
n\rightarrow \infty }{\overset{a.s.}{\rightarrow }}0.
\end{equation*}
\end{lem}

\begin{proof}
The result can be proved in the same way as in \cite{af16}.
\end{proof}

\begin{lem}
Under \textbf{A1}-\textbf{A6} the matrices $\bPi$, $\bJ_1$, $\bJ_2$,  $\bI_{1}$, and $\bI_{2}$ exist and are invertible.
\end{lem}

\begin{proof} Fist note that
$$\frac{\partial \pi_t(\bphi)}{\partial \bphi}=\left(\begin{array}{c}1\\B_{t-1}\\\pi_{t-1}\end{array}\right)+b \frac{\partial \pi_{t-1}(\bphi)}{\partial \bphi}$$ is well defined and is bounded, thus integrable. It follows that $\bPi$ is well defined. Let us argue by contradiction by assuming that $\bPi$ is not invertible. Then there exists $\lambda=(\lambda_1,\lambda_2,\lambda_3)^\top\neq 0$ such that $\lambda^\top\frac{\partial \pi_t(\bphi_0)}{\partial \bphi}=0$ a.s. By stationarity, this implies $\lambda_1+\lambda_2B_t+\lambda_3\pi_t=0$ a.s. Because $\pi_t\in (0,1)$, $B_t$ is not ${\cal F}_{t-1}-$measurable and we necessarily have $\lambda_2=0$. Because $\pi_t$ is not constant under {\bf A3(ii)}, we also have $\lambda_3=0$, and finally $\lambda_1=0$, which is not possible. By contradiction, we thus have shown that $\bPi$ is invertible.  

Now we have
$$\frac{\partial \lambda_{st}(\bpsi_s)}{\partial \bpsi_s}=Z_{st}(\bpsi_s)
+\sum_{j=1}^p\beta_{sj} \frac{\partial \lambda_{s,t-j}(\bpsi_s)}{\partial \bpsi_s},$$ with $$Z_{st}(\bpsi_s)=\left(\begin{array}{ccccccc}1&|Y_{t-1}|&\cdots&|Y_{t-q}|&\lambda_{s,t-1}(\bpsi_s)&\cdots&\lambda_{s,t-p}(\bpsi_s)\end{array}\right)^\top.$$ 
Note that, in the ratio $\frac{\partial \lambda_{st}(\bpsi_s)}{\partial \bpsi_s}/\lambda_{st}(\bpsi_s)$, the random variables that appear in the numerator are also present in the denominator, under {\bf A5}. By the arguments used to show (7.54) in \cite{fz19}, we thus have
\begin{equation}\label{fillette}
E\sup_{\bpsi_{s}\in V(\bpsi_{0s})}\left\|\frac{1}{\lambda_{st}(\bpsi_s)}\frac{\partial \lambda_{st}(\bpsi_s)}{\partial \bpsi_s}\right\|^d+\left\|\frac{1}{\lambda_{st}(\bpsi_s)}\frac{\partial^2 \lambda_{st}(\bpsi_s)}{\partial \bpsi_s\partial \bpsi_s^\top}\right\|^d<\infty
\end{equation}
for any integer $d$ and some neighborhood $V(\bpsi_{0s})$ of $\bpsi_{0s}$.
Using also \eqref{blue}, it follows that $\bJ_1$ and $\bJ_2$ are well defined. By {\bf A6}, $\bI_1$ and $\bI_2$ are also well defined. If $\bJ_s$ or $\bI_s$ is not invertible, then there exists $\lambda=(\lambda_1,\cdots,\lambda_{p+q+1})^\top\neq 0$ such that $\lambda^\top Z_{st}(\bpsi_{0s})=0$ a.s.  Because $|Y_{t-1}|$ is not a measurable function of $\lambda_{s,t-1},\cdots,\lambda_{s,t-p}$, we have $\lambda_2=0$. If $\lambda_{q+1}\neq 0$, then $\lambda_{st}$ follows a relation of the form \eqref{3.2b} with $p$ and $q$ replaced by $p-1$ and $q-1$, which is impossible under {\bf A3(i)}.
Continuing in this way, we show that $\lambda=0$ and conclude by contradiction that $\bJ_s$ and $\bI_s$ are invertible.
\end{proof}
\begin{lem}
\textit{Under} \textit{\textbf{A1}-\textbf{A6}}%
\begin{equation*}
n^{\frac{1}{2}}\frac{\partial L_{n}\left( \btheta _{0}\right) }{\partial
\btheta }\underset{n\rightarrow \infty }{\overset{D}{\rightarrow }}\mathcal{N}%
\left( 0,\bI\right) 
\end{equation*}%
\textit{where} $\bI=diag(\bPi ,\bI_{1},\bI_{2})$.
\end{lem}

\begin{proof}
Note that the sequence $\left\{ n^{\frac{1}{2}}\frac{%
\partial L_{n}\left( \btheta _{0}\right) }{\partial \btheta },t\in \mathbb{Z}%
\right\} $ is a square integrable $\mathcal{F}_{t}$-martingale with 
\begin{equation*}
n^{\frac{1}{2}}\frac{\partial L_{n}\left( \btheta _{0}\right) }{\partial
\btheta }=n^{-\frac{1}{2}}\sum_{t=1}^{n}\frac{\partial \ell_{t}\left( \btheta
_{0}\right) }{\partial \btheta }
\end{equation*}%
and $E\left( \frac{\partial \ell_{t}\left( \btheta _{0}\right) }{\partial
\btheta }\frac{\partial \ell_{t}\left( \btheta _{0}\right) }{\partial \btheta
^\top}\right) =diag(\bPi ,\bI_{1},\bI_{2})$. Hence, the result follows from
the central limit theorem of \cite{billingsley1961lindeberg} for square-integrable martingales.
\end{proof}

\begin{lem}\label{lem4.6}
\textit{Under} \textit{\textbf{A1}-\textbf{A5}}, if $\btheta_n\underset{n\rightarrow \infty }{\overset{a.s.}{%
\rightarrow }}\btheta_0$ then
\begin{equation*}
\frac{\partial ^{2}L_{n}\left( \btheta_n \right) }{\partial \btheta
\partial \btheta ^\top}\underset{n\rightarrow \infty }{\overset{a.s.}{%
\rightarrow }}-\bJ.
\end{equation*}%
\end{lem}

\begin{proof}
With notations of the proof of Lemma \ref{lem4.3}, in view of the stationarity and ergodicity of the
sequences 
\begin{equation*}
\left\{ \frac{\partial^{2}\ell_t\left( \btheta _{0}\right) }{%
\partial \theta _{i}\partial \theta _{j}} \right\}_t 
\quad\mbox{ and }\quad
\left\{ \sup_{\btheta \in V_{k}(\btheta _{0})}\left\vert \frac{%
\partial^{2} \ell_t\left( \btheta \right) }{\partial \theta _{i}\partial
\theta _{j}}-\frac{\partial^{2}\ell_t\left( \btheta _{0}\right) }{%
\partial \theta_{i}\partial \theta_{j}} \right\vert \right\}_t, 
\end{equation*}%
and the consistency of $\btheta_n$,  we have almost surely
\begin{equation*}
\lim_{n\to\infty}\left\vert \bJ\left( i,j\right) -\frac{\partial ^{2}L_{n}\left( \btheta_n\right) }{\partial \theta _{i}\partial \theta _{j}}\right\vert\leq E\sup_{\btheta \in 
V_{k}(\btheta _{0})}\left\vert \frac{\partial ^{2}\ell_{t}\left(
\btheta \right) }{\partial \theta _{i}\partial \theta _{j}}-\frac{%
\partial ^{2}\ell_{t}\left( \btheta _{0}\right) }{\partial \theta _{i}\partial
\theta _{j}}\right\vert \text{,}
\end{equation*}%
for all $k$. Using {\bf A2}, {\bf A6} and \eqref{fillette}, by the Hölder inequality we have $$E\sup_{\btheta \in 
V_{k}(\btheta _{0})}\left|\frac{\partial ^{2}\ell_{t}\left(
\btheta \right) }{\partial \theta _{i}\partial \theta _{j}}\right|<\infty.$$
The dominated convergence theorem then entails 
\begin{eqnarray*}
\lim_{k\rightarrow \infty }E\left( \sup_{\btheta \in V_{k}(\btheta _{0})}\left\vert \frac{\partial ^{2}\ell_{t}\left( \btheta \right) }{\partial
\theta _{i}\partial \theta _{j}}- \frac{\partial ^{2}\ell_{t}\left(
\btheta _{0}\right) }{\partial \theta _{i}\partial \theta _{j}}
\right\vert \right)  &=&E\left( \lim_{k\rightarrow \infty }\sup_{\btheta \in 
V_{k}(\btheta _{0})}\left\vert \frac{\partial ^{2}\ell_{t}\left(
\btheta \right) }{\partial \theta _{i}\partial \theta _{j}}- \frac{%
\partial ^{2}\ell_{t}\left( \btheta _{0}\right) }{\partial \theta _{i}\partial
\theta _{j}} \right\vert \right)  =0,
\end{eqnarray*}%
establishing the result.
\end{proof}

\begin{lem}\label{lemchaud}
If  $B_{t}$ and $X_{st}$ ($s=1,2$) are non degenerated and are conditionally
independent given $\mathcal{F}_{t-1}$,
the random variables $$\epsilon_t=\left\{X_{1t}-\lambda _{1t}\right\}B_t-\left\{X_{2t} - \lambda _{2t}\right\}(1-B_t)$$ and  $\bxi_t=(\xi_{1t},\xi_{2t},\xi_{3t})^\top$, defined in Theorem \ref{thm:est}, satisfy: 
\begin{equation}\label{demoncoeur}c_1\epsilon_t+\sum_{i=1}^3c_{i+1} \xi_{it}=0\;\mbox{ a.s.}\quad\Rightarrow\quad c_1=-c_3=c_4 \mbox{ and }c_2=0.
\end{equation} 
\end{lem}

\begin{proof}Since $B_t$ is not ${\cal F}_{t-1}$-measurable, the left-hand side of \eqref{demoncoeur} entails
$$c_1 X_{1t}-c_1\lambda _{1t}+c_2-c_2\pi_t+c_3X_{1t}-c_3\lambda_{1t}=0\;\mbox{ a.s.}$$ and
$$-c_1 X_{2t} + c_1\lambda _{2t}-c_2\pi_t+c_4X_{2t} -c_4\lambda_{2t}=0\;\mbox{ a.s.}$$
Since $X_{1t}$ and  $X_{2t}$ are not ${\cal F}_{t-1}$-measurable, it follows that $c_1=-c_3=c_4$, and
$$(c_1+c_3)\lambda _{1t}=c_2(1-\pi_t),\qquad (c_1-c_4)\lambda _{2t}=c_2\pi_t\;\mbox{ a.s.}$$
Since $\lambda_{st}$ is not $\mathcal{F}_{t-1}^{B}$-measurable this entails 
$c_1=-c_3$, $c_2=0$ and $c_1=c_4$. 
\end{proof}

\subsection{Proof of Theorem~\ref{thm:gof}}
Assume {\bf A1}-{\bf A7}.
First, we show the negligibility of initial values. Let
$$\widetilde\gamma_{h}(\btheta)=n^{-1}\sum_{t=1}^n\widetilde\epsilon_t(\btheta)\widetilde\epsilon_{t-h}(\btheta)\quad\mbox{ and }\quad\gamma_{h}(\btheta)=n^{-1}\sum_{t=1}^n\epsilon_t(\btheta)\epsilon_{t-h}(\btheta)$$ 
for $0\leq h<n$. 
By the arguments used to show (7.30) in \cite{fz19}, \eqref{4.2bis} entails the existence of a ${\cal F}_0$-measurable positive variable $K$ and a constant $\varrho\in [0,1)$ such that
\begin{equation}\label{etcruelle}
\sup_{\bpsi_{s}\in \bPsi_s}\left|\widetilde\lambda_{st}(\bpsi_s)-\lambda_{st}(\bpsi_s)\right|<K\varrho^t,\qquad \sup_{\btheta\in \bTheta}\left|\widetilde\epsilon_t(\btheta)-\epsilon_t(\btheta)\right|<K\varrho^t.
\end{equation}
Note also that {\bf A1}, {\bf A4} and {\bf A7} entail 
\begin{equation}\label{grace}
E \sup_{\btheta\in \bTheta}\left|\epsilon_t(\btheta)\right|^4<\infty.
\end{equation}
Similarly to Lemma~\ref{lem4.4}
we then have
\begin{equation}\label{paire}
\sqrt{n}\sup_{\btheta \in \bTheta }\left| \widetilde\gamma_{h}(\btheta)-\gamma_{h}(\btheta)\right| \underset{%
n\rightarrow \infty }{\overset{a.s.}{\rightarrow }}0.
\end{equation}

Using also the ergodic theorem, it follows that under $H_0$
$$\widehat\gamma_h \underset{%
n\rightarrow \infty }{\overset{a.s.}{\rightarrow }}\left\{\begin{array}{lll}E\epsilon_t^2&\mbox{if}&h=0\\ 0&\mbox{if}&h\neq 0.\end{array}\right.$$
To establish the asymptotic distribution of the test under $H_0$, it thus remains to  show  that
\begin{align}
\sqrt n 
\widehat\bgamma_{1:k}^\top
\underset{n\to\infty}{\overset{D}{\to}}{\cal N}\left(\bzero,(E\epsilon_t^2)^2\bV_0\right),\qquad 
\widehat{\bV}\overset{a.s.}{\underset{n\rightarrow \infty }{%
\rightarrow }} \bV_0\label{deCadix}
\end{align}
where $\widehat\bgamma_{1:k}=(\widehat\gamma_{1},\dots, \widehat\gamma_{k})^\top$.
By \eqref{paire} and a Taylor expansion,  we have
\begin{align*}
\sqrt n \widehat\gamma_{h}+o(1)=\sqrt{n}\gamma_{h}(\widehat{\btheta}_n)
=
\sqrt n \gamma_{h}(\btheta_0)
+
\frac{\partial}{\partial {\btheta}^\top} \gamma_{h}(\btheta)\Big|_{{\btheta}={\btheta}^*}
\sqrt n \skakko{\widehat{\btheta}_n-{\btheta}_0},
\end{align*}
where  ${\btheta}^*$ is between $\widehat{\btheta}_n$ and ${\btheta}_0$. 
It holds that
\begin{align*}
\frac{\partial}{\partial {\btheta}}\gamma_{h}(\btheta^*)
\to
\bm d_h
\quad \mbox{a.s.\ as $n\to\infty$},
\end{align*}
where 
$$\bd_h:=E\left\{\epsilon_t(\btheta_0)\frac{\partial}{\partial \btheta}\epsilon_{t + h}(\btheta_0)\right\}
\mbox{ is the }h\mbox{-th line of } \bD,
$$
by 
the strong consistency of $\widehat{\btheta}_n$, the ergodic theorem, Beppo-Levi's theorem, and the fact that 
$E\epsilon_{t}\frac{\partial}{\partial \btheta}
\epsilon_{t-h}(\btheta_0)=0$. 
The existence  of  $\bD$ and $\bE$ is guaranteed by {\bf A7}.
Let $\bgamma_{1:k}=\left(\gamma_{1}(\btheta_0),\dots,\gamma_{k}(\btheta_0)\right)^\top$.
Now, note that the central limit theorem for square integrable martingale differences and the derivations of the proof of Theorem~\ref{thm:est} entail
$$\sqrt{n}\left(\begin{array}{c}\bgamma_{1:k}\\\widehat\btheta_n-\btheta_0\end{array}\right)=\frac{1}{\sqrt n}\sum_{t=1}^n \left(\begin{array}{c}\epsilon_t\bepsilon_{t-1:t-k}\\\bJ^{-1}\frac{\partial}{\partial \btheta}\ell_t(\btheta_0)\end{array}\right)\underset{n\rightarrow \infty }{\overset{D}{\rightarrow }}\mathcal{N}%
\left\{ \bzero,\left(\begin{array}{cc}\bE&\bC\bJ^{-1}\\
\bJ^{-1}\bC^\top&\bSigma\end{array}\right)\right\} .$$
The first convergence in \eqref{deCadix} follows.
By the arguments used to show \eqref{paire} and Lemma~\ref{lem4.6}, we have $\widehat \bE\underset{n\rightarrow \infty }{\overset{a.s.}{%
\rightarrow }}\bE$ if $E\sup_{\btheta\in V(\btheta_0)}|\epsilon_t(\btheta)|^4<\infty$. This is entailed by {\bf A7} and the fact that, under \eqref{4.2bis}, there exist constants $K>0$ and $\varrho\in (0,1)$ such that $\lambda_{st}(\bpsi_s)\leq K\left(1+\sum_{i=1}^{\infty}\varrho^i|Y_{t-i}|\right)$ uniformly in $\bTheta$. The consistency of the other empirical estimators involved in $\widehat\bV$ is shown similarly. The second convergence in \eqref{deCadix} follows.

Now we show the invertibility of $\bV_0$.
Note that $\bV_0$  is the variance of 
$$\bv_t=\epsilon_t\bepsilon_{t-1:t-k}+\bD\bJ^{-1}\bDelta_t\bxi_t.$$
Let us argue by contradiction by assuming that $\bV_0$ is not invertible. Then there exists $\bmu=(\mu_1,\cdots,\mu_k)^\top\neq \bzero$ such that $\bmu^\top\bv_t=0$ a.s.
By Lemma \ref{lemchaud}, this entails that 
\begin{equation}\label{grosse}
\bmu^\top\bepsilon_{t-1:t-k}=-\bmu^\top\bD\bJ^{-1}\bP_2 \frac{1}{\lambda_{1t}}\frac{\partial \lambda_{1t}}{\partial\bpsi_1}=\bmu^\top\bD\bJ^{-1}\bP_3 \frac{1}{\lambda_{2t}-1}\frac{\partial \lambda_{2t}}{\partial\bpsi_1},\;\mbox{ a.s.}
\end{equation} and
$$\bmu^\top\bD\bJ^{-1}\bP_1 \frac{1}{\pi_t(1-\pi_t)}\frac{\partial \pi_{1}}{\partial\bphi}=\bzero,\;\mbox{ a.s.}$$
where $\bP_1$ is the $d\times 3$ matrix obtained by stacking the $3\times 3$ identity matrix and the $(d-3)\times 3$ zero matrix,  $\bP_2$ is the $d\times(p+q+1)$ matrix obtained by stacking the $3\times (p+q+1)$ zero matrix, the $(p+q+1)\times (p+q+1)$ identity matrix the $(p+q+1)\times (p+q+1)$ zero matrix, and $\bP_3$ is the $d\times(p+q+1)$ matrix obtained by stacking the $(d-p-q-1)\times (p+q+1)$ zero matrix  and the $(p+q+1)\times (p+q+1)$ identity matrix.

Note that $\epsilon_{t-1}=\left\{X_{1,t-1}-\lambda_{1,t-1}+X_{2,t-1}-\lambda_{2,t-1}\right\}B_{t-1}-X_{2,t-1}+\lambda_{2,t-1}$ and that
$\lambda_{st}$ and $\lambda_{st}^{-1}{\partial \lambda_{st}}/{\partial\bpsi_1}$  are measurable functions of (the sigma-field generated by) $\{|Y_i|, i<t\}$.
If $\mu_1\neq 0$, the first equality of \eqref{grosse} and {\bf A8} entail that $B_{t-1}$, {\em i.e.} the sign of $Y_{t-1}$, is a measurable function of $\{|Y_i|, i<t\}$ and 
$\{X_{s,t-1},\lambda_{s,t-1}, s=1,2\}$, which is not the case. Thus $\mu_1=0$. Similarly, we show that all the $\mu_i$'s are zero, hence the contradiction. 

The consistency is due to:
$n
\widehat\brho_{1:k}^\top
\widehat\bV^{-1}
\widehat\brho_{1:k}\to\infty$ a.s. under $H_1$.

\subsection{Proof of Theorem~\ref{thm:gof2}}

Let $E^*$, $\mbox{Var}^*$, $o_{a.s.}^*(1)$, $o_p^*(1)$, $O_p^*(1)$ be the expectation, variance, convergence to zero almost surely, convergence to zero in probability and bounded in probability conditional on $\{Y_t\}$.
For example, we have
$$E^*\widetilde\gamma_{h}^*\skakko{{\btheta}}:=E\left(\widetilde\gamma_{h}^*\skakko{{\btheta}}\mid \{Y_t\}\right)=\widetilde\gamma_{h}\skakko{{\btheta}},\qquad \mbox{Var}^*\left\{\widetilde\gamma_{h}^*\skakko{{\btheta}}\right\}=\frac{1}{n^2}\sum_{t=h+1}^n\widetilde\epsilon^2_t(\btheta)\widetilde\epsilon^2_{t-h}(\btheta).$$
We first state two elementary lemmas, whose proofs are provided for completeness.
\begin{lem}\label{lembrulant}
Let $(\bd_{t,n})$ and   $(\bs_{t,n})$ be two triangular arrays of real vectors and let $\bd$ be a vector,  such that  $$\lim_{n\to\infty}n^{-1}\sum_{t=1}^n\bd_{t,n}=\bd,\quad \limsup_{n\to\infty}n^{-1}\sum_{t=1}^n\bd_{t,n}^\top\bd_{t,n}<\infty\quad \mbox{ and }\quad \lim_{n\to\infty}\left\|n^{-1}\sum_{t=1}^n\bs_{t,n}\bs_{t,n}^\top\right\|=0.$$ We have
$$\frac{1}{n}\sum_{t=1}^nw_t^*\bd_{t,n}=\bd+o^*_P(1), \qquad \frac{1}{\sqrt{n}}\sum_{t=1}^n(w_t^*-1)\bs_{t,n}=o^*_P(1).$$
\end{lem}

\begin{proof}
The first result follows from
$$E^*\frac{1}{n}\sum_{t=1}^n(w_t^*-1)\bd_{t,n}=0,\qquad \mbox{Var}^*\frac{1}{n}\sum_{t=1}^n(w_t^*-1)\bd_{t,n}=\frac{1}{n^2}\sum_{t=1}^n\bd_{t,n}\bd_{t,n}^\top$$ and the second is obtained similarly.
\end{proof}

\begin{lem}\label{lemfroid}
Let $(\bs_{t})$ be a sequence of real vectors  such that $\lim_{n\to\infty}n^{-1}\sum_{t=1}^n\bs_t\bs_t^\top=\bS$ for some non-singular matrix $\bS$. We have
$$\frac{1}{\sqrt{n}}\sum_{t=1}^n(w_t^*-1)\bs_t\underset{n\rightarrow \infty }{\overset{D}{\rightarrow }}\mathcal{N}%
\left( 0,\bS\right). $$
\end{lem}

\begin{proof} Let $\blambda$ be an arbitrary non-zero vector of same size as $\bs_t$. 
By the Cramér-Wold device and Lindeberg's CLT for triangular arrays of independent and centered variables, the result follows by noting that 
\begin{equation*}
	\frac{1}{n}\sum_{t=1}^n\mbox{Var}^*\left\{(w_t^*-1)\blambda^{\top}\bs_t\right\}\rightarrow  \blambda^{\top}\bS\blambda>0 \quad\mbox{ as $n\to\infty$,}
\end{equation*}
and by showing that for all $\varepsilon>0$
\begin{equation}
	\label{CLTLindeberdCondition}
	\frac{1}{n}\sum_{t=1}^n E^*\left((w_t^*-1)^2\left\{\blambda^{\top}\bs_t\right\}^2 \mathbbm{1}_{\{|w_t^*-1||\blambda^{\top}\bs_t|\geq \sqrt{n}\varepsilon\}}\right)\rightarrow  0\quad\mbox{ as $n\to\infty$.}
\end{equation}
When $\blambda^{\top}\bs_t\neq 0$ we have
\begin{align*}
E^*\left((w_t^*-1)^2\left\{\blambda^{\top}\bs_t\right\}^2 \mathbbm{1}_{\{|w_t^*-1||\blambda^{\top}\bs_t|\geq \sqrt{n}\varepsilon\}}\right)
	=& \left\{\blambda^{\top}\bs_t\right\}^2
	E^*\left(\left|w^*_t-1\right|^2\mathbbm{1}_{\left\{\left|w^*_t-1\right|\geq \frac{\sqrt{n}\varepsilon}{\left|\blambda^{\top}\bs_t\right|}\right\}}\right).
\end{align*}
For any $A>0$, there exists $n_A$ such that if $n>n_A$, then the expectation on the right-hand side of the previous equality is bounded by
$$\int_{\left|\omega-1\right|\geq A} \left|\omega-1\right|^2dP_{w_t^*}(\omega),$$
which is arbitrarily small when $A$ is sufficiently large.
We then obtain  \eqref{CLTLindeberdCondition} by the Ces\`aro mean theorem. 
\end{proof}

We now come back to the proof of Theorem~\ref{thm:gof2}.
Let
$$\widetilde\gamma_{h}^*(\btheta)=n^{-1}\sum_{t=1}^n w_t^*\widetilde\epsilon_t(\btheta)\widetilde\epsilon_{t-h}(\btheta)\quad\mbox{ and }\quad\gamma_{h}^*(\btheta)=n^{-1}\sum_{t=1}^n w_t^*\epsilon_t(\btheta)\epsilon_{t-h}(\btheta).$$ 
Using \eqref{etcruelle}  we have
$$\sup_{\btheta \in \bTheta }\left| \widetilde\gamma_{h}^*(\btheta)-\gamma_{h}^*(\btheta)\right|\leq \frac{1}{n}\sum_{t=1}^n w_t^*K_t\varrho^t,\quad K_t=K\sup_{\btheta \in \bTheta }\left(|\epsilon_t(\btheta)|+|\epsilon_{t-h}(\btheta)|\right)$$
Using also \eqref{grace} we have
\begin{equation*}
\sqrt{n}\sup_{\btheta \in \bTheta }\left| \widetilde\gamma_{h}^*(\btheta)-\gamma_{h}^*(\btheta)\right|=o_{a.s.}^*(1)\quad \mbox{ as }
n\rightarrow \infty.
\end{equation*}
A Taylor expansion of $\gamma_{h}^*(\cdot)$ about $\widehat{\btheta}^*_n=\widehat{\btheta}_n$ then gives
\begin{align*}
\sqrt n\left( \widehat\gamma^*_{h}-\widehat\gamma_{h}\right)=&\sqrt{n}\left\{\gamma^*_{h}(\widehat{\btheta}^*_n)-\gamma_{h}(\widehat{\btheta}_n)\right\}+o_{a.s.}^*(1)
\\=&
\sqrt n \left\{\gamma^*_{h}(\widehat\btheta_n)-\gamma_{h}(\widehat\btheta_n)\right\}
+
\frac{\partial}{\partial {\btheta}^\top} \gamma^*_{h}(\btheta)\Big|_{{\btheta}={\btheta}^*}
\sqrt n \skakko{\widehat{\btheta}^*_n-\widehat\btheta_n}+o_{a.s.}^*(1),
\end{align*}
where  ${\btheta}^*$ is between $\widehat\btheta_n$ and $\widehat{\btheta}^*_n$. 
In view of \eqref{grace}, for all $\varepsilon>0$, there exists a neighborhood $V(\btheta_0)$ of $\btheta_0$ such that 
$$\lim_{n\to\infty}\frac{1}{n}\sum_{t=1}^n\sup_{\btheta\in V(\btheta_0)}\left|\epsilon_t(\btheta)\epsilon_{t-h}(\btheta)-\epsilon_t\epsilon_{t-h}\right|^2= E\sup_{\btheta\in V(\btheta_0)}\left|\epsilon_t(\btheta)\epsilon_{t-h}(\btheta)-\epsilon_t\epsilon_{t-h}\right|^2<\varepsilon.$$
In view of the consistency of $\widehat\btheta_n$, we thus have
$$\lim_{n\to\infty}\frac{1}{n}\sum_{t=1}^n\left|\epsilon_t(\widehat\btheta_n)\epsilon_{t-h}(\widehat\btheta_n)-\epsilon_t\epsilon_{t-h}\right|^2= 0,$$
and similarly
$$\lim_{n\to\infty}\frac{1}{n}\sum_{t=1}^n\bd_{t,n}^\top\bd_{t,n}=0,\quad  \bd_{t,n}:=\frac{\partial}{\partial {\btheta}}\left\{\epsilon_t(\btheta^*)\epsilon_{t-h}(\btheta^*)\right\}-\frac{\partial}{\partial {\btheta}}\left\{\epsilon_t(\btheta_0)\epsilon_{t-h}(\btheta_0)\right\}.$$
By Lemma \ref{lembrulant}, it follows
$$\sqrt n \left\{\gamma^*_{h}(\widehat\btheta_n)-\gamma_{h}(\widehat\btheta_n)\right\}=\frac{1}{\sqrt{n}}\sum_{t=1}^n( w_t^*-1)\epsilon_t\epsilon_{t-h}+o^*_P(1),\quad 
\frac{\partial}{\partial {\btheta}}\gamma^*_{h}(\btheta^*)
= \bd_h+o^*_P(1).$$
Now we conclude as in the proof of Theorem~\ref{thm:gof}, noting that 
$$\sqrt{n}\left(\widehat{\btheta}_n^*-\widehat{\btheta}_n\right)
= 
\bJ^{-1} 
\frac{1}{\sqrt{n}}\sum_{t=1}^n
\skakko{w_t^*-1}\frac{\partial}{\partial \btheta} {\ell}_t\skakko{\btheta_0}+o_P^*(1)$$
and, by Lemma \ref{lemfroid}, conditional on $\{Y_t\}$,
$$\frac{1}{\sqrt n}\sum_{t=1}^n (w_t^*-1) \left(\begin{array}{c}\epsilon_t\bepsilon_{t-1:t-k}\\\bJ^{-1}\frac{\partial}{\partial \btheta}\ell_t(\btheta_0)\end{array}\right)\underset{n\rightarrow \infty }{\overset{D}{\rightarrow }}\mathcal{N}%
\left\{ \bzero,\left(\begin{array}{cc}\bE&\bC\bJ^{-1}\\
\bJ^{-1}\bC^\top&\bSigma\end{array}\right)\right\}.$$

\end{document}